\documentclass[a4paper]{article}

\usepackage{comment}

\usepackage[all]{xy}\usepackage[latin1]{inputenc}        
\usepackage[dvips]{graphics,graphicx}
\usepackage{amsfonts,amssymb,amsmath,xcolor,mathrsfs, dsfont,amstext}
\usepackage{amsbsy, amsopn, amscd, amsxtra, amsthm,authblk, enumerate}
\usepackage{upref}
\usepackage{geometry}
\usepackage[displaymath]{lineno}
\usepackage{float}
\usepackage{bbm}

\usepackage[colorlinks,
            linkcolor=blue,
            anchorcolor=green,
            citecolor=blue
            ]{hyperref}

\numberwithin{equation}{section}

\def\e{\varepsilon}
\def\epsilon{\varepsilon}
\def\eps{\varepsilon}

\newcommand{\ol}{\overline}
\newcommand{\wt}{\widetilde}
\def\alb#1\ale{\begin{align*}#1\end{align*}}
\newcommand{\eqb}{\begin{equation}}
\newcommand{\eqe}{\end{equation}}

\newcommand{\bbC}{\mathbb{C}}
\newcommand{\bbD}{\mathbb{D}}
\newcommand{\bbE}{\mathbb{E}}
\newcommand{\bbH}{\mathbb{H}}

\newcommand{\bbR}{\mathbb{R}}
\newcommand{\bbP}{\mathbb{P}}

\newcommand{\cC}{\mathcal{C}}
\newcommand{\cT}{\mathcal{T}}
\newcommand{\cD}{\mathcal{D}}

\newcommand{\cL}{\mathcal{L}}

\newcommand{\QD}{\mathrm{QD}}
\newcommand{\QT}{\mathrm{QT}}
\newcommand{\LF}{\mathrm{LF}}
\newcommand{\SLE}{\mathrm{SLE}}
\newcommand{\CLE}{\mathrm{CLE}}
\newcommand{\Wd}{\mathrm{Weld}}
\newcommand{\Md}{{\mathcal{M}}^\mathrm{disk}}
\newcommand{\Mfd}{{\mathcal{M}}^\mathrm{f.d.}}

\newcommand{\LL}{L}
\newcommand{\RR}{R}
\newcommand{\Loop}{\mathcal{L}^o}

\newtheorem{theorem}{Theorem}[section]
\newtheorem{lemma}[theorem]{Lemma}
\newtheorem{proposition}[theorem]{Proposition}
\newtheorem*{proposition*}{Proposition}

\newtheorem*{corollary*}{Corollary}
\newtheorem{definition}[theorem]{Definition}
\newtheorem*{definitions*}{Definitions}

\newtheorem*{example*}{\bf Example}
\theoremstyle{remark}
\newtheorem{remark}[theorem]{Remark}

\numberwithin{equation}{section}

\title{Boundary touching probability and nested-path exponent for non-simple CLE}
\author{Morris Ang\thanks{Columbia University}  \qquad Xin Sun\thanks{Beijing International Center for Mathematical Research, Peking University.}\qquad Pu Yu\thanks{Massachusetts Institute of Technology} \qquad Zijie Zhuang\thanks{University of Pennsylvania}}
\date{}

\begin{document}

\maketitle

\begin{abstract}
The conformal loop ensemble (CLE) has two phases: for $\kappa \in (8/3, 4]$, the loops are simple and do not touch each other or the boundary; for $\kappa \in (4,8)$, the loops are non-simple and may touch each other and the boundary. For $\kappa\in(4,8)$, we derive the probability that the loop surrounding a given point touches the domain boundary.  We also obtain the law of the conformal radius of this loop seen from the given point conditioned on the loop touching the boundary or not, refining a result of Schramm-Sheffield-Wilson (2009). As an application, we exactly evaluate the CLE counterpart of the nested-path exponent for the Fortuin-Kasteleyn (FK) random cluster model recently introduced by Song-Tan-Zhang-Jacobsen-Nienhuis-Deng (2022). This exponent describes the asymptotic behavior of the number of nested open paths in the open cluster containing the origin  when  the cluster is large.
For Bernoulli percolation, which corresponds to $\kappa=6$, the exponent was derived recently in Song-Jacobsen-Nienhuis-Sportiello-Deng (2023) by a color switching argument. For $\kappa\neq 6$, and in particular for the FK-Ising case, our formula appears to be new. Our derivation begins with Sheffield's construction of CLE from which the quantities of interest can be expressed by radial SLE. We solve the radial SLE problem using the coupling between SLE and Liouville quantum gravity, along with the exact solvability of Liouville conformal field theory.
\end{abstract}

\section{Introduction}

The conformal loop ensemble (CLE$_\kappa$) is a natural random collection of non-crossing planar loops initially introduced in~\cite{SheffieldCLE, Sheffield-Werner-CLE} that possesses the conformal invariance property. It is conjectured that CLE$_\kappa$ describes the scaling limit of many statistical mechanics models including the Fortuin-Kasteleyn percolation and the $O(n)$ loop model. There is an extensive  literature on CLE. For instance, its relation with discrete models is explored in~\cite{Smirnov-01, camia-newman-sle6, Smirnov-10, Kemppainen-Smirnov-19, benoist-hongler-cle3, lupu-loop-soup-cle}, and its continuum properties are studied in~\cite{SSW09, MSWCLEgasket, werner-wu-explorations, werner-sphere-cle, ALS-CLE4, gmq-cle-inversion, CLE-percolations, MSWsimpleCLE, MSW-non-simple-2021, AS21}.

In this paper, we focus on non-simple CLE, namely CLE$_\kappa$ with $\kappa \in (4,8)$, in which case the loops are non-simple and may touch each other and the boundary. Our first main result is the exact evaluation of  the probability that the loop surrounding a given point touches the domain boundary; see Theorem~\ref{thm:touch}. Moreover, we obtain the law of the conformal radius of this loop seen from the given point conditioned on the loop touching the boundary or not. This refines the main result from~\cite{SSW09}. We also obtain the law  of the conformal radius of another naturally defined domain; see Theorem~\ref{thm:CR-widetilde-D}. This result yields the exact value of the CLE counterpart of the so-called nested-path exponent for the Fortuin-Kasteleyn (FK) random cluster model, which was introduced by~\cite{STZJ22} and describes the asymptotic behavior of the number of nested open paths in the percolation cluster containing the origin  when  the cluster is large.
For Bernoulli percolation, which corresponds to $\kappa=6$, the exponent was derived recently in~\cite{SJNS23}. For $\kappa\neq 6$, including the FK-Ising case (i.e. $\kappa=\frac{16}{3}$), our formula appears to be new.

Our derivation begins with the continuum tree construction of CLE$_\kappa$ as described in~\cite{SheffieldCLE} from which the aforementioned quantities can be expressed through the radial SLE exploration. More precisely, they are encoded by the conformal radius of the explored region at certain stopping times of a radial SLE curve. Then we solve the radial SLE problem using the coupling between SLE and Liouville quantum gravity (LQG), along with the exact solvability of Liouville conformal field theory (LCFT). This approach for extracting quantitative information about SLE curves was developed in prior works by the first and second named authors and their collaborators~\cite{AHS21, ARS21, AS21}.

Our paper is organized as follows. In Section~\ref{subsec:touching} and~\ref{subsec:nested-path} we state our main results.  In Sections~\ref{subsec:overview} and \ref{subsec:outlook}, we overview our proof strategy and discuss related works. In Section~\ref{sec:pre} we provide preliminaries on CLE and LQG. In Sections~\ref{sec:weld} and~\ref{sec:proof} we prove results on the conformal radii as outlined in Section~\ref{subsec:overview}. In Section~\ref{sec:thm1.4} we derive the nested path exponent.

\subsection{Boundary touching probability for non-simple CLE}
\label{subsec:touching}

For $\kappa \in (4,8)$, the CLE$_\kappa$ loops may touch the boundary, and a natural quantity to study is the probability that the CLE$_\kappa$ loop surrounding a given point touches the boundary. This is equivalent to the expected fraction of area surrounded by the boundary touching CLE$_\kappa$ loops. For concreteness, we let $\bbD$ be the unit disk and $\Gamma$ be a non-nested ${\rm CLE}_{\kappa}$ on $\bbD$. Let $\Loop$ be the loop in $\Gamma$ that surrounds the origin. Our first main result is:

\begin{theorem}
\label{thm:touch}
For $\kappa \in (4,8)$, we have
\eqb\label{eq:thm:touch}
\mathbb{P}[ \Loop \cap \partial \bbD \not = \emptyset ] =  1 - \frac{\sin(\pi(\frac{\kappa}{4} + \frac{8}{\kappa}))}{\sin(\pi \frac{\kappa-4}{4})}.
\eqe
\end{theorem}

By the conformal invariance of CLE, the formula~\eqref{eq:thm:touch} holds if $\bbD$ is replaced by any simply-connected domain $D$ with boundary and $\Loop$ is defined to be the loop surrounding any given interior point in $D$. Now we discuss the implications of Theorem~\ref{thm:touch} for the Fortuin-Kasteleyn (FK) percolation, a statistical mechanics model introduced in \cite{FK72}. Consider critical FK percolation with cluster-weight $q \in (0,4]$ on the discretized box $\mathbb{B}_N := \frac{1}{N} \mathbb{Z}^2 \cap [-1,1]^2$ equipped with the wired boundary condition. It is conjectured that the interfaces between open and dual open clusters converge, under a natural topology, to ${\rm CLE}_\kappa$ with $\kappa = \frac{4 \pi}{\pi - \arccos(\sqrt{q}/2)} \in [4,8)$. This conjecture has been confirmed in  the FK Ising case (when $q = 2$ and $\kappa = 16/3$) in~\cite{Smirnov-10, Kemppainen-Smirnov-19}; see also the recent work~\cite{DC-invariance} on the rotational invariance of sub-sequential limits for $q \in [1,4]$. For the Bernoulli percolation case (when $q = 1$ and $\kappa = 6$), the site percolation variant on the triangular lattice was proved in~\cite{Smirnov-01, camia-newman-sle6}.
Assuming this conjecture, Theorem~\ref{thm:touch} also applies to critical FK percolation and describes the limiting probability of the outermost open cluster that surrounds the origin touching the boundary as $N \rightarrow \infty$.\footnote{To transition from the convergence of interfaces to this result, we also need to show that if the outermost interface surrounding the origin is close to the boundary, then it is very likely to touch the boundary. When $q \in [1,4)$, we can deduce this using the fact that the half-plane three-arm exponent is larger than 1, see~\cite{DCMT21}.}

We observe that $\mathbb{P}[\Loop \cap \partial \bbD \not = \emptyset] = \frac{1}{2}$ at $\kappa =6$, tends to $0$ as $\kappa$ approaches $4$, and tends to $\frac{1}{2}$ as $\kappa$ approaches $8$. The behavior as $\kappa$ approaches $4$ can be seen from the continuity of the law of CLE$_\kappa$ in $\kappa$, and the absence of boundary-touching loops in CLE$_4$. For critical Bernoulli percolation, by duality and the independence of boundary conditions, we see that the outermost open cluster has asymptotically equal probabilities of touching or not touching the boundary. This is consistent with $\mathbb{P}[\Loop \cap \partial \bbD \not = \emptyset] = \frac{1}{2}$ at $\kappa =6$. To see why  $\mathbb{P}[\Loop \cap \partial \bbD \not = \emptyset]$ tends to $\frac{1}{2}$ when $\kappa$ approaches $8$, consider a uniform spanning tree on $\mathbb{B}_N$ with wired boundary condition. The $\kappa\to 8$ limit of $\CLE_\kappa$ can be viewed as a single space-filling loop describing the scaling limit of the interface separating this uniform spanning tree and its dual tree. Furthermore, $\{\Loop \cap \partial \bbD \not = \emptyset\}$ corresponds to the event that the origin is surrounded by this loop which covers asymptotically half of the domain. Therefore, $\lim_{\kappa\to 8}\mathbb{P}[\Loop \cap \partial \bbD \not = \emptyset]$ should be  $\frac{1}{2}$. It would be interesting to find a discrete explanation for the value of $\mathbb{P}[\Loop \cap \partial \bbD \not = \emptyset]$ for other values of $\kappa$. In~\cite{MW18}, a similar quantity about CLE is calculated and the authors gave such an explanation. We also observe that the function $\kappa \mapsto \mathbb{P}[\Loop \cap \partial \bbD \not = \emptyset]$ is increasing in $(4,\kappa_0)$ and decreasing in $(\kappa_0,8)$, where $\kappa_0 \approx 6.95061$ is the unique solution to $\tan(\pi(\frac{x}{4} + \frac{8}{x})) = \frac{x^2-32}{x^2} \tan(\frac{\pi x}{4})$ within $(4,8)$.

We prove Theorem~\ref{thm:touch} by proving the stronger Theorem~\ref{thm:CR} below. For a simply connected domain $D \subset \mathbb{C}$ and $z \in D$, let $f: \bbD \rightarrow D $ be a conformal map with $f(0) = z$. The conformal radius of $D$ seen from $z$ is defined as ${\rm CR}(z,D) := |f'(z)|$. Let $D_{\Loop}$ be the connected component of $\bbD \backslash \Loop$ that contains the origin; see Figure~\ref{fig:D} (left). In \cite[Theorem 1]{SSW09}, the law of ${\rm CR}(0, D_{\Loop})$ is obtained: for $\lambda \leq \frac{3\kappa}{32}+\frac{2}{\kappa}-1$, $\mathbb{E} [ {\rm CR}(0, D_{\Loop})^{\lambda}] = \infty$ and for $\lambda > \frac{3\kappa}{32}+\frac{2}{\kappa}-1$, we have
\begin{equation}
    \label{eq:SSW09}
    \mathbb{E} [ {\rm CR}(0, D_{\Loop})^\lambda] = \frac{\cos(\pi \frac{\kappa-4}{\kappa}) }{ \cos(\frac{\pi}{\kappa} \sqrt{(\kappa-4)^2 - 8 \kappa \lambda})}\,.
\end{equation}
Theorem~\ref{thm:CR} gives the moments of ${\rm CR}(0, D_{\Loop})$ restricted to the event $\{ \Loop \cap \partial \bbD \not = \emptyset \}$ or its complement.
\begin{theorem}
\label{thm:CR}
For $4<\kappa<8$, let $T= \{ \Loop \cap \partial \bbD \not = \emptyset \}$. We have:
\begin{itemize}
\item[(1).] For $\lambda \leq \frac{\kappa}{8}-1$, $\mathbb{E} [ {\rm CR}(0, D_{\Loop})^{\lambda} \mathbbm{1}_T ] =\infty$, and for $\lambda > \frac{\kappa}{8}-1$,

\eqb\label{eq:thm:CR-1}
\mathbb{E} [ {\rm CR}(0, D_{\Loop})^{\lambda}  \mathbbm{1}_T ] =  \frac{2\cos(\pi \frac{\kappa-4}{\kappa})\sin(\pi\frac{\kappa-4}{4\kappa} \sqrt{(\kappa-4)^2 - 8 \kappa \lambda})}{\sin(\frac{\pi}{4} \sqrt{(\kappa-4)^2 - 8 \kappa \lambda})}.
\eqe

\item[(2).] For $\lambda \leq \frac{3\kappa}{32}+\frac{2}{\kappa}-1$, $\mathbb{E} [ {\rm CR}(0, D_{\Loop})^{\lambda} \mathbbm{1}_{T^c} ] =\infty$, and for $\lambda > \frac{3\kappa}{32}+\frac{2}{\kappa}-1$,

\eqb\label{eq:thm:CR-2} \mathbb{E} [ {\rm CR}(0, D_{\Loop})^{\lambda} \mathbbm{1}_{T^c} ] = \frac{\cos(\pi \frac{\kappa-4}{\kappa})\sin(\pi\frac{8 - \kappa}{4\kappa} \sqrt{(\kappa-4)^2 - 8 \kappa \lambda})}{\cos(\frac{\pi}{\kappa} \sqrt{(\kappa-4)^2 - 8 \kappa \lambda}) \sin(\frac{\pi}{4} \sqrt{(\kappa-4)^2 - 8 \kappa \lambda})}.
\eqe
\end{itemize}
\end{theorem}

To prove Theorems~\ref{thm:touch} and~\ref{thm:CR}, we first use the coupling between SLE and LQG and the integrability of LCFT to compute the ratio $\frac{\mathbb{E} [ {\rm CR}(0, D_{\Loop})^{\lambda} \mathbbm{1}_T ]}{\mathbb{E} [ {\rm CR}(0, D_{\Loop})^{\lambda} \mathbbm{1}_{T^c} ] }$; see Section~\ref{subsec:computation-proof}. Then combined with~\eqref{eq:SSW09} we get both theorems. See Section~\ref{subsec:overview} for an overview of our derivation of this ratio.

\begin{figure}[H]
\centering
\includegraphics[scale=0.8]{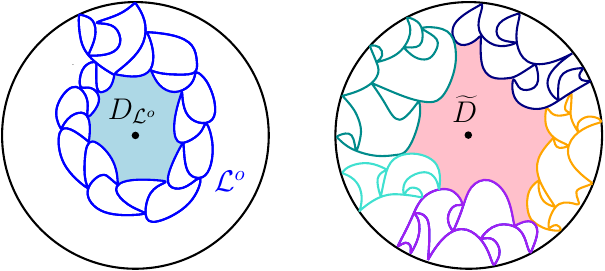}
\caption{Illustration of the domains considered in Theorems~\ref{thm:CR} and \ref{thm:CR-widetilde-D}. \textbf{Left:} The domain $D_{\Loop}$ 
on the event $\{ \Loop \cap \partial \bbD \not = \emptyset \}$. \textbf{Right:} The domain $\widetilde D$  on the event $\{\Loop \cap \partial \bbD = \emptyset\}$. The colored loops represent the CLE$_\kappa$ loops that touch the boundary, and $\Loop$ is contained within the pink domain in this case. The boundary of $\widetilde D$ is the first open circuit in the defitition of the CLE nested-path exponent $X_{\rm NP}$. }\label{fig:D}
\end{figure}

\subsection{The nested-path exponent}
\label{subsec:nested-path}
The coupling between SLE and LQG also allows us to prove the following Theorem~\ref{thm:CR-widetilde-D}, which is of a similar form as Theorem~\ref{thm:CR}.  In the setting of Theorem~\ref{thm:CR}, on the event $T^c= \{ \Loop \cap \partial \bbD  = \emptyset \} $,  let $\widetilde D$ be  the connected component containing the origin after all the boundary-touching loops in $\Gamma$ are removed from $\bbD$; see Figure~\ref{fig:D} (right). 
Theorem~\ref{thm:CR-widetilde-D} gives the moment of ${\rm CR}(0, \widetilde D)$.
\begin{theorem}
    \label{thm:CR-widetilde-D}
    Fix $\kappa \in (4,8)$. For $\lambda \leq  \frac{\kappa}{8}-1$, $\mathbb{E} [ {\rm CR}(0, \widetilde D)^{\lambda} \mathbbm{1}_{T^c}] = \infty$, and for $\lambda > \frac{\kappa}{8}-1$,
\eqb\label{eq:thm:CR-widetilde-D}\mathbb{E} [ {\rm CR}(0, \widetilde D)^{\lambda} \mathbbm{1}_{T^c} ] = \frac{\sin(\pi\frac{8 - \kappa}{4\kappa} \sqrt{(\kappa-4)^2 - 8 \kappa \lambda})}{\sin(\frac{\pi}{4} \sqrt{(\kappa-4)^2 - 8 \kappa \lambda})}.
\eqe
\end{theorem}
Theorem~\ref{thm:CR-widetilde-D} allows us  to derive the CLE counterpart of the nested-path exponent introduced in \cite{STZJ22}, which we now recall. Consider critical FK percolation on $\mathbb{B}_N$ with the wired boundary condition. We define open circuit to be a self-avoiding polygon consisting of open edges. We also view a single vertex as an open circuit of length zero.
Let $\mathcal{R}_N$ be the event that there exists an open path connecting the origin to the boundary. On this event, let the boundary of $\mathbb{B}_N$ be the zeroth open circuit by convention. Inductively, given the $k$-th open circuit, if it passes through the origin, we stop and set $\ell_N=k$. Otherwise, among all open circuits that surround the origin and do not use edges in the first  $k$ open circuits, there exists a unique outermost one, which we call the $(k+1)$-th open circuit. This defines a sequence of nested open circuits with a total count of $\ell_N$. For each $a > 0$, the nested-path exponent $X_{\rm NP}(a)$ in \cite{STZJ22} is specified by:
\begin{equation}
\label{eq:def-nested-path}
\mathbb{E}[a^{\ell_N} \mathbbm{1}_{\mathcal{R}_N}] = N^{-X_{\rm NP}(a) + o(1)} \quad \mbox{as }N \rightarrow \infty\,.
\end{equation}
A priori, we do not know whether this exponent exists. However, under the assumption that critical FK percolation converges to CLE (which is known to hold for the FK-Ising case), this exponent can be derived from its continuum counterpart in the range of $q \in [1,4)$, as explained in Remark~\ref{remark:cont-to-discrete}.

\begin{figure}[H]
\centering
\includegraphics[scale=0.4]{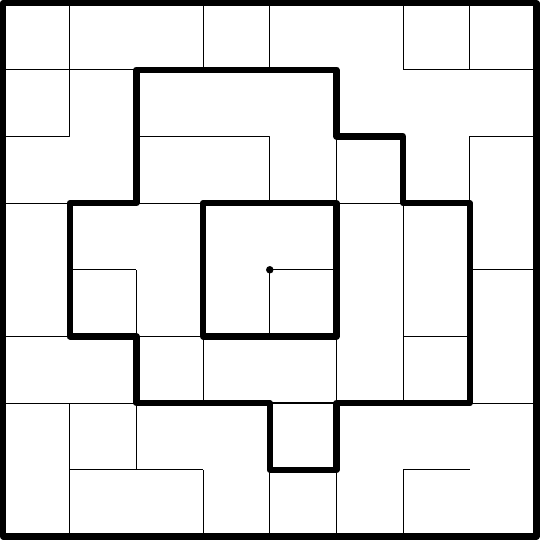}
\caption{Illustration of the nested open circuits for critical FK percolation with the wired boundary condition on a discretized box. There is an open path from the origin to the boundary, and the three bold circuits together with the origin are the four nested open circuits explored from outside in.}\label{fig:nested}
\end{figure}

Arm exponents in critical FK percolation capture important geometric information of critical percolation clusters. Previously, people have studied the watermelon exponent which describes the probability that there exists a given number, say $2k$, of disjoint percolation interfaces from the origin to distance $N$ as $N$ tends to infinity. Another family of exponents, the nested-loop exponents, is defined similarly to~\eqref{eq:def-nested-path} but replaces $\ell_N$ with the number of disjoint percolation interfaces surrounding the origin (see~\eqref{eq:def-nested-loop} below). These two families of exponents appear in the spectrum of the physical conformal field theory (CFT) describing FK percolation (see e.g.~\cite{NRJ-2024}). Their values were first calculated using physical approaches~\cite{SaleurDuplantier-1987, dN83, MN04}, and the mathematical derivations can be found in~\cite{SW01, Wu18, SSW09}. A natural question is what these exponents will be if we count the number of percolation paths instead of percolation interfaces. In this case, the watermelon exponent becomes the monochromatic arm exponent, and the nested-loop exponent leads to the nested-path exponent. 

Now we define the CLE counterpart of $X_{\rm NP}(a)$ by counting the number of nested ``open circuits'' that surround a small disk with respect to CLE. In the setting of Theorem~\ref{thm:CR-widetilde-D}, recall the loop $\Loop$ and the event $T = \{{\Loop} \cap \partial \bbD \neq \emptyset \}$. For $\epsilon>0$, let $\mathcal{R}_\epsilon := \{ \epsilon \bbD \not \subset D_{\Loop} \}$, which is the continuum analog of the event that the origin is connected to boundary by an open path.  We view $\partial \bbD$ as the zeroth open circuit. 
On the event $\mathcal{R}_\epsilon$, if $T$ occurs, we set $\ell_\eps=0$. Otherwise, ${\Loop} \cap \partial \bbD =  \emptyset$, and we let $\partial \widetilde D$ be the first open circuit. By the domain Markov property of CLE, inside $\widetilde D$ we have a CLE. Inductively, given the $k$-th open circuit, which is a simple loop surrounding the origin, if it intersects either $\epsilon\bbD$ or $\Loop$, we stop and let $\ell_\eps=k$. Otherwise we iterate the procedure to find the $(k+1)$-th open circuit surrounding the origin. For each $a > 0$, the CLE nested-path exponent $\widetilde{X}_{\rm NP}(a)$ is defined similarly to~\eqref{eq:def-nested-path} by:
\begin{equation}
\label{eq:def-nested-path-continuum}
    \mathbb{E}[a^{\ell_\epsilon} \mathbbm{1}_{\mathcal{R}_\epsilon}] = \epsilon^{\widetilde{X}_{\rm NP}(a) + o(1)} \quad \mbox{as }\epsilon \rightarrow 0\,.
\end{equation}
The following theorem gives the existence and exact value of $\widetilde{X}_{\rm NP}(a)$.

\begin{theorem}
\label{thm:nested-path-value}
    Fix $\kappa \in (4, 8)$. For any $a > 0$, the CLE nested-path exponent $\widetilde{X}_{\rm NP}(a)$ exists. Moreover, it is the unique solution smaller than $1-\frac{\kappa}{8}$ to the equation:
    \begin{equation}
    \label{eq:sol-nested-path}
    \sin\Big(\frac{\pi}{4} \sqrt{(\kappa-4)^2 + 8 \kappa x } \Big) = a \cdot \sin\Big(\frac{\pi (8 - \kappa)}{4 \kappa} \sqrt{(\kappa-4)^2 + 8 \kappa x }  \Big)\,. 
    \end{equation}
\end{theorem}
For $a>0$, let $\mathrm{Root}(a) $ be the unique solution smaller than $1-\frac{\kappa}{8}$ to the equation~\eqref{eq:sol-nested-path}.
By Theorem~\ref{thm:CR-widetilde-D}, we have $\mathbb{E}[{\rm CR}(0, \widetilde D)^{-\mathrm{Root}(a) } \mathbbm{1}_{T^c}] = \frac{1}{a}$. 
In Section~\ref{sec:thm1.4}, we will use this observation and a large deviation argument in a similar spirit to \cite{MWW16} to prove that $\widetilde{X}_{\rm NP}(a)$ exists and equals $\mathrm{Root}(a)$.

\begin{remark}
\label{remark:cont-to-discrete}
    For $\epsilon>0$ and integer $N \geq 1$, let $\bbD_{\epsilon N, N} = \frac{1}{N} \mathbb{Z}^d \cap \epsilon \bbD$. Then the event $\mathcal{R}_\epsilon$ can be seen as the $N \rightarrow \infty$ limit of the event that $\bbD_{\epsilon N, N}$ is connected to the boundary of $\mathbb{D}_N = \frac{1}{N} \mathbb{Z}^d \cap \bbD$ by an open path in critical $q$-FK percolation with the wired boundary condition, and $\ell_\epsilon$ is the $N \rightarrow \infty$ limit of the maximal count of nested open circuits surrounding $\bbD_{\epsilon N, N}$ in $\mathbb{D}_N$. Assuming the convergence of FK percolation to CLE, Equation~\eqref{eq:def-nested-path-continuum} implies that $\lim_{N \rightarrow \infty} \mathbb{E}[a^{\ell_{\epsilon N, N}} \mathbbm{1}_{\mathcal{R}_{\epsilon N, N}}] = \epsilon^{\widetilde{X}_{\rm NP}(a) + o(1)}$. 
    For $q \in [1,4)$ where quasi-multiplicative inequalities are available from~\cite{DCMT21} (their Proposition 6.3 is stated for the arm events, but similar inequalities are expected to hold for the number of nested paths), we expect that Equation~\eqref{eq:def-nested-path} follows. This reasoning is used in \cite{SW01}, and later in e.g.~\cite{Wu18, KL-Fuzzy}. For brevity, we will not pursue it here. \end{remark}

    Equation~\eqref{eq:sol-nested-path} greatly simplifies when $q=1$ and $q=2$, which yields:
    \begin{align}
        &\widetilde{X}_{\rm NP}(a)= \frac{3}{4 \pi^2}\arccos (\frac{a-1}{2})^2 - \frac{1}{12}, \quad && q=1,\kappa=6;\label{eq:per}\\
        &\widetilde{X}_{\rm NP}(a)= \frac{3}{2 \pi^2}\arccos (\frac{a}{2})^2 - \frac{1}{24}, \quad && q=2,\kappa=16/3. \label{eq:Ising}
    \end{align}
Our~\eqref{eq:per} agrees with the formula for $X_{\rm NP}(a)$ in the Bernoulli percolation case derived by \cite{SJNS23}. Our~\eqref{eq:Ising} agrees with the unpublished numerical finding by Youjin Deng et.al.\ for the FK Ising case\footnote{Private communication with Youjin Deng.}.

The argument in~\cite{SJNS23} for Bernoulli percolation is based on a link to the so-called nested-loop exponent. Let $t_N$ be the number of interfaces that surround the origin in critical FK percolation on $\mathbb{B}_N$. For $a>0$, the nested-loop exponent $X_{\rm NL}(a)$  is defined by:
$$
\mathbb{E}[a^{t_N}] = N^{-X_{\rm NL}(a) + o(1)} \quad \mbox{as } N \rightarrow \infty.
$$  For critical Bernoulli percolation, an exact formula for $X_{\rm NL}(a)$ was given in \cite{dN83, MN04}. An elementary color switching argument in \cite{SJNS23}, which is specific to the critical site Bernoulli percolation on the triangular lattice or the bond one on the square lattice, yields that  $X_{\rm NP}(a + 1) = X_{\rm NL}(a)$ in this case.

For $\kappa\in (4,8)$, similar to $\widetilde{X}_{\rm NP}(a)$ in~\eqref{eq:def-nested-path-continuum}, we can define the CLE nested-loop exponent $\widetilde{X}_{\rm NL}(a)$ for $a>0$ by $  \mathbb{E}[a^{t_\epsilon}] = \epsilon^{\widetilde{X}_{\rm NL}(a) + o(1)}$, where  $t_\epsilon$ counts the number of nested loops in a CLE on $\bbD$ lying inside $\bbD\setminus \epsilon\bbD$.  The proof of Theorem~\ref{thm:nested-path-value} then gives that $\widetilde{X}_{\rm NL}(a)$ exists and satisfies $\mathbb{E}[{\rm CR}(0, D_{\Loop})^{-\widetilde{X}_{\rm NL}(a)}] = \frac{1}{a}$. This can essentially be extracted from \cite[Lemma 3.2]{MWW16}, which is based on \cite{SSW09}. We leave the detail to the reader.
By Equation~\eqref{eq:SSW09} (\cite[Theorem 1]{SSW09}), we conclude that $\widetilde{X}_{\rm NL}(a)$ is the unique solution smaller than $1-\frac{2}{\kappa} - \frac{3 \kappa}{32}$ to the equation:
$$
    \cos\Big(\frac{\pi}{\kappa} \sqrt{(\kappa-4)^2 + 8 \kappa x } \Big) = a \cdot \cos\Big(\pi \frac{\kappa-4}{\kappa}\Big)\,. 
$$

\subsection{Overview of the proof based on Liouville quantum gravity}
\label{subsec:overview}

 Originated from string theory, Liouville quantum gravity (LQG) is introduced by Polyakov in his seminal work~\cite{polyakov1981quantum}. LQG has a parameter $\gamma\in(0,2]$, and it has close relation with the scaling limits of random planar maps, see e.g.~\cite{LeGall13, BM17,HS19,gwynne2021convergence}. As observed by Sheffield~\cite{She16a}, one key aspect of random planar geometry is the \emph{conformal welding} of random surfaces, where the interface under the conformal welding of two LQG surfaces is an SLE curve. Similar type of results were also proved in~\cite{DMS14,AHS20,ASY22,AHSY23, AG23a}.
 
Liouville conformal field theory (LCFT) is a 2D quantum field theory rigorously developed in~\cite{DKRV16} and subsequent works. LCFT is closely related to LQG, as it has been demonstrated that many natural LQG surfaces can be described by LCFT~\cite{AHS17,cercle2021unit,AHS21,ASY22}. In the framework of Belavin, Polyakov, and Zamolodchikov's conformal field theory~\cite{BPZ84}, extensively explored in physics literature~\cite{DO94, ZZ96, PT02} and mathematically in~\cite{KRV20,RZ22,ARS21, GKRV20_bootstrap, GKRV21_Segal, ARSZ23}, LCFT enjoys rich and deep exact solvability. Alongside the conformal welding of LQG surfaces mentioned earlier, in~\cite{AHS21,ARS21,AS21}, the first and second named authors, along with Holden and Remy, derived several exact formulae regarding SLE and CLE.

Our proof of Theorems~\ref{thm:touch}-\ref{thm:CR-widetilde-D} is another example of  exact formula of SLE/CLE based on conformal welding of LQG surfaces and LCFT. In earlier works of~\cite{MSWsimpleCLE, MSW-non-simple-2021}, the coupling between CLE and LQG was crucially used to derive properties of CLE. There the authors relied on the advanced exploration mechanisms for CLE percolations from~\cite{CLE-percolations}. In contrast, we work directly with the classical construction of CLE in~\cite{SheffieldCLE} in terms of the continuum exploration tree. Based on this construction, the boundary touching event along with the quantities in these theorems can be expressed in terms of radial $\SLE_\kappa(\kappa-6)$; see Section~\ref{subsec:pre-CLE} for more details. In Section~\ref{sec:weld}, we derive Theorem~\ref{thm:welding}, a novel result on conformal welding of $\gamma$-LQG surfaces with radial $\SLE_\kappa(\kappa-6)$ being the interface where $\gamma = \frac{4}{\sqrt\kappa}$. This allows us to express~\eqref{eq:thm:touch}-\eqref{eq:thm:CR-widetilde-D} in terms of boundary lengths of LQG surfaces. 
The key LQG surfaces in Theorem~\ref{thm:welding} is what we call a generalized quantum triangle; see Definition~\ref{def:line-segment}.  It extends the notation of generalized quantum surfaces considered in~\cite{DMS14,MSW-non-simple-2021,AHSY23} to quantum triangles introduced in~\cite{ASY22} by three of us. A priori, we  need the three-point structure constant  for boundary LCFT from~\cite{RZ22} to handle quantum triangles, which is highly involved. We circumvent this difficulty in Section~\ref{sec:proof} via an auxiliary conformal welding result.

The proof of Theorem~\ref{thm:welding} has its own interest as well. In the mating-of-trees theory established by~\cite{DMS14}, one can identify an independent coupling between space-filling SLE and LQG with a pair of correlated Brownian motions. Several variants are also studied in~\cite{MS19,AG19,AY23}. We start with the Brownian excursion description of the  LQG disk $\cD$  decorated with space-filling SLE loop $\eta$ in~\cite{AG19}. Then we add an interior marked point $z$ on $\cD$ and look at the two parts $(\cD_1,\eta_1)$ and $(\cD_2,\eta_2)$ of $(\cD,\eta)$ before and  after $\eta$ hits $z$. We identify the law of $(\cD_1,\eta_1)$ and $(\cD_2,\eta_2)$ via the corresponding Brownian excursions, which further gives the conformal welding result  in Proposition~\ref{prop-mot-QD11}. Since the ``spine'' of $\eta$ stopped when hitting $z$ is the radial $\SLE_\kappa(\kappa-6)$ targeted at $z$ (see Proposition~\ref{prop:space-filling-loop-radial}), a re-arrangement of $(\cD_1,\eta_1)$ and $(\cD_2,\eta_2)$ gives the desired Theorem~\ref{thm:welding}.
We expect that Theorem~\ref{thm:welding} will be useful for extending exact results for simple CLE proved in~\cite{AS21} to the non-simple case.

\subsection{Outlook and perspectives}
\label{subsec:outlook}

In this section, we discuss related works and future directions.

\begin{itemize}
    \item With Remy, the first and second named authors have derived the annulus partition function of the dilute $O(n)$ loop model, as predicted by physicists~\cite{Saleur-Bauer-1989, Car06}, in~\cite{ARS2022moduli}. This approach can be extended to the dense $O(n)$ case by using the conformal welding of non-simple SLE. In a forthcoming work by the second and fourth named authors with Xu, we will apply the approach in~\cite{ARS2022moduli} to obtain the annulus crossing probabilities for critical percolation as predicted by Cardy~\cite{Car02, Car06}, where Theorem~\ref{thm:CR} will be a crucial input.
    
    \item In~\cite{CLE-percolations}, a variant of CLE known as boundary conformal loop ensembles (BCLE) was introduced. BCLE$_\kappa(\rho)$, involving an additional parameter $\rho$, can be expressed in terms of an SLE variant called SLE$_\kappa(\rho; \kappa - 6 - \rho)$ and describes the conjectural scaling limit of the fuzzy Potts model, a generalization of the $q$-Potts model; see~\cite{MSW-non-simple-2021, KL-Fuzzy}. In a future work, we hope to extend the results in this paper to SLE$_\kappa(\rho; \kappa - 6 - \rho)$ and derive the probability that a given point is surrounded by various loops in BCLE as well as the corresponding conformal radii. These results can be used to give the one-arm exponent for the fuzzy Potts model which is not known yet; see \cite{KL-Fuzzy} for the derivation of all the other arm exponents.
    
    \item For the case of percolation, i.e.\ $\kappa = 6$, predictions for the nested-loop exponent have been given in~\cite{dN83, MN04} based on conformal field theory (CFT) considerations and subsequently applied to the nested-path exponent in~\cite{SJNS23}. A CFT derivation for the nested-loop and path exponents for other values of $\kappa$ would be highly desirable. We also observe that the nested-path exponent has a similar look to the backbone exponent recently derived in \cite{NQSZ}, which is also obtained using the SLE/LQG coupling and the integrability of LCFT. It would be interesting to find an explanation about this phenomenon.
\end{itemize}

\medskip
\noindent\textbf{Acknowledgements.} 
We thank Ewain Gwynne for telling us the touching probability question, which he learned from Jason Miller. We thank Youjin Deng for bringing to our attention the question of deriving the nested-path exponent and sharing an earlier version of~\cite{SJNS23} and their unpublished numerical work for the FK Ising case. We thank Baojun Wu for earlier discussions on the nested-path exponent. We also thank the anonymous referees for their careful reading and many helpful comments. M.A.\ was supported by the Simons Foundation as a Junior Fellow at the Simons Society of Fellows.
X.S.\ was partially supported by the NSF Career award 2046514, a start-up grant from the University of Pennsylvania, and a fellowship from the Institute for Advanced Study (IAS) at Princeton. 
P.Y.\  was partially supported by NSF grant DMS-1712862.  Z.Z.\ was partially supported by NSF grant DMS-1953848. 

\section{Preliminaries}\label{sec:pre}
In this paper we work with non-probability measures and extend the terminology of ordinary probability to this setting. For a finite or $\sigma$-finite  measure space $(\Omega, \mathcal{F}, M)$, we say $X$ is a random variable if $X$ is an $\mathcal{F}$-measurable function with its \textit{law} defined via the push-forward measure $M_X=X_*M$. In this case, we say $X$ is \textit{sampled} from $M_X$ and write $M_X[f]$ for $\int f(x)M_X(dx)$. \textit{Weighting} the law of $X$ by $f(X)$ corresponds to working with the measure $d\tilde{M}_X$ with Radon-Nikodym derivative $\frac{d\tilde{M}_X}{dM_X} = f$, and \textit{conditioning} on some event $E\in\mathcal{F}$ (with $0<M[E]<\infty$) refers to the probability measure $\frac{M[E\cap \cdot]}{M[E]} $  over the space $(E, \mathcal{F}_E)$ with $\mathcal{F}_E = \{A\cap E: A\in\mathcal{F}\}$. If $M$ is finite, we write $|M| = M(\Omega)$ and $M^\# = \frac{M}{|M|}$ for its normalization. Throughout this section, we also fix the notation $|z|_+:=\max\{|z|,1\}$ for $z\in\bbC$.

\subsection{${\rm CLE}_\kappa$ and radial ${\rm SLE}_\kappa(\kappa-6)$}\label{subsec:pre-CLE}

We start with the chordal \emph{Schramm Loewner evolution (SLE)} process on the upper half plane $\bbH$. Let $(B_t)_{t\ge0}$ be the standard Brownian motion. For $\kappa > 0$, the $\SLE_\kappa$ is the probability measure on non-self-crossing curves $\eta$ in $\overline{\bbH}$, whose mapping out function $(g_t)_{t\ge0}$ (i.e., the unique conformal transformation from the unbounded component of $\mathbb{H}\backslash \eta([0,t])$ to $\mathbb{H}$ such that $\lim_{|z|\to\infty}|g_t(z)-z|=0$) can be described by
\begin{equation}\label{eq:def-sle}
g_t(z) = z+\int_0^t \frac{2}{g_s(z)-W_s}ds, \ z\in\mathbb{H},
\end{equation} 
where $W_t=\sqrt{\kappa}B_t$ is the Loewner driving function.

For $\kappa>0$, the radial $\SLE_\kappa$ in $\bbD$ from $1$ to $0$ is a random curve $\eta:[0,\infty) \to \ol \bbD$ with $\eta(0) = 1$ and $\lim_{t \to \infty} \eta(t) = 0$. Let $K_t$ be the compact subset of $\ol \bbD$ such that $\ol \bbD \backslash K_t$ is the connected component of $\bbD \backslash \eta([0,t])$ containing $0$, and let $g_t: \bbD \backslash K_t \to \bbD$ be the conformal map with $g_t(0) = 0$ and $g_t'(0) > 0$. The curve $\eta$ is parametrized by log conformal radius, meaning that for each $t$ we have $g_t'(0) = e^t$. It turns out that there is a random process $U_t \stackrel d=e^{i\sqrt{\kappa}B_t}$ (where $B_t$ is standard Brownian motion) such that 
\begin{equation}\label{eq-rad-sle}
dg_t(z) = \Phi(U_t, g_t(z))\,dt\quad \text{ for } z\in\mathbb{D}\backslash K_t  \text{ and }\Phi(u,z) := z\frac{u+z}{u-z}.
\end{equation}
{In fact,~\eqref{eq-rad-sle} 
and the initial condition $g_0(z) = z$ define the family  of conformal maps $(g_t)_{t \geq 0}$ and hence radial $\SLE_\kappa$}, see \cite{Law08} for details.

Let $\rho>-2$ and $x\in\partial \bbD$. The radial $\SLE_\kappa(\rho)$ process with force point at $x$ is characterized by the same radial Loewner evolution~\eqref{eq-rad-sle}, except that $U_t$ is the solution to 
\begin{equation*}
    dU_t = -\frac{\kappa}{2}U_tdt + i\sqrt{\kappa}U_t dB_t + \frac{\rho}{2}\Phi(g_t(x), W_t)dt.
\end{equation*}
It has been shown in~\cite{ig4} that the radial $\SLE_\kappa(\rho)$ process exists and generate a continuous curve up to time $\infty$. Moreover, for  $\rho<\frac{\kappa}{2}-2$ and $x = e^{i0^-}$, the curve a.s.\ hits the boundary $\partial\bbD\backslash\{1\}$. 

By taking conformal maps, one can also define radial $\SLE_\kappa(\rho)$ processes from 1 targeted at a given interior point $w\in\bbD$. For $\kappa\in(4,8)$ and $\rho = \kappa-6$, it has been shown in~\cite{SW05} that the radial $\SLE_\kappa(\kappa-6)$ satisfies \emph{target invariance}:
\begin{proposition}[Proposition 3.14 and Section 4.2 of~\cite{SheffieldCLE}]
\label{prop:target-invariance}
    Let $(a_k)_{k\ge1}$ be a countable dense sequence in $\bbD$. For $\kappa\in(4,8)$, there exists a coupling of radial $\SLE_\kappa(\kappa-6)$  curves $\eta^{a_k}$ in $\bbD$ from 1 and targeted at $a_k$ with force point $e^{i0^-}$ such that for any $k,l\ge 1$, $\eta^{a_k}$ and $\eta^{a_l}$ agree a.s. (modulo time change) up to the first time that the curves separate $a_k$ and $a_l$, and evolve independently thereafter. 
\end{proposition}
The above target invariance extends to the setting where some points $a_k$ lie on $\partial\bbD$, in which case the corresponding $\eta^{a_k}$ curves are chordal $\SLE_\kappa(\kappa-6)$ (see Section~\ref{subsec:ig} for a brief introduction to chordal $\SLE_\kappa(\underline\rho)$ curves). For $a\notin (a_k)_{k\ge1}$, we may take a subsequence $(a_{k_n})_{n\ge1}$ converging to $a$, from which we can a.s.\ uniquely define a curve $\eta^a$ targeted at $a$ using $(\eta^{a_{k_n}})_{n\ge1}$ such that for any $n \geq 1$, $\eta^a$ agrees with $\eta^{a_{k_n}}$ before the first time that the curves separate $a$ and $a_{k_n}$; see Section 4.2 of~\cite{SheffieldCLE}. For any given $a\in\bbD$, the law of $\eta^a$ is the radial $\SLE_\kappa(\kappa-6)$ curve targeted at $a$, and the coupling $(\eta^{a})_{a\in\ol \bbD}$  introduced above, whose law is invariant of the choice of $(a_k)_k$, is referred as \emph{the continuum exploration tree}.

For $\kappa \in (4,8)$, (the non-nested) ${\rm CLE}_\kappa$  is a random collection $\Gamma$ of non-simple loops. It was first introduced in \cite{SheffieldCLE} who constructed it using the continuum exploration tree. Without loss of generality assume $a_1=0$. We review the construction of the loop $\Loop$ surrounding the origin which a.s.\ exists; for   $k\ge2$, the corresponding loop $\cL^{a_k}$ surrounding $a_k$ can be constructed analogously. The CLE$_\kappa$ is then defined by $\Gamma = \{\cL^{a_k} : k \ge 1  \}$.

\begin{enumerate}[(i)]

    \item Let $\eta:=\eta^o$, whose law is a radial ${\rm SLE}_\kappa(\kappa-6)$ in $\bbD$ from 1 and targeted at $0$ with the force point $e^{i0^-}$. Let $\sigma_0 = 0$, and let $\sigma_1< \sigma_2 < \ldots$ be the subsequent times at which $\eta$ makes a closed loop around $0$ in either the clockwise or counterclockwise direction,  i.e.,   $\sigma_n$ is the first time $t > \sigma_{n-1}$ that $\eta([\sigma_{n-1}, t])$ separates $0$ from $\eta([0,\sigma_{n-1}])$.

    \item Let $\sigma_m$ be the first time that the loop is formed in the counterclockwise direction for some integer $m \geq 1$. Let $z$ be leftmost intersection point of $\eta([\sigma_{m-1}, \sigma_m])\cap \partial (\bbD\backslash\eta([0,\sigma_{m-1}]))$ on the boundary of the connected component of $\bbD\backslash\eta([0,\sigma_{m-1}])$ containing 0; see Figure~\ref{fig:CLE-construct}. 

    \item Let $t_0$ be the last time before $\sigma_m$ that $\eta$ visits $z$. Let $\wt\eta$ be the branch $\eta^{z}$ reparametrized so that $\wt\eta|_{[0,t_0]} = \eta|_{[0,t_0]}$. Then $\Loop$ is defined to be the loop $\wt\eta|_{[t_0,\infty)}$.  
    
\end{enumerate}
Thus, the  loop $\Loop$ agrees in law with the concatenation of $\eta([t_0,\sigma_m])$ and an independent chordal $\SLE_\kappa$ curve in the connected component of $\bbD \backslash \eta([0,\sigma_m])$ containing $z$ from $\eta(\sigma_m)$ to $z$. 

\begin{figure}[H]
    \centering
    \includegraphics[scale = 0.7]{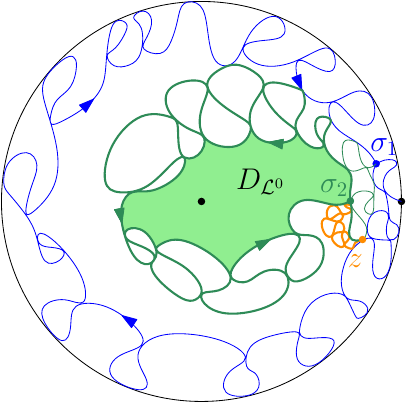}
    \caption{Illustration of a radial ${\rm SLE}_\kappa(\kappa-6)$ curve in the case of $m=2$. The loop $\Loop$ is the union of the bold green and orange curves, and $D_{\Loop}$ is the light green region.}
    \label{fig:CLE-construct}
\end{figure}

\begin{proposition}\label{prop:prelim}
In the setting of Theorems~\ref{thm:touch}-\ref{thm:CR-widetilde-D}, suppose  $\Loop$ is constructed using the curve $\eta$ as described above.  Let $D_1$ be the connected component of $\bbD \backslash \eta([0,\sigma_1])$ containing 0.    We have
    \begin{enumerate}
        \item The event $T=\{\Loop \cap \partial \bbD \neq \emptyset \}$ a.s.\ equals the event that  $\eta([0,\sigma_1])$ is a counterclockwise loop.
    
        \item On the event $T$ we have $D_{\Loop}=D_1$ hence ${\rm CR}(0, D_{\Loop}) = {\rm CR}(0, D_1)$ a.s.

        \item 
        The law of $\frac{{\rm CR}(0, D_{\Loop})}{{\rm CR}(0, D_1)}$ conditioned on the event $T^c$ is the same as  
        the unconditional law of ${\rm CR}(0, D_{\Loop})$.
    \end{enumerate}
\end{proposition}
\begin{proof}
     In the CLE construction described above, it is clear that if $\eta([0,\sigma_1])$ is a counterclockwise loop, then the event $T$ occurs. On the other hand, the probability that $\eta$ hits any boundary point  first when tracing a clockwise loop and then again when tracing a subsequent counterclockwise loop is zero (see e.g.~\cite[Proposition 4.9]{ig4}). Therefore if $\eta([0,\sigma_1])$ is a  clockwise loop, then it is a.s.\ the case that $\Loop$ is disjoint from $\partial\bbD$ (see also~\cite[Figure 2]{MSWCLEgasket}). This gives the first two assertions. Furthermore, using the Markov property of radial ${\rm SLE}_\kappa(\kappa - 6)$ and $\CLE_\kappa$ as in~\cite[Proposition 2.3]{MSWCLEgasket}, conditioned on the event $T^c$, we have a CLE inside $D_1$ with $\Loop$ being the loop surrounding the origin. This gives the desired description for the conditional law of  $\frac{{\rm CR}(0, D_{\Loop})}{{\rm CR}(0, D_1)}$ in the third assertion.
\end{proof}

Recall the domain $\wt D$ from Theorem~\ref{thm:CR-widetilde-D}, which is the connected component of $\bbD$ containing 0 after removing all the boundary touching loops in the $\CLE_\kappa$ $\Gamma$.
\begin{lemma}\label{lem:wt-D}
    In the setting of Proposition~\ref{prop:prelim}, on the event where $\eta([0,\sigma_1])$ is a  clockwise loop, $\wt D$ a.s.\ agrees with the domain $D_1$.
\end{lemma}
We need the following technical input for the proof of Lemma~\ref{lem:wt-D}.
\begin{lemma}\label{lem:dense-chordal}
    Let $\kappa\in(4,8)$ and $\eta$ be a chordal $\SLE_\kappa(\kappa-6)$ curve in $\bbH$ with force point at $0^-$. Then the points on the boundary of some counterclockwise loop made by   $\eta$ is dense on the trace of $\eta$.
\end{lemma}
\begin{proof} 
We first verify the analogous statement for a chordal $\SLE_\kappa$ curve $\hat \eta$ in $\bbH$.  Let $S$ be the union of the counterclockwise loops made by $\hat\eta$. Indeed, with positive probability $\hat p$, $\hat\eta([0,1])\cap S\neq\emptyset$, and by scale invariance   $\bbP(\hat\eta([0,\e])\cap S\neq\emptyset) = \hat p$ with every $\e>0$. Therefore by the Blumenthal's 0-1 law $\hat p = 1$, and by the domain Markov property $S$ is a dense subset of $\hat\eta$.

     Now by~\cite[Proposition 7.30]{MS16a}, the conditional law of $\hat\eta$ given its left and right boundaries is $\SLE_\kappa(\frac{\kappa}{2}-4;\frac{\kappa}{2}-4)$, and thus the same statement hold for chordal $\SLE_\kappa(\frac{\kappa}{2}-4;\frac{\kappa}{2}-4)$ curves. Since chordal $\SLE_\kappa(\frac{\kappa}{2}-4;\frac{\kappa}{2}-4)$ is boundary-filling, the lemma follows by applying ~\cite[Proposition 7.30]{MS16a} once again for the chordal $\SLE_\kappa(\kappa-6)$ curve $\eta$.
\end{proof}

\begin{proof}[Proof of Lemma~\ref{lem:wt-D}]
    We first prove that $D_1\subset \wt D$ a.s.\ under the event where $\eta([0,\sigma_1])$ is a clockwise loop. Let $(a_k)_{k\ge1}$ be the countable dense set in Proposition~\ref{prop:target-invariance}. By Proposition~\ref{prop:prelim} (with $0$ replaced by $a_k$), $\cL^{a_k}$ is a.s.\ disjoint from the boundary $\partial \bbD$ for every $k$ with $a_k\in D_1$. Therefore under this probability one event, we have $D_1\subset \wt D$, since otherwise there would be a boundary touching loop intersecting $D_1$.

    To prove  $ \wt D\subset D_1$ a.s., consider the coupling in Proposition~\ref{prop:target-invariance} between $\eta = \eta^o$ and $(\eta^{w_k})_{k\ge1}$ where $w_k$ is a dense subset of $\partial\bbD$. Then each $\eta^{w_k}$ is a chordal $\SLE_\kappa(\kappa-6)$ curve. Now by Lemma~\ref{lem:dense-chordal}, the points which lie on the boundary of some counterclockwise loop formed by $\eta^{w_k}$ is an a.s.\ dense subset of $\eta^{w_k}$. Then it follows from the continuum exploration tree construction that these counterclockwise loops formed by $\eta^{w_k}$ are parts of boundary touching loops in the CLE$_\kappa$ $\Gamma$, and from the coupling between $\eta$ and $(\eta^{w_k})_{k\ge1}$ the boundary touching loops contains a dense subset of $\eta([0,\sigma_1])$. Therefore  $\eta([0,\sigma_1])\subset \bbD\backslash\wt D$ a.s., which further implies that $ \wt D\subset D_1$ a.s.\ and conclude the proof. 
\end{proof}

\subsection{Liouville quantum gravity and Liouville fields}\label{subsec:lqg-lcft}

Let   $m_\bbH$ be the uniform probability measure on the unit circle   half circle $\bbH\cap\partial \bbD$.  Define the Dirichlet inner product $\langle f,g\rangle_\nabla = (2\pi)^{-1}\int_\bbH \nabla f\cdot\nabla g $ on the space $\{f\in C^\infty(\bbH):\int_\bbH|\nabla f|^2<\infty; \  \int f(z)m_\bbH(dz)=0\},$ and let $H(\bbH)$ be the closure of this space w.r.t.\ the inner product $\langle f,g\rangle_\nabla$. Let $(f_n)_{n\ge1}$ be an orthonormal basis of $H(\bbH)$, and $(\alpha_n)_{n\ge1}$ be a collection of independent standard Gaussian variables. Then the summation
$$h_\bbH = \sum_{n=1}^\infty \alpha_nf_n$$
a.s.\ converges in the space of distributions on $\bbH$, and $h_\bbH$ is the \emph{Gaussian free field (GFF)} on $\bbH$ normalized such that $\int h_\bbH(z)m_\bbH(dz) = 0$. See~\cite[Section 4.1.4]{DMS14} for more details.

Let $|z|_+ = \max\{|z|,1\}$. For $z,w\in\bar{\bbH}$, we define
$$
    G_\bbH(z,w)=-\log|z-w|-\log|z-\bar{w}|+2\log|z|_++2\log|w|_+;  \quad G_\bbH(z,\infty) = 2\log|z|_+.
$$
Then $h_\bbH$ is the centered Gaussian field on $\bbH$ with covariance structure $\bbE [h_\bbH(z)h_\bbH(w)] = G_\bbH(z,w)$.

We now introduce the notion of a \emph{Liouville quantum gravity (LQG)} surface. Let $\gamma\in(0,2)$ and $Q=\frac{2}{\gamma}+\frac{\gamma}{2}$. Consider the space of pairs $(D,h)$, where $D\subseteq \mathbb C$ is a planar domain and $h$ is a distribution on $D$ (often some variant of the GFF). For a conformal map $g:D\to\tilde D$ and a generalized function $h$ on $D$, define the generalized function $g\bullet_\gamma h$ on $\tilde{D}$ by setting 
\begin{equation}\label{eq:lqg-changecoord}
    g\bullet_\gamma h:=h\circ g^{-1}+Q\log|(g^{-1})'|.
\end{equation}
 Define the equivalence relation $\sim_\gamma$ {as follows. We say that} $(D, h)\sim_\gamma(\wt{D}, \wt{h})$ if there is a {conformal map $g:D\to\wt{D}$ such that $\tilde h = g\bullet_\gamma h$}.
A \textit{quantum surface} $S$ is an equivalence class of pairs $(D,h)$ under the {equivalence} relation $\sim_\gamma$, and we say {that} $(D,h)$ is an \emph{embedding} of $S$ if $S = (D,h)/\mathord\sim_\gamma$. Likewise, a \emph{quantum surface with} $k$ \emph{marked points} is an equivalence class of tuples of the form
	$(D, h, x_1,\dots,x_k)$, where $(D,h)$ is a quantum surface, the points  $x_i\in {\overline{D}}$, and with the further
	requirement that marked points (and their ordering) are preserved by the conformal map $\varphi$ in \eqref{eq:lqg-changecoord}. A \emph{curve-decorated quantum surface} is an equivalence class of tuples $(D, h, \eta_1, ..., \eta_k)$,
	where $(D,h)$ is a quantum surface, $\eta_1, ..., \eta_k$ are curves in $\overline D$, and with the further
	requirement that $\eta$ is preserved by the conformal map $g$ in \eqref{eq:lqg-changecoord}. Similarly, we can
	define a curve-decorated quantum surface with $k$ marked points. 

 For a $\gamma$-quantum surface $(D, h, z_1, ..., z){/{\sim_\gamma}}$, its \textit{quantum area measure} $\mu_h$ is defined by taking the weak limit $\epsilon\to 0$ of $\mu_{h_\epsilon}:=\epsilon^{\frac{\gamma^2}{2}}e^{\gamma h_\epsilon(z)}d^2z$, where $d^2z$ is the Lebesgue area measure on $D$ and $h_\epsilon(z)$ is the average of $h$ over $\partial B(z, \epsilon) \cap D$. When $D=\mathbb{H}$, we can also define the  \textit{quantum boundary length measure} $\nu_h:=\lim_{\epsilon\to 0}\epsilon^{\frac{\gamma^2}{4}}e^{\frac{\gamma}{2} h_\epsilon(x)}dx$ where $h_\epsilon (x)$ is the average of $h$ over the semicircle $\{x+\epsilon e^{i\theta}:\theta\in(0,\pi)\}$. It has been shown in \cite{DS11, SW16} that all these weak limits are well-defined  for the GFF and its variants we are considering in this paper, {and that} $\mu_{h}$ and $\nu_{h}$ {can} be conformally extended to other domains using the relation $\bullet_\gamma$.

Consider a pair $(D,h)$ where $D$ is now a closed set (not necessarily homeomorphic to a {closed} disk)
such that each component of its interior together with its prime-end boundary is homeomorphic to
the closed disk, and $h$ is only defined as a distribution on each of these components.  We extend the equivalence relation $\sim_\gamma$ described after~\eqref{eq:lqg-changecoord}, such that $g$ is now allowed to be any homeomorphism from $D$ to $\tilde D$ that is conformal on each component of the interior of $D$.  A \emph{beaded quantum surface} $S$ is an equivalence class of pairs $(D,h)$ under the equivalence relation $\sim_\gamma$ as described above, and we say $(D,h)$ is an embedding of $S$ if $S = (D,h)/{\sim_\gamma}$. Beaded quantum surfaces with marked points and curve-decorated beaded quantum surfaces can be defined analogously.

We now introduce Liouville fields, which are  closely related with Liouville quantum gravity. Note that these definitions implicitly depend on the choice of LQG parameter $\gamma$ via $Q = \frac\gamma2 + \frac2\gamma$. Write $P_\bbH$ for the law of the Gaussian free field $h_\bbH$ defined at the beginning of this section. 

\begin{definition}
Let $(h, \mathbf c)$ be sampled from $P_\bbH \times [e^{-Qc} dc]$ and $\phi = h-2Q\log|z|_+ + \mathbf c$. We call $\phi$ the Liouville field on $\bbH$, and we write
$\LF_\bbH$ for the law of $\phi$.
\end{definition}

\begin{definition}\label{def:lf-insertion}
Let  $(\alpha,w)\in\bbR\times\bbH$ and $(\beta,s)\in\bbR\times\partial\bbH$. Let
\begin{equation*}
\begin{split}
&C_{\bbH}^{(\alpha,w),(\beta,s)} = (2\,\mathrm{Im}\,w)^{-\frac{\alpha^2}{2}}|w|_+^{-2\alpha(Q-\alpha)}|s|_+^{-\beta(Q-\frac{\beta}{2})}e^{\frac{\alpha\beta}{2}G_\bbH(w,s)}.   
\end{split}
\end{equation*}
Let $(h,\mathbf{c})$ be sampled from $C_{\bbH}^{(\alpha,w),(\beta,s) }P_\bbH\times [e^{(\alpha+\frac{\beta}{2}-Q)c}dc]$, and 
\begin{equation*}
    \phi(z) = h(z)-2Q\log|z|_++\alpha G_\bbH(z,w)+\frac{\beta}{2}G_\bbH(z,s)+\mathbf{c}.
\end{equation*}
We write $\LF_{\bbH}^{(\alpha,w),(\beta,s) }$ for the law of $\phi$ and call a sample from $\LF_{\bbH}^{(\alpha,w),(\beta,s) }$ the Liouville field on $\bbH$ with insertions ${(\alpha,w),(\beta,s) }$.
\end{definition}
\begin{definition}[Liouville field with  boundary insertions]\label{def-lf-H-bdry}
	Let $\beta_i\in\mathbb{R}$  and $s_i\in \partial\mathbb{H}\cup\{\infty\}$ for $i = 1, ..., m,$ where $m\ge 1$ and all the $s_i$'s are distinct. Also assume $s_i\neq\infty$ for $i\ge 2$. We say $\phi$ is a \textup{Liouville Field on $\mathbb{H}$ with insertions $\{(\beta_i, s_i)\}_{1\le i\le m}$} if $\phi$ can be produced as follows by  first sampling $(h, \mathbf{c})$ from $C_{\mathbb{H}}^{(\beta_i, s_i)_i}P_\mathbb{H}\times [e^{(\frac{1}{2}\sum_{i=1}^m\beta_i - Q)c}dc]$ with
	$$C_{\mathbb{H}}^{(\beta_i, s_i)_i} =
	\left\{ \begin{array}{rcl} 
\prod_{i=1}^m  |s_i|_+^{-\beta_i(Q-\frac{\beta_i}{2})} \exp(\frac{1}{4}\sum_{j=i+1}^{m}\beta_i\beta_j G_\mathbb{H}(s_i, s_j)) & \mbox{if} & s_1\neq \infty\\
	\prod_{i=2}^m  |s_i|_+^{-\beta_i(Q-\frac{\beta_i}{2}-\frac{\beta_1}{2})}\exp(\frac{1}{4}\sum_{j=i+1}^{m}\beta_i\beta_j G_\mathbb{H}(s_i, s_j)) & \mbox{if} & s_1= \infty
	\end{array} 
	\right.
	 $$
	and then setting
	\begin{equation}\label{eqn-def-lf-H}
\phi(z) = h(z) - 2Q\log|z|_++\frac{1}{2}\sum_{i=1}^m\beta_i G_\mathbb{H}(s_i, z)+\mathbf{c}
	\end{equation}
with the convention $G_\mathbb{H}(\infty, z) = 2\log|z|_+$. We write $\textup{LF}_{\mathbb{H}}^{(\beta_i, s_i)_i}$ for the law of $\phi$.
\end{definition}

\subsection{Quantum disks and triangles}\label{subsec:qsurfaces}
In this section we gather the definitions for various quantum surfaces considered in this paper. These surfaces are constructed using the Gaussian free field and Liouville fields introduced in Section~\ref{subsec:lqg-lcft}. 

We begin with the  quantum disks with two points on the boundary introduced in~\cite{DMS14,AHS20}. Recall the space $H(\bbH)$ at the beginning of Section~\ref{subsec:lqg-lcft}. This space admits a natural decomposition $H(\bbH) = H_1(\bbH)\oplus H_2(\bbH)$, where $H_1(\bbH)$ (resp.\ $H_2(\bbH)$) is the set of functions in $H(\bbH)$ with  same value (resp.\ average zero) on the semicircle $\{z\in\bbH:\ |z| = r\}$ for each $r>0$.

\begin{definition}[Thick quantum disk]\label{def-quantum-disk}
Fix a weight parameter $W\ge\frac{\gamma^2}{2}$ and let $\beta = \gamma+ \frac{2-W}{\gamma}\le Q$. 
		Let $(B_s)_{s\ge0}$ and $(\wt{B}_s)_{s\ge0}$ be independent standard one-dimensional Brownian motions conditioned on  $B_{2t}-(Q-\beta)t<0$ and  $ \wt{B}_{2t} - (Q-\beta)t<0$ for all $t>0$.   Let $\mathbf{c}$ be sampled from the infinite measure $\frac{\gamma}{2}e^{(\beta-Q)c}dc$ on $\bbR$ independently from $(B_s)_{s\ge0}$ and $(\wt{B}_s)_{s\ge0}$.
			Let 	
			\begin{equation*}
				Y_t=\left\{ \begin{array}{rcl} 
					B_{2t}+\beta t+\mathbf{c} & \mbox{for} & t\ge 0,\\
					\wt{B}_{-2t} +(2Q-\beta) t+\mathbf{c} & \mbox{for} & t<0.
				\end{array} 
				\right.
			\end{equation*}
			 Let $h$ be a free boundary  GFF on $\mathbb{H}$ independent of $(Y_t)_{t\in\bbR}$ with projection onto $H_2(\mathbb{H})$ given by $h_2$. Consider the random distribution
			\begin{equation*}
				\psi(\cdot)=Y_{-\log|\cdot|} + h_2(\cdot) \, .
		\end{equation*}
		Let the infinite measure $\mathcal{M}_{0,2}^{\mathrm{disk}}(W)$ be the law of $({\mathbb{H}}, \psi,0,\infty)/\mathord\sim_\gamma $. 
		We call a sample from $\mathcal{M}_{0,2}^{\textup{disk}}(W)$ a \emph{quantum disk} of weight $W$ with two marked points.
		
		We call $\nu_\psi((-\infty,0))$ and $\nu_\psi((0,\infty))$ the left and right{, respectively,} quantum   {boundary} length of the quantum disk  $(\mathbb{H}, \psi, 0, \infty)/{\sim_\gamma}$.
	\end{definition}

\begin{definition}\label{def:QD}
  Let $(\mathbb{H},\phi, 0, \infty)$ be the embedding of a sample from $\Md_{0,2}(2)$ as in Definition~\ref{def-quantum-disk}. We write $\QD$ for the law of $(\bbH,\phi)/{\sim_\gamma}$ weighted by $\nu_\phi(\partial\bbH)^{-2}$, {and $\QD_{0,1}$ for the law of $(\bbH,\phi,0)/{\sim_\gamma}$ weighted by $\nu_\phi(\partial\bbH)^{-1}$}. Let $\QD_{1,1}$ be the law of $(\bbH,\phi,0,z)/{\sim_\gamma}$ where $(\bbH,\phi,0)$ is sampled from $\mu_\phi(\bbH)\QD_{0,1}$ and $z$ is sampled according to $\mu_\phi^\#$.  
\end{definition}

	When $0<W<\frac{\gamma^2}{2}$, we define the \emph{thin quantum disk} as the concatenation of weight $\gamma^2-W$ thick disks with two marked points as in \cite[Section 2]{AHS20}.

	\begin{definition}[Thin quantum disk]\label{def-thin-disk}
	For $W\in(0, \frac{\gamma^2}{2})$, the infinite measure $\mathcal{M}_{0,2}^{\textup{disk}}(W)$ is defined as follows. First sample a random variable $T$ from the infinite measure $(1-\frac{2}{\gamma^2}W)^{-2}\textup{Leb}_{\mathbb{R}_+}$; then sample a Poisson point process $\{(u, \mathcal{D}_u)\}$ from the intensity measure $\mathds{1}_{t\in [0,T]}dt\times \mathcal{M}_{0,2}^{\textup{disk}}(\gamma^2-W)$; and finally consider the ordered (according to the order induced by $u$) collection of doubly-marked thick quantum disks $\{\mathcal{D}_u\}$, called a \emph{thin quantum disk} of weight $W$.
		
		Let $\mathcal{M}_{0,2}^{\textup{disk}}(W)$ be the law of this ordered collection of doubly-marked quantum disks $\{\mathcal{D}_u\}$.
		The left (resp.\ right) boundary length of a sample from $\mathcal{M}_{0,2}^{\textup{disk}}(W)$ is defined to be the sum of the left (resp.\ right) boundary lengths of the quantum disks $\{\mathcal{D}_u\}$.
	\end{definition}

 We  also define quantum disks with one bulk and one boundary insertion.
\begin{definition}\label{def:mdisk11}
    Fix $\alpha\in\bbR,\beta<Q$.  Let $\phi$ be a sample from $\frac{1}{Q-\beta}\LF_{\bbH}^{(\alpha,i),(\beta,0)}$. We define the infinite measure $\Md_{1,1}(\alpha,\beta)$ to be the law of $(\bbH,\phi,i,0)/{\sim_\gamma}$.
\end{definition}
\begin{proposition}[Proposition 3.9 of~\cite{ARS21}]\label{prop:mdisk11-QD}
    For some constant $C$, we have $\Md_{1,1}(\gamma,\gamma) = C\QD_{1,1}$.
\end{proposition}

Next we recall the notion of quantum triangle as in \cite{ASY22}. It is a quantum surface parameterized by weights $W_1,W_2,W_3>0$ and defined based on Liouville fields with three insertions and the thick-thin duality.

\begin{definition}[Thick quantum triangles]\label{def-qt-thick}
	Fix {$W_1, W_2, W_3>\frac{\gamma^2}{2}$}. Set $\beta_i = \gamma+\frac{2-W_i}{\gamma}<Q$ for $i=1,2,3$, and {let $\phi$ be sampled from $\frac{1}{(Q-\beta_1)(Q-\beta_2)(Q-\beta_3)}\textup{LF}_{\mathbb{H}}^{(\beta_1, \infty), (\beta_2, 0), (\beta_3, 1)}$}. Then we define the infinite measure $\textup{QT}(W_1, W_2, W_3)$ to be the law of $(\bbH, \phi, \infty, 0,1)/{\sim_\gamma}$.
\end{definition}

\begin{definition}[Quantum triangles with thin vertices]\label{def-qt-thin}
	Fix {$W_1, W_2, W_3\in (0,\frac{\gamma^2}{2})\cup(\frac{\gamma^2}{2}, \infty)$}. Let $I:=\{i\in\{1,2,3\}:W_i<\frac{\gamma^2}{2}\}$. Let $\tilde W_i = W_i$ if $i \not \in I$ and $\tilde W_i = \gamma^2 - W_i$ if $i \in I$. Sample $(S_0, (S_i)_{i \in I})$ from 
 \[\QT(\tilde W_1, \tilde W_2, \tilde W_3) \times \prod_{i\in I } (1-\frac{2W_i}{\gamma^2}) \Md_2(W_i).  \]
 Embed $S_0$ as $(\tilde D, \phi, \tilde a_1, \tilde a_2, \tilde a_3)$, for each $i \not \in I$ let $a_i = \tilde a_i$, and for each $i \in I$ embed $S_i$ as $(\tilde D_i, \phi, \tilde a_i, a_i)$ in such a way that the $\tilde D_i$ are disjoint and $\tilde D_i \cap \tilde D = \tilde a_i$. Let $D = \tilde D \cup \bigcup_{i \in I} \tilde D_i$ and let $\QT(W_1, W_2, W_3)$ be the law of $(D, \phi, a_1, a_2, a_3)/{\sim_\gamma}$.
\end{definition} 

\begin{figure}[ht]
	\centering
	\includegraphics[scale=0.43]{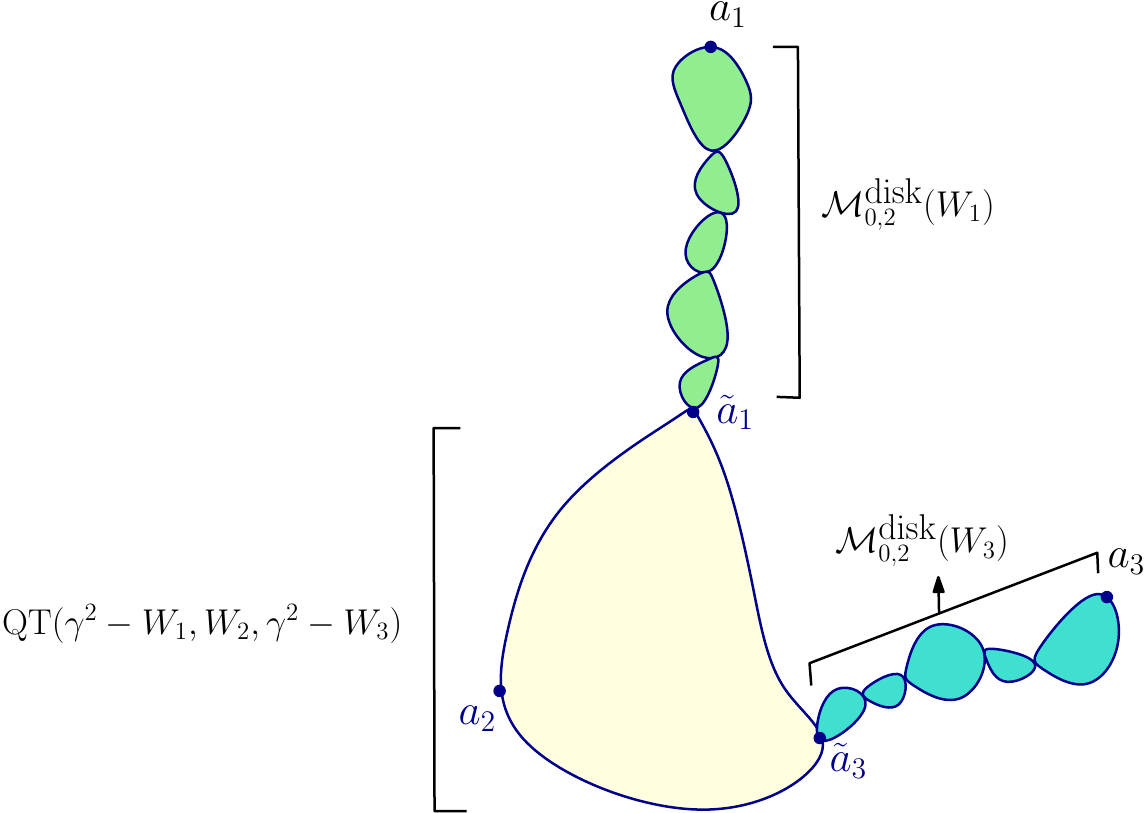}
	\caption{A quantum triangle with $W_2\geq\frac{\gamma^2}{2}$ and $W_1,W_3<\frac{\gamma^2}{2}$ embedded as $(D,\phi,a_1,a_2,a_3).$  The two thin disks (colored green) are concatenated with the thick triangle (colored yellow) at points $\tilde{a}_1$ and $\tilde{a}_3$.}\label{fig-qt}
\end{figure}

See Figure~\ref{fig-qt} for an illustration. The points $a_1,a_2,a_3$ are often referred as the weight $W_1,W_2,W_3$ vertices. 
If one or more of $W_1,W_2,W_3$ is equal to $\frac{\gamma^2}{2}$, then the measure $\QT(W_1,W_2,W_3)$ can be defined by a limiting procedure. See \cite[Section 2.5]{ASY22} for more details. This case is not needed in our paper hence we exclude it from certain statements.

For $W>0$, we write $\Md_{0,2,\bullet}(W)$ for the law of the three-pointed quantum surfaces obtained by (i) sampling a quantum disk from $\Md_{0,2}(W)$ and weighting its law by the quantum length of its left boundary arc and (ii) sampling a marked point on the left boundary arc from the probability measure proportional to the quantum boundary length measure. Then we have
\begin{lemma}
\label{lem:QT(W,2,W)}
For $W \in (0, \frac{\gamma^2}{2}) \cup (\frac{\gamma^2}{2}, \infty)$, we have
$$
  \mathcal{M}^{\rm disk}_{0,2, \bullet}(W) = \frac{\gamma(Q-\gamma)}{2} {\rm QT}(W,2,W)\,.
$$
\end{lemma}
\begin{proof}
    By~\cite[Lemma 6.12]{ASY22}, we have $\mathcal{M}^{\rm disk}_{0,2, \bullet}(W) = C {\rm QT}(W,2,W)$. The value of the constant $C$ follows from a comparison over~\cite[Proposition 2.18]{AHS21}, Definition~\ref{def-qt-thick} and Definition~\ref{def-qt-thin}.
\end{proof}

Given a measure $\mathcal{M}$ on quantum surfaces, we can disintegrate $\mathcal{M}$ over the quantum lengths of the boundary arcs.   For instance, for $W>0$, one can {disintegrate} the measure $\Md_{0,2}(W)$ according to its the quantum length of the left and right boundary arc, i.e., 
 \begin{equation}\label{eq:disint}
     \Md_{0,2}(W) = \int_0^\infty\int_0^\infty \Md_{0,2}(W;\ell_1,\ell_2)d\ell_1\,d\ell_2,
 \end{equation}
where each $\Md_{0,2}(W;\ell_1,\ell_2)$ is supported on the set of doubly-marked quantum surfaces with left and right boundary arcs having quantum lengths $\ell_1$ and $\ell_2$, respectively. One can also define  $\Md_{0,2}(W;\ell) := \int_0^\infty \Md_{0,2}(W;\ell,\ell')d\ell'$, i.e., the disintegration over the quantum length of the left (or right) boundary arc. 

We can also disintegrate $\Md_{1,1}(\alpha,\beta)$ over the boundary length, where for $\ell>0$, there exists a measure $\Md_{1,1}(\alpha,\beta;\ell)$ supported on quantum surfaces with one bulk marked point and one boundary marked point whose boundary has length $\ell$ such that 
     \begin{equation}\label{eq:M11-dis}
     \Md_{1,1}(\alpha,\beta) = \int_0^\infty \Md_{1,1}(\alpha,\beta;\ell)\,d\ell.
 \end{equation}
 Moreover, we have
\begin{lemma}\label{lem:M11-bdry-length}
    For $\alpha\in\bbR$, $\beta<Q$ with $\alpha+\frac{\beta}{2}>Q$,  one has $|\Md_{1,1}(\alpha,\beta;\ell)| = C\ell^{\frac{2\alpha+
     \beta-2Q}{\gamma}-1}$ for some finite constant $C>0$.
\end{lemma}
\begin{proof}
    The proof is the same as that of~\cite[Lemma 2.7]{AY23}.
\end{proof}

Similarly, for quantum triangles, we have
\begin{equation}\label{eq:QT-dis}
    \QT(W_1,W_2,W_3) = \iiint_{\bbR_+^3}\QT(W_1,W_2,W_3;\ell_1,\ell_2,\ell_3)\ d\ell_1d\ell_2d\ell_3
\end{equation}
where $\QT(W_1,W_2,W_3;\ell_1,\ell_2,\ell_3)$ is the measure supported on the set of quantum surfaces $(D,\phi,a_1,a_2,a_3)/{\sim_\gamma}$ such that the boundary arcs between $a_1a_2$, $a_1a_3$ and $a_2a_3$ have quantum lengths $\ell_1,\ell_2,\ell_3$. We can also disintegrate over one or two boundary arc lengths of quantum triangles. For instance, we can define 
$$\QT(W_1,W_2,W_3;\ell_1,\ell_2) = \int_0^\infty \QT(W_1,W_2,W_3;\ell_1,\ell_2,\ell_3)\ d\ell_3$$
and $$\QT(W_1,W_2,W_3;\ell_1) = \iint_{\bbR^2_+}\QT(W_1,W_2,W_3;\ell_1,\ell_2,\ell_3)\ d\ell_2 d\ell_3.$$

\subsection{Generalized quantum surfaces for $\gamma \in (\sqrt2, 2)$}

In this section we recall the forested lines and generalized quantum surfaces considered in~\cite{DMS14,MSW-non-simple-2021,AHSY23}, following the treatment of~\cite{AHSY23}. For $\kappa\in(4,8)$, the forested lines are defined in~\cite{DMS14} using the \emph{$\frac{\kappa}{4}$-stable looptrees} studied in~\cite{CK13looptree}. We set $\gamma = \frac{4}{\sqrt{\kappa}}$.  Let $(X_t)_{t\ge0}$ be a stable L\'{e}vy process starting from 0 of index $\frac{\kappa}{4}\in(1,2)$ with only upward jumps, {so} $X_t\overset{d}{=}t^{\frac{4}{\kappa}}X_1$ for any $t>0$. As shown in~\cite{CK13looptree},  one can construct a tree of topological disks from $(X_t)_{t\geq0}$ as in Figure~\ref{fig:forestline-def}, and the forested line is defined by replacing each disk with an independent sample of the probability measure obtained from $\QD$ by conditioning on the boundary length to be the size of the corresponding jump. 
The quantum disks are glued in a clockwise length-preserving way, where the rotation is chosen uniformly at random. The unique point corresponding to $(0,0)$ on the graph of $X$ is called the \emph{root}.  
The closure of the collection of the points on the boundaries of the quantum disks is referred as the \emph{forested boundary arc}, while the set of the points corresponding to the running infimum of $(X_t)_{t\ge0}$ is called the \emph{line boundary arc}. Since $X$ only has positive jumps, the quantum disks are lying on the same side of the line boundary arc. 

\begin{figure}
    \centering
    \begin{tabular}{cc} 
		\includegraphics[scale=0.49]{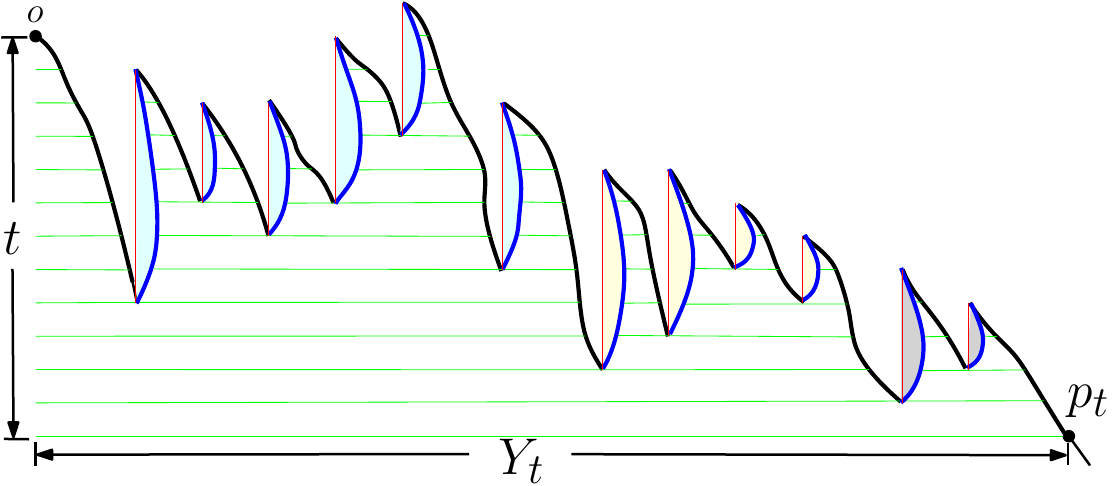}
		&
	   \includegraphics[scale=0.65]{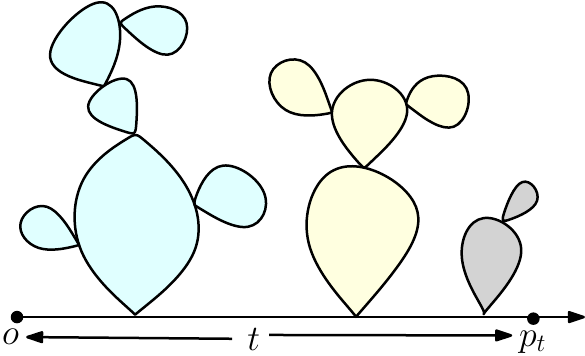}
	\end{tabular}
 \caption{\textbf{Left:} The graph of the L\'{e}vy process $(X_s)_{s>0}$ with only upward jumps. We draw the blue curves for each of the jump, and identify the points that are on the same green horizontal line.   \textbf{Right:} The L\'{e}vy tree of disks obtained from the left panel. For each topological disk we assign a quantum disk $\QD$ conditioned on having the same boundary length as the size of the jump, with the points on each red line in the left panel collapsed to a single point. The quantum length of the line segment between the root $o$ and the point $p_t$ is $t$, while the segment along the forested boundary between $o$ and $p_t$ has generalized quantum length $Y_t = \inf\{s>0:X_s\le -t\}$, i.e., $Y_t$ is the first time when $X_s$ lie below $-t$. As discussed in~\cite{AHSY23}, one can further make It\^{o} decomposition for $X_t - \inf_{s\leq t}X_s$ over 0, and each excursion would correspond to a single tree of disk on the right panel.  }\label{fig:forestline-def}
 \end{figure}

\begin{definition}[Forested line]\label{def:forested-line}
For $\gamma \in (\sqrt2,2)$, let $(X_t)_{t\geq 0}$ be a stable L\'evy process of index $\frac{4}{\gamma^2}>1$ with only positive jumps {satisfying $X_0=0$ a.s.}. For $t>0$, let $Y_t=\inf\{s>0:X_s\le -t\}$, and {fix the multiplicative constant of $X$ such that} $\bbE[e^{-Y_1}] = e^{-1}$. Define the forested line as {described above}.

 The line boundary arc is parametrized by quantum length. The forested boundary arc is parametrized by \emph{generalized quantum length}; that is, the length of the corresponding interval of $(X_t)$. For a point $p_t$ on the line boundary arc with LQG distance $t$ to the root, the segment of the forested boundary arc between $p_t$ and the root has generalized quantum length $Y_t$. 
\end{definition}
 
As  in~\cite{AHSY23}, one can define a \emph{truncation} operation on forested lines. For $t>0$ and a forested line $\Loop$ with root $o$, mark the point $p_t$ on the line boundary arc with quantum length $t$ from $o$. By \emph{truncation of $\Loop$ at quantum length $t$}, we refer to the surface $\cL_t$ which is the union of the line boundary arc and the quantum disks on the forested boundary arc between $o$ and $p_t$. In other words, $\cL_t$ is the surface generated by $(X_s)_{0\le s\le Y_t}$ in the same way as Definition~\ref{def:forested-line}, and the generalized quantum length of the forested boundary arc of $\cL_t$ is $Y_t$. The beaded quantum surface $\cL_t$ is called a forested line segment.

\begin{definition}\label{def:line-segment}
    Fix $\gamma\in(\sqrt{2},2)$. Define $\mathcal{M}_2^\mathrm{f.l.}$ as the law of the surface obtained by first sampling $\mathbf{t}\sim \mathrm{Leb}_{\bbR_+}$ and truncating an independent forested line at quantum length $\mathbf{t}$. 
\end{definition}

The following is from~\cite[Lemma 3.5]{AHSY23}.
\begin{lemma}[Law of forested segment length]\label{lem:fs-len}
Fix $q \in \bbR$. 
Suppose we sample $\mathbf t \sim 1_{t > 0} t^{-q} dt$ and independently sample a forested line $\Loop$.  
For $q<2$,  the law of  
$Y_{\mathbf t}$ is $\ C_q  \cdot1_{L>0} L^{-\frac{\gamma^2}4q + \frac{\gamma^2}4 - 1} dL.$ where $C_q := \frac{\gamma^2}4 \bbE[Y_1^{\frac{\gamma^2}4 (q-1)}]<\infty$. If $q\geq2$, then for any $0<a<b$, the event $\{Y_{\mathbf t}\in[a,b]\}$ has infinite measure. 
\end{lemma} 

Now we recall the  definition of generalized quantum surfaces in~\cite{AHSY23}. Let $n\ge1$, and $(D,\phi,z_1,...,z_n)$ be an embedding of a (possibly beaded) quantum surface $S$ of finite volume, 
with $z_1,...,z_n\in\partial D$  ordered clockwise. 
    We sample independent forested lines $\cL^1,...,\cL^n$, truncate them such that their quantum lengths match the length of boundary segments $[z_1,z_2],...,[z_n,z_1]$ and glue them to $\partial D$ correspondingly. 
    Let $S^f$ be the resulting beaded quantum surface. 
    
\begin{definition}\label{def:f.s.}
   We call a beaded quantum surface $S^f$ as above a (finite volume) \emph{generalized quantum surface}. 
   We call this procedure \emph{foresting the boundary} of $S$, and  
   say $S$  \emph{is the spine of}  $S^f$. 
\end{definition}

We present two types of generalized quantum surfaces needed in Theorem~\ref{thm:welding} below.
\begin{definition}\label{def:q.t.f.s.}
    Let $\alpha,W,W_1,W_2,W_3>0$ and $\beta<Q$. Recall from Definitions~\ref{def-qt-thick} and~\ref{def-qt-thin} the notion $\QT(W_1,W_2,W_3)$, and the notion $\Md_{1,1}(\alpha,\beta)$ from Definition~\ref{def:mdisk11}. We write $\QT^f(W_1,W_2,W_3)$ for the law of the generalized quantum surface obtained by foresting the three boundary arcs of a quantum triangle sampled from $\QT(W_1,W_2,W_3)$. Likewise, we write $\Mfd_{1,1}(\alpha,\beta)$ for the law of the generalized quantum surface obtained by foresting the   boundary arc of a quantum disk sampled from $\Md_{1,1}(\alpha,\beta)$, and define $\Mfd_{0,2}(W)$ via $\Md_{0,2}(W)$ similarly.
\end{definition}

Recall the disintegration \eqref{eq:disint} of the quantum disk measure. By disintegrating over the values of $Y_t$, we can similarly define a disintegration of the measure $\mathcal M_2^\mathrm{f.l.}$: 
\[
\mathcal{M}_2^\mathrm{f.l.} = \int_{\bbR_+^2}\mathcal{M}_2^\mathrm{f.l.}(t;\ell)\,dt\,d\ell.
\]
where $\mathcal{M}_2^\mathrm{f.l.}(t;\ell)$ is the measure on forested line segments with quantum length $t$ for the line boundary arc and generalized quantum length $\ell$ for the forested boundary arc. We write $\mathcal{M}_2^\mathrm{f.l.}(\ell):=\int_0^\infty\mathcal{M}_2^\mathrm{f.l.}(t;\ell)dt$, i.e., the law of forested line segments whose forested boundary arc has generalized quantum length $\ell$.  A similar disintegration holds as in~\eqref{eq:M11-dis} and ~\eqref{eq:QT-dis} for  $\Mfd_{1,1}(\alpha,\beta)$ and $\QT^f(W_1,W_2,W_3)$.
Indeed, this follows by defining the measure $\Mfd_{1,1}(\alpha,\beta;\ell)$ via
\begin{equation}\label{eq:def-forest-bdry}
\int_{\bbR_+}\mathcal{M}_2^\mathrm{f.l.}(t;\ell)\times \Md_{1,1}(\alpha,\beta;t) dt.    
\end{equation}
The measures  $\QT^f(W_1,W_2,W_3;\ell_1,\ell_2,\ell_3)$, $\QT^f(W_1,W_2,W_3;\ell_1,\ell_2)$ and $\QT^f(W_1,W_2,W_3;\ell_1)$ can be defined analogously.
The following is immediate from Lemma~\ref{lem:M11-bdry-length} and Lemma~\ref{lem:fs-len}.
\begin{lemma}\label{lem:M11-f.s.length}
Let $\alpha\in\bbR,\beta<Q$ with $\alpha+\frac{\beta}{2}>Q$. Then there exists a constant $c>0$ such that $|\Mfd_{1,1}(\alpha,\beta;\ell)| = c\ell^{\frac{\gamma}{4}(2\alpha+\beta-2Q)-1}.$    
\end{lemma}

\section{Radial $\SLE_\kappa(\kappa-6)$ from conformal welding of forested quantum triangles}\label{sec:weld}

Given a pair of certain quantum surfaces, following ~\cite{She16a,DMS14}, there exists a way to \emph{conformally weld} them together according to the length measure provided that the interface lengths agree; see e.g.~\cite[Section 4.1]{AHS21} and~\cite[Section 4.1]{ASY22} for more explanation. In~\cite{DMS14,AHSY23}, it is shown that for $\kappa\in(4,8)$, by drawing an independent {$\SLE_\kappa$} curve (or its variants) $\eta$ on top of a certain $\gamma$-LQG surface $S$ with $\gamma = \frac{4}{\sqrt\kappa}$, one cuts $S$ into independent generalized quantum surfaces $S_1$ and $S_2$ (conditioned on having the same interface length if $S$ has  finite volume). Moreover, given $(S_1,S_2)$, there a.s.\ exists a unique way to recover the pair $(S,\eta)$, and this procedure is defined to be the conformal welding of $S_1$ and $S_2$. As explained in~\cite{DMS14}, the points on the interfaces are glued together according to the generalized quantum length, which follows from the quantum natural time parametrization of the $\SLE_\kappa$ curves. This is originally done for forested lines in~\cite{DMS14} and later extended to forested line segments in~\cite{AHSY23}. As a consequence, this operation is well-defined for the generalized quantum surfaces from Definition~\ref{def:f.s.}. In light of the recent work~\cite{kavvadias2023conformal} on conformal removability of non-simple SLEs for $\kappa\in(4,\kappa_1)$, where $\kappa_1\approx 5.61$ (the constant is from \cite{GP-bubble-connect}), in this range it is possible to identify the recovery of $(S, \eta)$ from $(S_1,S_2)$ as actual conformal welding as in the $\kappa\in(0,4)$ case.

Let $\mathcal{M}^1, \mathcal{M}^2$ be measures on the space of (possibly generalized) quantum surfaces with boundary marked points. For $i=1,2$, fix a boundary arc $e_i$ of finite (possibly generalized) quantum length on a sample from $\mathcal{M}^i$, and define the measure $\mathcal{M}^i(\ell_i)$ via the disintegration 
$$\mathcal{M}^i = \int_0^\infty \mathcal{M}^i(\ell_i)d\ell_i$$
as in Section~\ref{subsec:qsurfaces}. For $\ell>0$, given a pair of surfaces sampled from the product measure $\mathcal{M}^1(\ell)\times \mathcal{M}^2(\ell)$, we can conformally weld them together according to (possibly generalized) quantum length. This yields a single quantum surface decorated  by a curve, namely, the welding interface. We write $\text{Weld}( \mathcal{M}^1(\ell), \mathcal{M}^2(\ell))$ for the law of the resulting curve-decorated surface, and let $$\text{Weld}(\mathcal{M}^1, \mathcal{M}^2):=\int_{\bbR}\, \text{Weld}(  \mathcal{M}^1(\ell), \mathcal{M}^2(\ell))\,d\ell$$ be the  welding of $\mathcal{M}^1,  \mathcal{M}^2$ along the boundary arcs $e_1$ and $e_2$. The case where we have only one surface and $e_1,e_2$ are different boundary arcs of this surface can be treated analogously. 

The aim of this section is to prove the following theorem; see Figure~\ref{fig:thm-welding} for an illustration.
\begin{theorem}\label{thm:welding}
    Let $\kappa\in(4,8)$ and $\gamma = \frac{4}{\sqrt\kappa}$. Then there exists a  $\gamma$-dependent constant $C_\gamma$  such that 
    \begin{equation}\label{eq:thm-welding}
        \Mfd_{1,1}(\gamma,\gamma)\otimes \mathrm{raSLE}_\kappa(\kappa-6) =C_\gamma \int_0^\infty \mathrm{Weld}\big(\QT^f(2-\frac{\gamma^2}{2},2-\frac{\gamma^2}{2},\gamma^2-2;\ell,\ell)\big)\,d\ell
    \end{equation}
    Here the left hand side of~\eqref{eq:thm-welding} stands for drawing an independent radial $\SLE_\kappa(\kappa-6)$ curve (with the force point lying immediately to the left of the root)  on top of a forested quantum disk from $\Mfd_{1,1}(\gamma,\gamma)$; on the right hand side of~\eqref{eq:thm-welding} the boundary arc between the weight $2-\frac{\gamma^2}2$ vertices is conformally welded to the boundary arc immediately counterclockwise to it.
\end{theorem}

\begin{figure}
    \centering
    \includegraphics[scale=0.75]{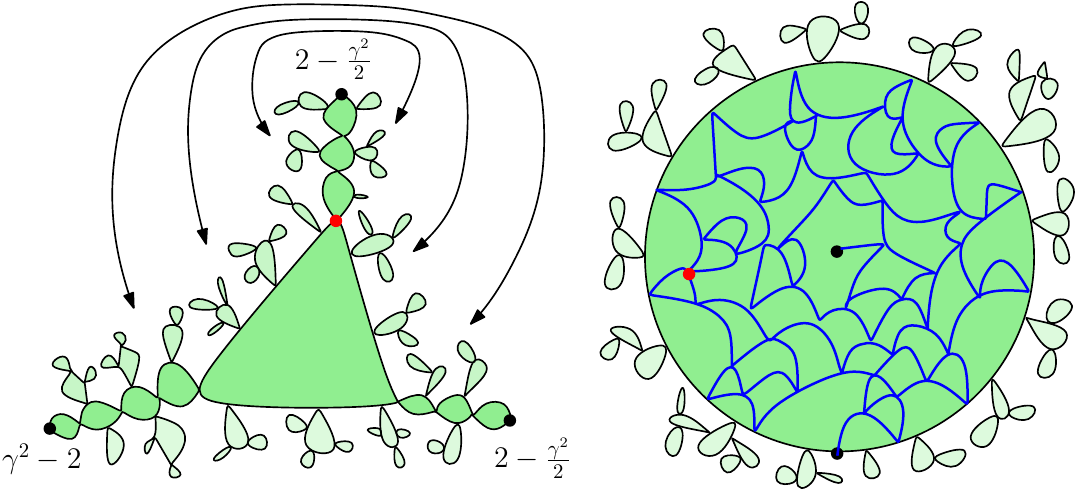}
    \caption{An illustration of the conformal welding in Theorem~\ref{thm:welding}. See Figure~\ref{fig:qt-decomposition} for a variant of this figure where the two pieces separated by the red point are colored differently.}
    \label{fig:thm-welding}
\end{figure}

 In Section~\ref{subsec:ig}, we recall certain variants of $\SLE$ and results of imaginary geometry in~\cite{MS16a,ig4}. In Section ~\ref{subsec:qt-qd-welding}, we recall the conformal welding of quantum disks and quantum triangles in~\cite{AHS20,ASY22,AHSY23}. In Section~\ref{subsec:mot-qs}, we give a mating-of-trees description of some special quantum disks and quantum triangles. Finally  in Section~\ref{subsec:pf-welding}, we prove Theorem~\ref{thm:welding}. Readers interested in the application of Theorem~\ref{thm:welding} to the proof of our mains theorem may skip the rest of this section in the first reading.

\subsection{Chordal $\SLE_\kappa(\underline\rho)$ and imaginary geometry}\label{subsec:ig}

 Fix force points $x^{k,L}<...<x^{1,L}<x^{0,L}= 0^-<x^{0,R}= 0^+< x^{1,R}<...<x^{\ell, R}$ and weights $\rho^{i,q}\in\bbR$, The $\SLE_\kappa(\underline{\rho})$ process is defined in the same way as $\SLE_\kappa$, except that its Loewner driving function $(W_t)_{t\ge 0}$ is now defined by 
\begin{equation}\label{eq:def-sle-rho}
\begin{split}
&W_t = \sqrt{\kappa}B_t+\sum_{q\in\{L,R\}}\sum_i \int_0^t \frac{\rho^{i,q}}{W_s-g_s(x^{i,q})}ds
\end{split}
\end{equation}
where $B_t$ is standard Brownian motion.
It has been proved in \cite{MS16a} that the SLE$_\kappa(\underline{\rho})$ process a.s. exists, is unique and generates a continuous curve until the \textit{continuation threshold}, the first time $t$ such that $W_t = V_t^{j,q}$ with $\sum_{i=0}^j\rho^{i,q}\le -2$ for some $j$ and $q\in\{L,R\}$.

 Let $D\subsetneq \mathbb{C}$ be a domain. We recall the construction the GFF on $D$ with \textit{Dirichlet} \textit{boundary conditions}  as follows. Consider the space of compactly supported smooth functions on $D$ with finite Dirichlet energy, and let $H_0(D)$ be its closure with respect to the inner product $(f,g)_\nabla=\int_D(\nabla f\cdot\nabla g)\ dx dy$. Then the (Dirichlet) GFF on $D$ is defined by 
\begin{equation}\label{eqn-def-gff}
h = \sum_{n=1}^\infty \xi_nf_n
\end{equation}
where $(\xi_n)_{n\ge 1}$ is a collection of i.i.d. standard Gaussians and $(f_n)_{n\ge 1}$ is an orthonormal basis of $H_0(D)$. The sum \eqref{eqn-def-gff} a.s.\ converges to a random distribution whose law is independent of the choice of the basis $(f_n)_{n\ge 1}$. For a function $g$ defined on $\partial D$ with harmonic extension $f$ in $D$ and a zero boundary GFF $h$, we say that $h+f$ is a GFF on $D$ with boundary condition specified by $g$. 
 See \cite[Section 4.1.4]{DMS14} for more details. 

Now we briefly recall the theory of imaginary geometry. For $\kappa>4$, let
$$ \wt\kappa = \frac{16}{\kappa}, \qquad  \lambda = \frac{\pi}{\sqrt\kappa}, \qquad  \wt\lambda = \frac{\pi\sqrt{\kappa}}{4}, \qquad \chi = \frac{\sqrt\kappa}{2} - \frac{2}{\sqrt\kappa}. $$
 Given a Dirichlet GFF $h^{\rm IG}$ on $\bbH$ with piecewise boundary conditions and $\theta\in\bbR$, it is possible to construct \emph{$\theta$-angle flow lines} $\eta_{\theta}^{z}$ of $h^{\rm IG}$ starting from $z\in \overline{\bbH}$ as shown in \cite{MS16a, ig4}. Informally, $\eta_{\theta}^{z}$ is the solution to the ODE $(\eta_{\theta}^{z})'(t) = \exp(ih^{\rm IG}(\eta_{\theta}^{z}(t))/\chi +\theta)$.  When $z\in\bbR$ and the flow line is targeted at $\infty$,  as shown in~\cite[Theorem 1.1]{MS16a}, $\eta_{\theta}^{z}$ is an $\SLE_{\wt\kappa}(\underline\rho)$ process. One can also construct \emph{counterflowlines} of $h^{\rm IG}$, which are variants of $\SLE_\kappa$ processes (with $\kappa>4$).

  Let $h^{\rm IG}$ be the Dirichlet GFF on $\bbH$ with boundary value $-\lambda$ on $\bbR$. For $\kappa\in(4,8)$ and $z\in\bbH$, let $\eta_z^{L}$ and $\eta_z^{R}$ be the flow lines started at $z$ with angles $\frac{\pi}{2}$ and $-\frac{\pi}{2}$. Then $\eta_z^L$ and $\eta_z^R$ may hit and bounce-off each other. Let $\tau_z^R$ be the first time $\eta_z^R$ hits $\bbR$, and $\sigma_z^R$ be the last time before $\tau_z^R$ when $\eta_z^R$ hits $\eta_z^L$.  By~\cite[Theorem 1.7]{ig4}, $\eta_z^R$ can bounce off upon hitting $\bbR$ and  be continued to $\infty$. In fact, $\eta_z^R|_{[\tau_z^R,\infty)}$ is an $\SLE_{\wt\kappa}(\wt\kappa-4;-\frac{\wt\kappa}{2})$ in the connected component of $\bbH\backslash(\eta_z^L\cup\eta_z^R|_{[0,\tau_z^R]})$ containing $\eta_z^R(\tau_z^R)^-$ with force point $\eta_z^R(\tau_z^R)^-$ and $\eta_z^R(\sigma_z^R)$. The \emph{counterclockwise space-filling $\SLE_\kappa$ loop} $\eta'$ in $\bbH$ from $\infty$ to $\infty$ is defined in~\cite[Section 1.2.3]{ig4}, with the property that for any $z\in\ol\bbH$, the left and right boundaries of $\eta'$ stopped when first hitting $z$ are a.s.\ the flow lines $\eta_z^L$ and $\eta_z^R|_{[0,\tau_z^R]}$\footnote{The reason we stop at time $\tau_z^R$ is that $\eta_z^R$ hits $\bbR$ at height difference zero if we choose the orientation of $\bbR$ to be counterclockwise; see the text between Theorem 1.15 and Theorem 1.16 in~\cite{ig4} for more details.}. On the other hand, following~\cite[Theorem 3.1]{ig4}, the counterflowline $\eta^z$ of $h^{\rm IG}$ from $\infty$ to $z\in\bbH$ is a radial $\SLE_\kappa(\kappa-6)$ curve with force point at $\infty^-$ (i.e., $+\infty$). By ~\cite[Theorem 4.1]{ig4}, the left and right boundaries of $\eta^z$ (when lifted to a path in the universal cover of $\mathbb{C}\backslash\{z\}$) are precisely the flow lines $\eta_z^L$ and $\eta_z^R$, and the law of $\eta^z$ given $\eta_z^L$ and $\eta_z^R$ is $\SLE_\kappa(\frac{\kappa}{2}-4;\frac{\kappa}{2}-4)$ in each pocket between $\eta_z^L$ and $\eta_z^R$, which is boundary-filling. To summarize, we have the following:
  \begin{proposition}\label{prop:space-filling-loop-radial}
     Let $\kappa\in(4,8)$, $\tilde\kappa = \frac{16}{\kappa}$ and $z\in\bbH$. Consider a counterclockwise space-filling $\SLE_\kappa$ loop $\eta'$ in $\bbH$ from $\infty$ to $\infty$. Let $\eta_z^L$ and $\wt\eta_z^R$ be the left and right boundaries of $\eta'$ stopped when first hitting $z$. Let $\tau_z^R$ be the first time $\wt\eta_z^R$ hits $\bbR$, and $\sigma_z^R$ be the last time before $\tau_z^R$ when $\wt\eta_z^R$ hits $\eta_z^L$. Let $\eta_z^R$ be the concatenation of $\wt\eta_z^R$ with an independent $\SLE_{\wt\kappa}(\wt\kappa-4;-\frac{\wt\kappa}{2})$ in the connected component of $\bbH\backslash(\eta_z^L\cup\wt\eta_z^R|_{[0,\tau_z^R]})$ containing $\wt\eta_z^R(\tau_z^R)^-$ with force points $\wt\eta_z^R(\tau_z^R)^-$ and $\wt\eta_z^R(\sigma_z^R)$. Further draw an independent $\SLE_\kappa(\frac{\kappa}{2}-4;\frac{\kappa}{2}-4)$ curve $\eta_D$ in each  connected component $D$ of $\bbH\backslash(\eta_z^L\cup\eta_z^R)$ between $\eta_z^L$ and $\eta_z^R$ process starting from the last point on the component boundary traced by $\eta_z^L$ and targeted at the first. Then the concatenation of all the $\eta_D$'s (with $\eta_z^L$, $\eta_z^R$ as the boundaries) has the law radial $\SLE_\kappa(\kappa-6)$ from $\infty$ targeted at $z$ with force point at $\infty^-$.
  \end{proposition}

One can also construct the space-filling $\SLE_\kappa$ curve $\eta_0'$ in $\bbH$ from 0 to $\infty$ in a similar manner, where the boundary condition of the GFF is now $\lambda$ on $\bbR_-$ and $-\lambda$ on $\bbR_+$. For $x\in\bbR_+$, the law of the left boundary $\eta_x^L$ of $\eta_0'$  stopped at the time $\eta_x$ when hitting $x$ is now $\SLE_{\wt\kappa}(\frac{\wt\kappa}{2}-2,-\frac{\wt\kappa}{2};-\frac{\wt\kappa}{2})$ from $x$ to $\infty$ with force points $x^-,0;x^+$. Note that this curve merges into $\bbR_-$ upon hitting $\bbR_-$ at some point $y\in\bbR_-$ due to the continuation threshold. Moreover, the conditional law of $\eta_0'([0,\tau_x])$ given $\eta_x^L$ is the space-filling $\SLE_\kappa(\frac{\kappa}{2}-4;0)$ from 0 to $x$ in the domain $\eta_0'([0,\tau_x])$ from 0 to $x$ with force point at $y$.

\subsection{Conformal welding for quantum disks and quantum triangles}\label{subsec:qt-qd-welding}
For a measure $\mathcal{M}$ on the space of quantum surfaces (possibly with marked points) and a conformally invariant measure $\mathcal{P}$ on curves, we write $\mathcal{M}\otimes\mathcal{P}$ for the law of curve decorated quantum surface described by sampling $(S,\eta)$ from  $\mathcal{M}\times\mathcal{P}$ and then drawing $\eta$ on top of $S$. To be more precise, for a domain $\cD = (D,z_1,...,z_n)$ with marked points, suppose for $\phi$ sampled from  some measure $\mathcal{M}_{\cD}$, $(D,\phi,z_1,...,z_n)/{\sim_\gamma}$ has the law $\mathcal{M}$. Let $\mathcal{P}_{\cD}$ be the measure $\mathcal{P}$ on the domain $\cD$,  and assume that for any conformal map $f$ one has $\mathcal{P}_{f\circ\cD} = f\circ \mathcal{P}_{\cD}$, i.e., $\mathcal{P}$ is invariant under conformal maps. Then $\mathcal{M}\otimes\mathcal{P}$ is defined by  $(D,\phi,\eta,z_1,...,z_n)/{\sim_\gamma}$ for $\eta\sim\mathcal{P}_{\cD}$. This notion is well-defined for the quantum surfaces and SLE-type curves considered in this paper.

We begin with the conformal welding of two quantum disks.
\begin{theorem}[Theorem 2.2 of \cite{AHS20}]\label{thm:disk-welding}
    Let $\gamma\in(0,2),\wt\kappa=\gamma^2$ and $W_1,W_2>0$. Then there exists a constant $c:=c_{W_1,W_2}\in(0,\infty)$ such that
$$\Md_{0,2}(W_1+W_2)\otimes \SLE_{\wt\kappa}(W_1-2;W_2-2) = c\,\Wd(\Md_{0,2}(W_1),\Md_{0,2}(W_2)).  $$
\end{theorem}

\begin{figure}[ht]
	\centering
	\begin{tabular}{ccc} 
			\includegraphics[scale=0.7]{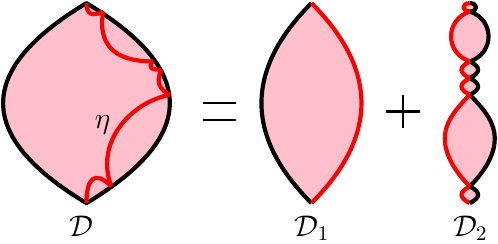}
			& & \ \ \ \
			\includegraphics[scale=0.5]{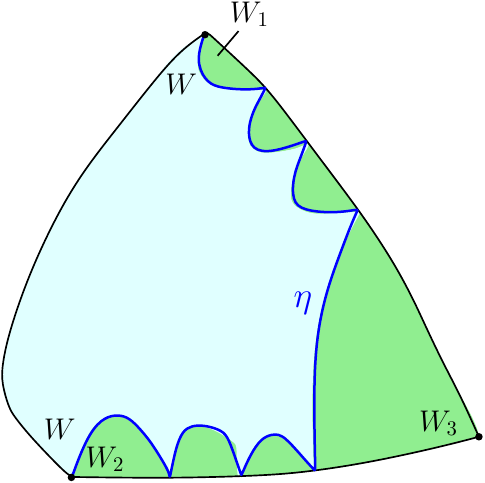}
		\end{tabular}
	\caption{\textbf{Left}: Illustration of Theorem~\ref{thm:disk-welding} with $W_1\ge\frac{\gamma^2}{2}$ and $W_2<\frac{\gamma^2}{2}$. \textbf{Right}: Illustration of Theorem~\ref{thm:disk+QT} with $W,W_3\ge\frac{\gamma^2}{2}$ and $W_1,W_2<\frac{\gamma^2}{2}$.}\label{fig-qt-weld} 
\end{figure}

Here, if $W_1+W_2<\frac{\gamma^2}{2}$, then $\Md_{0,2}(W_1+W_2)\otimes \SLE_{\wt\kappa}(W_1-2;W_2-2)$ is understood as drawing independent $ \SLE_{\wt\kappa}(W_1-2;W_2-2)$ curves in each bead of the weight $W_1+W_2$ disk, and the $\SLE_{\wt\kappa}(W_1-2;W_2-2)$ is defined by their concatenation.  To be more explicit, the concrete definition is given by replacing the measure $\Md_{0,2}(\gamma^2-W_1-W_2)$ with $\Md_{0,2}(\gamma^2-W_1-W_2)\otimes \SLE_{\wt\kappa}(W_1-2;W_2-2)$ in the Poisson point process construction of $\Md_{0,2}(W_1+W_2)$ in Definition~\ref{def-thin-disk}.

For a quantum triangle of weights $W+W_1,W+W_2,W_3$ with $W_2+W_3 = W_1+2$ embedded as $(D,\phi,a_1,a_2,a_3)$, we start by making sense of the $\SLE_{\wt\kappa}(W-2;W_1-2,W_2-W_1)$ curve $\eta$ from $a_2$ to $a_1$. If the domain $D$ is simply connected (which corresponds to the case where $W+W_1,W+W_2, W_3\ge\frac{\gamma^2}{2}$), $\eta$ is just the ordinary $\SLE_{\wt\kappa}(W-2;W_1-2,W_2-W_1)$  with force points at $a_2^-,a_2^+$ and $a_3$. Otherwise, let $(\tilde{D}, \phi,\tilde{a}_1,\tilde{a}_2, \tilde{a}_3)$ be the thick quantum triangle component as in Definition~\ref{def-qt-thin}, and sample an $\SLE_{\wt\kappa}(W-2;W_1-2,W_2-W_1)$ curve $\tilde{\eta}$ in $\tilde{D}$ from $\tilde{a}_2$ to $\tilde{a}_1$. Then our curve $\eta$ is the concatenation of $\tilde{\eta}$ with independent $\SLE_{\wt\kappa}(W-2;W_1-2)$ curves in each bead of the weight $W+W_1$   quantum disk (if $W+W_1<\frac{\gamma^2}{2}$) and $\SLE_{\wt\kappa}(W-2;W_2-2)$ curves in each bead of the weight $W+W_2$  quantum disk (if $W+W_2<\frac{\gamma^2}{2}$). In other words, if $W+W_1<\frac{\gamma^2}{2}$ and $W+W_2<\frac{\gamma^2}{2}$, then the notion $\QT(W_1,W_2,W_3)\otimes\SLE_{\wt\kappa}(W-2;W_1-2,W_2-W_1)$ is defined through $\QT(\gamma^2-W-W_1,\gamma^2-W-W_2,W_3)\otimes \SLE_{\wt\kappa}(W-2;W_1-2,W_2-W_1)$, $\Md_{0,2}(W+W_1)\otimes\SLE_{\wt\kappa}(W-2;W_1-2)$ and $\Md_{0,2}(W+W_2)\otimes\SLE_{\wt\kappa}(W-2;W_2-2)$ as in Definition~\ref{def-qt-thin}, while other cases follows similarly.

With this notation, we state the welding of quantum disks with quantum triangles below.
\begin{theorem}[Theorem 1.1 of \cite{ASY22}]\label{thm:disk+QT}
Let $\gamma\in(0,2)$ and $\wt\kappa=\gamma^2$. Fix $W,W_1,W_2,W_3>0$ such that $W_2+W_3=W_1+2$. There exists some constant $c:=c_{W,W_1,W_2,W_3}\in (0,
\infty)$ such that
\begin{equation}\label{eqn:disk+QT}
    \QT(W+W_1,W+W_2,W_3)\otimes \SLE_{\wt\kappa}(W-2;W_2-2,W_1-W_2) = c\Wd (\Md_{0,2}(W),\QT(W_1,W_2,W_3)).
\end{equation}
\end{theorem}

\begin{definition}[Weight zero quantum disks and quantum triangles]
We define the weight zero quantum disk to be a line segment  modulo homeomorphisms of $\bbR^2$ parametrized  by quantum length where the total length is $\mathbf {t}\sim \mathds{1}_{t>0}dt$, and write $\Md_{0,2}(0)$ for its law.     For $W_1,W_2,W_3\ge0$ where one or more of $W_1,W_2,W_3$ is zero, we define the measure $\QT(W_1,W_2,W_3)$ using $\Md_2(0)$ in the same way as Definition~\ref{def-qt-thin}.
\end{definition}

\begin{proposition}\label{prop:disk+QT-0}
    Theorem~\ref{thm:disk+QT} holds for $(W_1,W_2,W_3) = (0,\frac{\gamma^2}{2},2-\frac{\gamma^2}{2})$ and $\gamma\neq\sqrt 2$.
\end{proposition}

\begin{figure}
    \centering
    \includegraphics[scale=0.6]{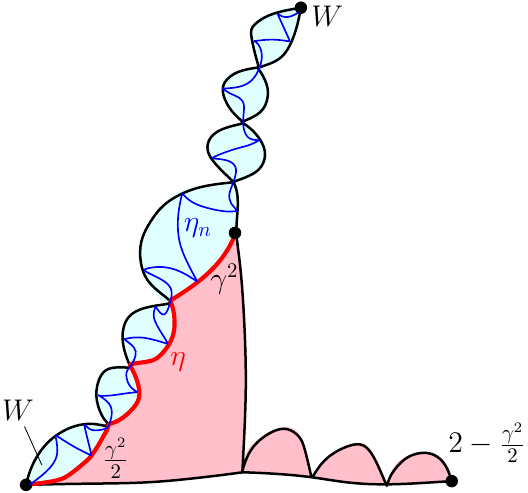}
    \caption{An illustration of Proposition~\ref{prop:disk+QT-0} with $\gamma\in(\sqrt{2},2)$ and $W<\frac{\gamma^2}{2}$.}
    \label{fig:weld-qt-0}
\end{figure}

We impose the constraint $\gamma \neq \sqrt2$ in Proposition~\ref{prop:disk+QT-0} to avoid technical difficulties, but expect that it also holds for $\gamma = \sqrt2$. This suffices for the present work since we only consider $\kappa \in (4,8)$, corresponding to $\gamma \in (\sqrt2, 2)$.

\begin{proof}
    We disintegrate over the quantum length $r$ of the boundary arc between the weight $\gamma^2$ vertex and the weight $\frac{\gamma^2}{2}$ vertex of the weight $(\gamma^2,\frac{\gamma^2}{2},2-\frac{\gamma^2}{2})$ quantum triangle, i.e., for $\ell>0$, we have
    \begin{equation*}
        \QT(0,\frac{\gamma^2}{2},2-\frac{\gamma^2}{2};\ell) = \int_0^\ell\,\Md_{0,2}(0;\ell-r)\times \QT(\gamma^2,\frac{\gamma^2}{2},2-\frac{\gamma^2}{2};r)\,dr
    \end{equation*}
    where $\Md_{0,2}(0;\ell-r)$ stands for a line segment with quantum length $\ell-r$. Now we mark the point on the weight $W$ quantum disk on the interface with distance $r$ to the root. By Lemma~\ref{lem:QT(W,2,W)}, for fixed $\ell>r$, the three-pointed quantum surface has law $c\QT(W,2,W;\ell-r,r)$. Therefore 
    \begin{equation}\label{eq:disk+QT-0-1}
    \begin{split}
        \Wd (\Md_{0,2}(W),\QT(0,\frac{\gamma^2}{2},2-\frac{\gamma^2}{2})) &= c\int_0^\infty\int_0^\ell \Wd\big(\QT(W,2,W;\ell-r,r),\QT(\gamma^2,\frac{\gamma^2}{2},2-\frac{\gamma^2}{2};r)\big)\, dr\,d\ell\\
        &= c\int_0^\infty \Wd\big(\QT(W,2,W;r), \QT(\gamma^2,\frac{\gamma^2}{2},2-\frac{\gamma^2}{2};r)\big)\,dr.
        \end{split}
    \end{equation}
    By~\cite[Corollary 4.11]{SY23}, since the Liouville field insertion size for the weight $2+\gamma^2$ is $\gamma+\frac{2-(2+\gamma^2)}{\gamma} = 0$, it follows that there exists a probability measure $\mathsf{m}_W$ on curves such that~\eqref{eq:disk+QT-0-1} is equal to a constant times $\QT(W,W+\frac{\gamma^2}{2},2-\frac{\gamma^2}{2})\otimes\mathsf{m}_W$.

 Now we identify the law $\mathsf{m}_W$. First assume $W<\frac{\gamma^2}{2}$. We condition on the event where the right boundary arc (i.e., the boundary arc between the weight $W$ vertex and the weight $2-\frac{\gamma^2}{2}$ vertex) has quantum length between $[1,2]$. This event has finite measure following~\cite[Proposition 2.24]{ASY22}. Let $n>W^{-1}$. In each bead of the weight $W$ quantum disk, we draw an $\SLE_{\wt\kappa}(W-2-\frac{1}{n};\frac{1}{n}-2)$ curve and let $\eta_n$ be the concatenation. By Theorem~\ref{thm:disk-welding}, $\eta_n$ cuts the weight $W$ quantum disk into a weight $W-\frac{1}{n}$ quantum disk and a weight $\frac{1}{n}$ quantum disk, and therefore $(\eta_n,\eta)$ are the interfaces under the welding $$\iint_{\bbR_+^2}\Wd\bigg(\Md_{0,2}(W-\frac{1}{n};\ell_1),\Md_{0,2}(\frac{1}{n};\ell_1,\ell_2),\QT(0,\frac{\gamma^2}{2},2-\frac{\gamma^2}{2};\ell_2)\bigg)\,d\ell_1d\ell_2.$$ By the previous paragraph, we can first weld the weight $\frac{1}{n}$ quantum disk with the quantum triangle to get a weight $(\frac{1}{n}, \frac{1}{n}+\frac{\gamma^2}{2},2-\frac{\gamma^2}{2})$ quantum triangle, and therefore it follows from Theorem~\ref{thm:disk+QT} that the marginal law of $\eta_n$ is now $\SLE_{\wt\kappa}(W-\frac{1}{n}-2;\frac{\gamma^2}{2}+\frac{1}{n}-2,-\frac{\gamma^2}{2})$. On the other hand, by Lemma~\ref{lm:SLE-rho--2} below, if we embed $\QT(W,W+\frac{\gamma^2}{2},2-\frac{\gamma^2}{2})\otimes\mathsf{m}_W$ on a compact domain, the Hausdorff distance between $\eta_n$ and $\eta$ converges in probability as $n\to\infty$. Therefore using the continuity of the Loewner chains (see e.g.~\cite[Section 6.1]{Kem17SLE}), we conclude that the law $\mathsf{m}_W$ equals $\SLE_{\wt\kappa}(W-2;\frac{\gamma^2}{2}-2,-\frac{\gamma^2}{2})$ if $W<\frac{\gamma^2}{2}$.
 
 Finally if $W\ge\frac{\gamma^2}{2}$, consider the welding $$\iint_{\bbR_+^2}\Wd\bigg(\Md_{0,2}(W-\frac{\gamma^2}{4};\ell_1),\Md_{0,2}(\frac{\gamma^2}{4};\ell_1,\ell_2),\QT(0,\frac{\gamma^2}{2},2-\frac{\gamma^2}{2};\ell_2)\bigg)\,d\ell_1d\ell_2,$$
 and let $(\eta_0,\eta)$ be the interfaces. Then using Theorem~\ref{thm:disk-welding}, the law of $\eta$ is the desired measure $\mathsf{m}_W$; on the other hand, from the previous paragraph, we know that  $\eta_0$ is an $\SLE_{\wt\kappa}(W-\frac{\gamma^2}{4}-2;\frac{3\gamma^2}{2}-2,-\frac{\gamma^2}{2})$ curve by Theorem~\ref{thm:disk+QT}, whereas $\eta$ is an $\SLE_{\wt\kappa}(\frac{\gamma^2}{4}-2;\frac{\gamma^2}{2}-2,-\frac{\gamma^2}{2})$ curve to the right of $\eta_0$. Therefore from the imaginary geometry theory~\cite[Theorem 1.1]{MS16a} we can read off the marginal law of $\eta$ under this setting, which gives the desired conclusion.
  \end{proof}

\begin{lemma}\label{lm:SLE-rho--2}
   Let $\wt\kappa\in(0,4)$. Let $(D,x,y)$ be a bounded simply connected domain and $x,y\in\partial D$. Let $\rho>-2$, and let $\eta_n$ be an $\SLE_{\wt\kappa}(\rho-\frac{1}{n};\frac{1}{n}-2)$ curve in $\ol D$ from $x$ to $y$ with force points $x^\mp$. Then as $n\to\infty$, the Hausdorff distance between $\eta_n$ and the right boundary arc of $(D,x,y)$ converges to 0 in probability.
\end{lemma}
\begin{proof}

    First assume that $\rho>\frac{\kappa}{2}-2$. Consider the imaginary geometry field $h$ on $\bbH$ whose boundary values are given by $-\lambda(1+\rho)$ on $(-\infty,0)$ and $-\lambda$ on $(0,\infty)$. Let $\tilde\eta_n$ be the angle $\frac{\lambda}{n\chi}$ flow line of $h$. Then following~\cite[Theorem 1.1]{MS16a}, $\tilde\eta_n$ is an $\SLE_{\wt\kappa}(\rho-\frac{1}{n};\frac{1}{n}-2)$ curve in $\bbH$ from $0$ to $\infty$ with force points $0^\mp$. 
    Moreover, following the monotonicity of flow lines~\cite[Theorem 1.5]{MS16a}, for $m>n$, $\tilde\eta_m$ stays to the right side of $\tilde\eta_n$.  For sufficiently large enough $n$, $\tilde\eta_n\cap(-\infty,0)=\emptyset$. Let $\tilde D_n$ be the connected component of $\bbH\backslash\tilde\eta_n$ with $-1$ on the boundary, and $\psi_n:\tilde D_n\to\bbH$ be the conformal map fixing $0,-1,\infty$. Then $\tilde D_n$ is increasing in $n$, and let $\tilde D$ be the limit.  On the other hand, following~\cite[Theorem 1.1]{AHS21}, $\psi_n'(-1)$ tends to 1 in probability.  Moreover, using Schwartz reflection over $(-\infty,0)$, by the Carath\'{e}odory kernel theorem, $\psi_n^{-1}$ converges uniformly on compact subsets of $\bbH\cap\bbR$ to $\psi^{-1}$, where $\psi$ is the conformal map from $\tilde D$ to $\bbH$ fixing $0,\infty,-1$. Therefore $\psi$ can be viewed as a conformal map from $\tilde D\cup\ol{\tilde D}\cup (-\infty,0)$ to $\mathbb{C}\backslash [0,\infty)$ fixing $-1$ with $\psi'(-1)=1$, which by Schwartz's lemma implies that $\tilde D = \bbH$. Therefore the conclusion follows by taking a conformal map $f$ from $\bbH$ to $D$. Indeed, assume on the contrary and there exists $\e_0>0$ such that for any $n$, there exists some point $z_n$ lying on the right hand side of $f(\tilde\eta_n)$ and stays at least $\e_0$ distance away from the right boundary arc of $\partial D$. Using monotonicity of $f(\tilde D_n)$ and the boundedness of $D$, one can find some point $z$ which lies on $\ol D\backslash f(\tilde D)$ and stays at least $\e_0$ distance away from the right boundary arc of $\partial D$. Then this would contradict with $\tilde D = \bbH$. 

    Finally if $\rho\in(-2,\frac{\kappa}{2}-2]$, consider the imaginary geometry field $h$ on $\bbH$ whose boundary value is $-\lambda$ on $\bbR$. Let $\tilde\eta$ be the angle $\frac{\lambda(2+\rho)}{\chi}$ flow line of $h$, and $\tilde\eta_n$ be the angle $\frac{\lambda}{n\chi}$ flow line $\tilde\eta_n$ of $h$. Then by the previous paragraph, for any conformal map $\varphi$ from $\bbH$ to a bounded simply connected domain, $\varphi(\tilde\eta_n)$ converges to $\varphi((0,\infty))$ in Hausdorff topology. Moreover, following~\cite[Theorem 1.1]{MS16a}, the conditional law of $\tilde\eta_n$ given $\tilde\eta$ is $\SLE_{\wt\kappa}(\rho-\frac{1}{n};\frac{1}{n}-2)$ in each connected component of $\bbH\backslash\tilde\eta$ to the right of $\tilde\eta$. Therefore we conclude the proof by conditioning on $\tilde\eta$, pick a connected component of $\bbH\backslash\tilde\eta$ to the right of $\tilde\eta$ and conformally map to $D$.
\end{proof}

We end this section with the following result on conformal welding of forested line segments.
\begin{proposition}[Proposition 3.25 of~\cite{AHSY23}]\label{prop:weld:segment}
     Let $\kappa\in(4,8)$ and $\gamma = \frac{4}{\sqrt{\kappa}}$. Consider a quantum disk $\mathcal{D}$ of weight $W=2-\frac{\gamma^2}{2}$, and let $\tilde\eta$ be the concatenation of an independent $\SLE_\kappa(\frac{\kappa}{2}-4;\frac{\kappa}{2}-4)$ curve on each bead of $\cD$. Then for some constant $c$, $\tilde\eta$
 divides $\cD$ into two   forested lines segments $\wt\cL_-,\wt\cL_+$, whose law is
 \begin{equation}\label{eq:weld:segment}
      c\int_0^\infty \mathcal{M}_2^\mathrm{f.l.}(\ell)\times\mathcal{M}_2^\mathrm{f.l.}(\ell)d\ell.
 \end{equation}
{Moreover, $\wt\cL_\pm$ a.s.\ uniquely determine $(\cD,\tilde\eta)$ in the sense that $(\cD,\tilde\eta)$ is measurable with respect to the $\sigma$-algebra generated by $\wt\cL_\pm$.}
\end{proposition}

\subsection{Mating-of-trees descriptions of quantum surfaces}\label{subsec:mot-qs}
Mating-of-trees theorems allow us to identify special SLE-decorated LQG surfaces with 2D Brownian motion trajectories. In Section~\ref{subsec-mot-meas} 
we discuss the map sending Brownian trajectories to SLE-decorated LQG surfaces. In Section~\ref{subsec-mot-triangle} we use the Markov property of Brownian motion to obtain a new mating-of-trees theorem for $\QT(2-\frac{\gamma^2}2, \gamma^2, \frac{\gamma^2}2)$ (Proposition~\ref{prop-mot-triangle-down}), and in Section~\ref{subsec-mot-welding} we use the Markov property in a different way to obtain a conformal welding identity (Proposition~\ref{prop-mot-QD11}) which will be used to prove Theorem~\ref{thm:welding} in Section~\ref{subsec:pf-welding}.

Let $\mathbbm a^2 = 2/\sin(\frac{\pi\gamma^2}4)$ be the mating-of-trees variance, as derived in \cite[Theorem 1.3]{ARS21}. Consider Brownian motion $Z:=(L_t, R_t)_{t \geq 0}$ with 
\eqb\label{eq-cov}
\mathrm{Var}(L_t) = \mathrm{Var}(R_t) = t \quad \text{and} \quad  \mathrm{Cov}(L_t, R_t) = -\cos(\frac{\pi\gamma^2}4) \mathbbm a^2 t \qquad \text{ for }t \geq 0.
\eqe
We will introduce versions of the process $Z$ taking values in the positive quadrant $\bbR_+^2 = (0,\infty)^2$; as we will see, these variants will correspond to special quantum disks and triangles. 

Let $\mu^\gamma(t,z)$ denote the law of Brownian motion with covariance~\eqref{eq-cov} started at $z \in \bbC$ and run for time $t > 0$. Let $\mu^\gamma(t; z,w)$ be the disintegration of $\mu^\gamma(t; z)$ over its endpoint, so each measure $\mu^\gamma(t; z,w)$ is supported on the set of paths from $z$ to $w$, and $\mu^\gamma(t; z) = \int_{\bbC} \mu^\gamma(t; z,w) \, dw$. Note that $|\mu^\gamma(t; z)| = 1$ for all $t,z$, but $|\mu^\gamma(t;z,w)|$ is typically not 1 (rather, it is the probability density function for the endpoint of a sample from $\mu^\gamma(t; z)$).
The Markov property of Brownian motion can then be written as 
\eqb\label{eq-markov-bulk}
\mu^\gamma(t_1 + t_2; z_1, z_2) = \int_{\bbC} \mu^\gamma(t_1; z_1, w) \times \mu^\gamma(t_2; w, z_2)\, dw,
\eqe
meaning a sample from the left hand side can be obtained by concatenating the pair of paths sampled from the right hand side. 

We can define Brownian motion from $z$ to $w$ without fixing its duration via $\mu^\gamma(z, w) := \int_0^\infty \mu^\gamma(t; z,w) \, dt$. This measure is scaling-invariant: for $\lambda > 0$, the law of a sample from $\mu^\gamma(z,w)$ after Brownian rescaling by a factor of $\lambda$ is $\mu^\gamma(\lambda z, \lambda w)$. (The standard Brownian bridge, having null covariance, has the stronger property of \emph{conformal} invariance.) For a planar domain $D \subset \bbC$ and distinct points $z, w \in D$, let $\mu^\gamma_D(z;w)$ be the restriction of $\mu^\gamma(z,w)$ to paths staying in $D$.

We now discuss Brownian motions which start or end on the boundary of certain domains. Note that after defining $\mu^\gamma_D(z,w)$ for $z, w \in \ol D$, we can disintegrate by duration to obtain $\mu^\gamma_D(t; z,w)$ for each $t>0$. 

\noindent \textbf{Bulk to boundary in $\bbH$:} 
For $z \in \bbH$, let $\mu_{\bbH,\mathrm{exit}}^\gamma (z)$ be the law of Brownian motion started at $z$ and run until it hits $\bbR$; this is a probability measure. Let $\{\mu^\gamma_{\bbH}(z, x)\}_{x \in \bbR}$ be the disintegration of $\mu_{\bbH ,\mathrm{exit}}^\gamma (z)$ over the endpoint $x$ of the trajectory: $\mu_{\bbH, \mathrm{exit}}^\gamma (z) = \int_{\bbR} \mu^\gamma_{\bbH}(z, x) \, dx$. 

\noindent \textbf{Boundary to bulk in $\bbH$:} For $z \in \bbH$ and $x \in \bbR$, define 
$\mu^\gamma_{\bbH}(x, z) = \lim_{\eps \to 0} C \eps^{-1} \mu^\gamma_{\bbH}(x + \eps i , z)$ where $C>0$ is a constant. We can choose $C$ such that $\mu^\gamma_{\bbH} (z, x) = \mathrm{Rev}_* \mu^\gamma_{\bbH}(x,z)$, where $\mathrm{Rev}$ is the function which sends a curve to its time-reversal. See~\cite[Section 3.2.3]{LW04} for details on the limiting definition and time-reversal property. 

\noindent \textbf{Bulk to boundary in $\bbR_+^2$:}
For $z \in \bbR_+^2$ and nonzero $x \in \partial (\bbR_+^2)$ we can define $\mu^\gamma_{\bbR_+^2}(z,x)$ by disintegration and continuity. Equivalently, if $x \in \bbR_+$ then $\mu^\gamma_{\bbR_+^2}(z,x)$ is the restriction of $\mu^\gamma_{\bbH}(z,x)$ to paths lying in $\ol{\bbR_+^2}$; a similar statement holds for $x \in \{0\} \times \bbR_+$. On the other hand, for the atypical boundary point $x = 0$ one must take a limit:  $
 \mu^\gamma_{{\bbR_+^2}} (z, 0) = \lim_{\eps \to 0} \eps^{-4/\gamma^2} \mu^\gamma_{{\bbR_+^2}}(z, \eps e^{i\pi/4})$. 

 \noindent \textbf{Boundary to origin in $\bbR_+^2$:} For $x \in \partial (\bbR_+^2)$ and $t>0$, let the law $\mu^\gamma_{\bbR_+^2} (t;x, 0) = \int_{\bbR_+^2} \mu^\gamma_{\bbR_+^2}(\frac t2, x, z) \mu^\gamma_{\bbR_+^2}(\frac t2, z, 0) \, dz$, meaning that a sample from $\mu^\gamma_{\bbR_+^2} (t;x, 0)$ is defined to be the concatenation of a pair of paths sampled from the right hand side. We set $\mu^\gamma_{\bbR_+^2}(x,0) = \int_0^\infty\mu^\gamma_{\bbR_+^2}(t;x,0)\,dt$.

These measures all satisfy Markov properties inherited from~\eqref{eq-markov-bulk}. The limiting definition of $\mu_{\bbR_+^2}^\gamma(z,0)$ can be seen to make sense by \cite{shimura-cone}, see~\cite[Section 4.1]{AG19} or~\cite[Section 7]{AHS20} for details.

\subsubsection{Obtaining an SLE-decorated quantum surface from Brownian motion}\label{subsec-mot-meas}
In this section we explain that certain Brownian motion/excursion trajectories can be identified with  SLE-decorated quantum surfaces via a map we denote by $F$. We introduce $F$ in the setting of the original  mating-of-trees theorem of \cite{DMS14}, and will later use $F$ in other settings. 

If $(\bbC, \phi, 0, \infty)/{\sim}$ is an embedding of a \emph{$\gamma$-quantum cone}, and $\eta$ is an independent \emph{space-filling $\SLE_\kappa$ curve in $\bbC$ from $\infty$ to $\infty$} parametrized by LQG area, then one can define a \emph{boundary length process} $(L_t, R_t)_{(-\infty,\infty)}$ keeping track of the changes in the left and right quantum boundary lengths of $\eta([t, \infty))$ as $t$ varies. \cite[Theorem 1.9]{DMS14} shows that this process is two-sided Brownian motion: the covariance of $(L_t, R_t)_{[0,\infty)}$ is~\eqref{eq-cov}, and $(L_{-t}, R_{-t})_{[0,\infty)} \stackrel d= (L_t, R_t)_{[0,\infty)}$. Moreover, $(\bbC, \phi, \eta, 0, \infty)/{\sim}$ is measurable with respect to $(L_t, R_t)_{(-\infty, \infty)}$. 

 For each $a>0$, let $x_{L,a}$ be the point on the left boundary arc of $\eta([0,a])$ furthest from $0$ such that the clockwise boundary arc from $0$ to $x_{\ell,a}$ is a subset of the left boundary of $\eta([0,\infty))$, and similarly define $x_{r,a}$. 
 \cite[Section 2.4]{AY23} explains that the curve-decorated quantum surface $\cC = (\eta([0,a]), \eta|_{[0,a]}, x_{\ell, a}, x_{r, a})/{\sim}$  is measurable with respect to $(L_\cdot, R_\cdot)_{[0,a]}$; let $F$ be the map such that $F((L_\cdot, R_\cdot)_{[0,a]}) = \cC$ a.s.. 
The map $F$ satisfies two key properties which we now state. Let $\partial_\ell^- \cC$ and $\partial_\ell^+ \cC$ denote the successive clockwise boundary arcs of $\cC$ from $0$ to $x_{\ell, a}$ to $\eta(a)$, and let $\partial_r^- \cC$ and $\partial_r^+ \cC$ be the successive counterclockwise boundary arcs from $0$ to $x_{r, a}$ to $\eta(a)$.
 \smallskip 
 
 \noindent 
 \textbf{Reversibility:} Setting $(\wt L_t, \wt R_t)_{[0,a]} = (R_{a - t} - R_a, L_{a - t} - L_a)_{[0,a]})$ and $\wt \eta_a := \eta(a - \cdot)|_{[0,a]}$, we have $F((\wt L_\cdot, \wt R_{\cdot})_{[0,a]}) = (\eta([0,a]), \wt \eta_a, x_{r,a}, x_{\ell,a})/{\sim}$ a.s. \cite[Lemma 2.14]{AY23}. 
\smallskip 

 \noindent 
 \textbf{Concatenation compatibility:} Let $a_1, a_2>0$. 
 For the SLE-decorated quantum surface  $\cC = 
 F((L_t, R_t)_{[0,a_1+a_2]})$, a.s.\ the quantum surfaces $\cC_1$ and $\cC_2$ obtained by restricting to the domains parametrized by its curve on time intervals $[0,a_1]$ and $[a_1, a_1+a_2]$ satisfy $\cC_1 = F((L_t, R_t)_{[0,a_1]})$ and $\cC_2 = F((L_{t+a_1}, R_{t+a_1})_{[0,a_2]})$ \cite[Lemma 2.15]{AY23}. Moreover, $\cC$ can be recovered from $\cC_1$ and $\cC_2$ by identifying the endpoint of the curve of $\cC_1$ with the starting point of the curve of $\cC_2$, conformally welding 
 $\partial_\ell^+ \cC_1$ to $\partial_\ell^- \cC_2$ such that the entirety of the shorter boundary arc is welded to the corresponding segment of the longer boundary arc, and likewise conformally welding $\partial_r^+ \cC_1$ to $\partial_r^- \cC_2$.
 
 Finally, while $F$ is a priori only defined for Brownian motion trajectories, it can be extended to Brownian excursion trajectories by conformal welding. For instance, if $Z$ is a sample from $\mu_{\bbH}(1 ; z, 0)$ and $(t_n)_{n \geq 0}$ is a deterministic increasing sequence with $t_0 = 0$ and $\lim_{n \to \infty} t_n = 1$, then $F(Z)$ can be defined as the conformal welding of $F(Z|_{[t_{n}, t_{n+1}]})$ for $n = 0, 1, \dots$; by concatenation compatibility, the resulting $F(Z)$ does not depend on the choice of $(t_n)_{n \geq 0}$.

\subsubsection{Mating of trees for the quantum triangle with weight $(2-\frac{\gamma^2}2, \gamma^2, \frac{\gamma^2}2)$}\label{subsec-mot-triangle}

\begin{figure}
    \centering
    \includegraphics[scale=0.37]{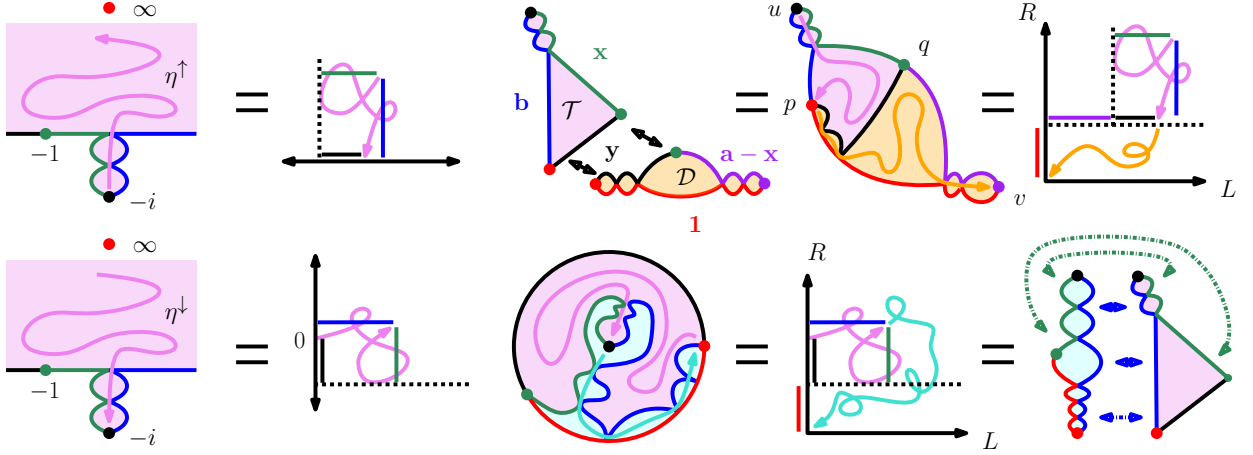}
    \caption{ 
    \textbf{Top left:} The left hand side is an embedding of a sample from $\QT^\uparrow(2-\frac{\gamma^2}2, \gamma^2, \frac{\gamma^2}2)$ with the vertices of weights $2-\frac{\gamma^2}2, \gamma^2, \frac{\gamma^2}2$ colored black, green, red respectively, the right hand side is its boundary length process. Colored boundary arcs (left hand side) have their lengths depicted by the same color (right hand side). The boundary length law 
    of $\QT^\uparrow(2-\frac{\gamma^2}2, \gamma^2, \frac{\gamma^2}2)(b)$ for $b>0$ is identified as Brownian motion in Lemma~\ref{lem-mot-triangle}.
    \textbf{Bottom left:} An embedding of a sample from $\QT^\downarrow(2-\frac{\gamma^2}2, \gamma^2, \frac{\gamma^2}2)$ and its boundary length process. For $\QT^\downarrow(2-\frac{\gamma^2}2, \gamma^2, \frac{\gamma^2}2)(b)$ the boundary length process is identified in Proposition~\ref{prop-mot-triangle-down}. 
    \textbf{Top right:} Diagram for the proof of Lemma~\ref{lem-mot-triangle}. The pink region corresponds to $\QT^\uparrow(2-\frac{\gamma^2}2, \gamma^2, \frac{\gamma^2}2)$ and the orange region corresponds to $\Md_{0,2}(2-\frac{\gamma^2}2)$. \textbf{Bottom right:} Diagram for the proof of Proposition~\ref{prop-mot-QD11}. The pink region corresponds to $\QT^\downarrow(2-\frac{\gamma^2}2, \gamma^2, \frac{\gamma^2}2)$ and the blue region corresponds to $\QT(2-\frac{\gamma^2}2, 2-\frac{\gamma^2}2, 2)$.
    }
    \label{fig:mot-triangle}
\end{figure}

Embed a sample from $\QT(2-\frac{\gamma^2}2, \gamma^2, \frac{\gamma^2}2)$ as $(D, \phi, -i, -1, \infty)$ such that the points $-i, -1, \infty$ correspond to the vertices having weights $2-\frac{\gamma^2}2, \gamma^2, \frac{\gamma^2}2$ respectively, the connected component of $D$ having $-1$ on its boundary is $\bbH$, and $\ol{D \backslash \bbH} \cap \bbR = 0$. See Figure~\ref{fig:mot-triangle} (left). Independently sample a space-filling $\SLE_{\kappa}$ curve in $D \backslash \bbH$ from $-i$ to $0$ and a space-filling $\SLE_\kappa(\frac{\kappa}{2}-4;0)$ curve in $\bbH$ from $0$ to $\infty$ with force point at $-1$. Let $\eta^\uparrow$ be the concatenation of these two curves parametrized by quantum area, so it is a space-filling curve in $D$ from $-i$ to $\infty$, and let $\eta^\downarrow$ be the time-reversal of $\eta^\uparrow$. Let $\QT^{\uparrow/\downarrow}(2-\frac{\gamma^2}2, \gamma^2, \frac{\gamma^2}2)$ be the law of $(D, \phi,\eta^{\uparrow/\downarrow}, -i, -1, \infty)/{\sim}$. 

Now we define the boundary length process $(L_t, R_t)$ associated to a sample from $\QT^\uparrow(2-\frac{\gamma^2}2, \gamma^2, \frac{\gamma^2}2)$. 
Let $T$ be the duration of $\eta^\uparrow$ and let $\tau$ be the first time $\eta^\uparrow$ hits $-1$. For  $t \leq T$ let $R_t$ be the  quantum length of the right boundary arc of $\eta^\uparrow([t, T])$. For $t \leq \tau$, consider the left boundary arc of $\eta^\uparrow([0,t])$; let $L_t^+$ (resp.\ $L_t^-$) be the quantum length of the segment inside $D$ (resp.\ on $\partial D$), and let $L_t = L_t^+ - L_t^-$. For $t \in (\tau, T]$ let $L_t$ be $L_\tau$ plus the quantum length of the  left  boundary arc of $\eta^\uparrow([\tau, t])$. 

We likewise define the boundary length process $(L_t, R_t)$ of  a sample from $\QT^\downarrow(2-\frac{\gamma^2}2, \gamma^2, \frac{\gamma^2}2)$: Let $T$ be the duration of $\eta^\downarrow$ and let $\tau$ be the first time $\eta^\downarrow$ hits $-1$. For $t \leq T$ let $L_t$ be the  quantum length of the left boundary arc of $\eta^\downarrow([0,t])$. For $t \leq \tau$, consider the right boundary arc of $\eta^\downarrow([0,t])$; let $R_t^+$ (resp.\ $R_t^-$) be the quantum length of the segment inside $D$ (resp.\ on $\partial D$), and let $R_t = R_t^+ - R_t^-$. For $t \in (\tau, T]$ let $R_t$ be $R_\tau$ plus the quantum length of the  right boundary arc of $\eta^\downarrow([\tau,t])$. 

By definition, for a sample from $\QT^\uparrow(2-\frac{\gamma^2}2, \gamma^2, \frac{\gamma^2}2)$ (resp.\ $\QT^\downarrow(2-\frac{\gamma^2}2, \gamma^2, \frac{\gamma^2}2)$), the quantum lengths of the boundary arcs clockwise from the weight $2-\frac{\gamma^2}2$ vertex are given by $(-\inf_{t \leq T} L_t, L_T -\inf_{t \leq T} L_t, R_0)$ (resp.\ $(R_T - \inf_{t \leq T} R_t, -\inf_{t \leq T} R_t, L_t)$).

The result we need from this section is the following mating-of-trees result for $\QT^\downarrow(2-\frac{\gamma^2}2, \gamma^2, \frac{\gamma^2}2)$. Let $\QT^{\uparrow/\downarrow} (2-\frac{\gamma^2}2, \gamma^2, \frac{\gamma^2}2)(b)$ be the disintegration of $\QT^{\uparrow/\downarrow}(2-\frac{\gamma^2}2, \gamma^2, \frac{\gamma^2}2)$  according to the quantum length of the boundary arc between the vertices of weights $2-\frac{\gamma^2}2$ and $\frac{\gamma^2}2$.  (For $\QT^\downarrow(2-\frac{\gamma^2}2, \gamma^2, \frac{\gamma^2}2)$ this length equals $L_T$.)

\begin{proposition}\label{prop-mot-triangle-down}
Assume $\gamma \neq \sqrt2$. There is a constant $C$ such that for all $b>0$,
    the boundary length process of a sample from $\QT^\downarrow(2-\frac{\gamma^2}2, \gamma^2, \frac{\gamma^2}2)(b)$ has law $C \int_{\bbR}\mu^\gamma_{\bbR_+ \times \bbR}(0, b + ci) \, dc$. {Moreover, the map $F$ from Section~\ref{subsec-mot-meas} a.s.\ recovers the decorated quantum surface from its boundary length process. }
\end{proposition}

In order to prove Proposition~\ref{prop-mot-triangle-down} we will need the following mating-of-trees result for $\Md_{0,2}(2-\frac{\gamma^2}2)$. Let $\Md_{0,2}(2-\frac{\gamma^2}2; \ell, r) \otimes \SLE_\kappa^\mathrm{sf}$ denote the law of a sample from $\Md_{0,2}(2-\frac{\gamma^2}2; \ell, r)$ decorated by an independent space-filling $\SLE_\kappa$ curve between its two marked points. 

\begin{proposition}\label{prop-thin-disk-mot}
    There is a constant $C$ such that the following is true for all $\ell, r$. Sample from  $\Md_{0,2}(2-\frac{\gamma^2}2; \ell, r) \otimes \SLE_{16/\gamma^2}^\mathrm{sf}$, and parametrize the SLE curve $\eta$ by quantum area covered (so the total duration $T$ equals the total quantum area). For $t \leq T$ let $L_t$ and $R_t$ denote the left and right boundary lengths of $\eta([t,T])$. Then the law of $(L_t, R_t)_{[0,T]}$ is $C \mu_{{\bbR_+^2}} (\ell + r i, 0)$. {Moreover, the map $F$ from Section~\ref{subsec-mot-meas} a.s.\ recovers the decorated quantum surface from its boundary length process. }
\end{proposition}
\begin{proof}
    For the case where $r = 1$, this is stated as \cite[Proposition 7.3]{AHS20}. The general $r$ case follows from the $r=1$ case by rescaling, since for any $\lambda>0$
    \[ \frac{|\Md(2-\frac{\gamma^2}2; \lambda \ell, \lambda r)|}{|\Md(2-\frac{\gamma^2}2;  \ell, r)|} = \lambda^{- \frac4{\gamma^2}} = 
    \frac{\lim_{\eps \to 0} (\lambda \eps)^{-4\gamma^2} |\mu_{{\bbR_+^2}}(\lambda \ell + \lambda r i, \lambda \eps e^{i\pi/4})|}{\lim_{\eps \to 0} \eps^{-4\gamma^2}|\mu_{{\bbR_+^2}}(\ell + r i, \eps e^{i\pi/4})|}
    = \frac{|\mu_{{\bbR_+^2}}(\lambda \ell + \lambda r i, 0)|}{|\mu_{{\bbR_+^2}}(\ell + r i, 0)|}.\]
    The first equality follows from \cite[Lemma 2.24]{AHS20} with $W = 2-\frac{\gamma^2}2$, the second from the scaling-invariance of $\mu_{\bbR_+^2}(z, w)$, and the third from the limiting definition of $\mu^\gamma_{\bbR_+^2}(z, 0)$.
\end{proof}

\begin{lemma}\label{lem-mot-triangle}
Assume $\gamma \neq \sqrt2$. 
    There is a constant $C$ such that for any $b>0$, the boundary length process of a sample from $\QT^\uparrow(2-\frac{\gamma^2}2, \gamma^2, \frac{\gamma^2}2)(b)$ has law $C \mu^\gamma_{\bbH, \mathrm{exit}}(bi)$. {Moreover, the map $F$ from Section~\ref{subsec-mot-meas} a.s.\ recovers the decorated quantum surface from its boundary length process. }
\end{lemma}

\begin{proof}
Let $a>0$. 
Consider a sample $\cD$ from  $\Md_{0,2}(2-\frac{\gamma^2}2; a, b+1)\otimes \SLE_{16/\gamma^2}^\mathrm{sf}$; let $u$ and $v$ be the starting and ending points of the space-filling SLE  $\eta$. Let $p$ be the point on the right boundary arc of $\cD$ at quantum length $b$ from $u$. See Figure~\ref{fig:mot-triangle} (top right).  The law of $\cD$ further marked by point $p$ is $\int_0^\infty \QT(2-\frac{\gamma^2}2, 2-\frac{\gamma^2}2, 2; a, b, 1)\otimes \SLE_{16/\gamma^2}^\mathrm{sf} \, da$. 

Let $\eta_1$ be $\eta$ run until the time it hits the point $p$, and let $\eta_2$ be $\eta$ starting from the time it hits $p$. Let $q$ be the last point on the left boundary of $\cD$ hit by $\eta_1$. 
Let $\cT$ be the quantum surface parametrized by the trace of $\wt \eta$, decorated by $\eta_1$ and points $u, p, q$.  Let $\cD'$ be the quantum surface parametrized by the trace of $\eta_2$, decorated by $\eta_2$ and points $p, v$. Let $\cD_0$ be the connected component of $\cD$ containing $p$, and $u',v'$ be the starting and ending points of $\eta|_{\cD_0}$.  Then as explained in Section~\ref{subsec:ig}, the interface between $\eta_1$ and $\eta_2$ is a chordal $\SLE_{\gamma^2}(\frac{\gamma^2}{2}-2,-\frac{\gamma^2}{2};-\frac{\gamma^2}{2})$ curve in $\cD_0$ from $p$ to $v'$ with force points $p^-,u';p^+$, respectively. Therefore, by 
Proposition~\ref{prop:disk+QT-0}, the law of $(\cT, \cD')$ is
\eqb\label{eq-triangle-mot-proof}
C \int_0^\infty \int_0^a \QT^\uparrow(2-\frac{\gamma^2}2, \gamma^2, \frac{\gamma^2}2)(x, y, b) \times \left(\Md_{0,2}(2-\frac{\gamma^2}2; a-x +y ; 1)\otimes \SLE_{16/\gamma^2}^\mathrm{sf}\right) \,dx \, dy.
\eqe
Here $\QT^\uparrow(2-\frac{\gamma^2}2, \gamma^2, \frac{\gamma^2}2)(x, y, b)$ is the disintegration of $\QT^\uparrow(2-\frac{\gamma^2}2, \gamma^2, \frac{\gamma^2}2)$ where $x,y,b$ are the quantum lengths of the boundary arcs in clockwise order from the weight $2-\frac{\gamma^2}2$ vertex.

By Proposition~\ref{prop-thin-disk-mot}, the boundary length process $Z$ of $\cD$ has law $C' \mu_{\bbR_+^2}(a + (b+1)i, 0)$ for some $C'$. Let $T$ be the random duration of $Z$, let $\tau$ be the time $Z$ hits $\{ \mathrm{Im} (z) = 1\}$, and let $(Z^1(t))_{[0,\tau]} = (Z(t) - a - i)_{[0,\tau]}$ and $(Z_2(t))_{[0,T-\tau]} = (Z(t + \tau))_{[0, T-\tau]}$. By the Markov property of Brownian motion, the law of $(Z^1, Z^2)$ is 
\[ C' \int_{-a}^\infty (1_{E_a} \mu^\gamma_{{\bbH} } (b i, c) )  \times \mu_{{\bbR_+^2}} (a+ c+  i, 0) \, dc, \qquad E_a = \{ \text{curve stays in } \{ \mathrm{Re}(z) > -a\}. \]
We can disintegrate $\mu_\bbH^\gamma(bi, c) = \int_\bbR \mu_\bbH^\gamma(bi, c; x)$ according to $X = -\inf_t \mathrm{Re}(Z^1(t))$, then use a change of variables $y = x+c$ to write the law of $(Z^1, Z^2)$ as 
\[C'  \int_0^\infty \int_{0}^a\mu_\bbH^\gamma(bi, y-x; x) \times \mu_{\bbR_+^2}^\gamma(a-x+y + i, 0) 
 \, dx \, dy .\]
 From the definition of the boundary length process $Z$ (Proposition~\ref{prop-thin-disk-mot}), the integration variables $(x,y)$ immediately above agree with those in~\eqref{eq-triangle-mot-proof}; that is, for each $x \in (0,a)$ and $y>0$ 
 \[\{ \text{boundary lengths of }\cT \text{ are }(x,y,b)\} = \{-\inf_t \mathrm{Re} (Z^1(t)) = x \text{ and } Z^1(\tau) = y-x\}.\] By disintegrating, we conclude that for all $x \in (0,a)$ and $y>0$, for $\cT$ sampled from $\QT^\uparrow(2-\frac{\gamma^2}2, \gamma^2, \frac{\gamma^2}2)(x,y,b)$, the boundary length process of $\cT$ has law $C'' \mu_\bbH^\gamma(bi, y-x; x)$. Since $a$ was arbitrary we can remove the restriction on $x$ to get the first claim. 

 The second claim on recovering the curve-decorated quantum surface from its boundary length process follows from the second claim of Proposition~\ref{prop-thin-disk-mot}.
\end{proof}

\begin{proof}[{Proof of Proposition~\ref{prop-mot-triangle-down}}]
Let $\cT^\uparrow$ be a sample from $\QT^\uparrow (2-\frac{\gamma^2}2, \gamma^2, \frac{\gamma^2}2)(b)$. Let $\cT^\downarrow$ be $\cT^\uparrow$ with its curve replaced by its time-reversal, so the law of $\cT^\downarrow$ is $\QT^\downarrow (2-\frac{\gamma^2}2, \gamma^2, \frac{\gamma^2}2)(b)$. Denote the respective boundary length processes of $\cT^{\uparrow/\downarrow}$ by $(L_t^{\uparrow/\downarrow}, R_t^{\uparrow/\downarrow})_{[0,T]}$.
By Lemma~\ref{lem-mot-triangle},  the law of $(L_t^\uparrow, R_t^\uparrow)_{[0,T]}$ of $\cT^\uparrow$ is $C\mu^\gamma_{\bbH, \mathrm{exit}}(bi) = \int_\bbR \mu^\gamma_\bbH(bi, c)\, dc$. Directly from the definitions we have  $(R_t^\downarrow, L_t^\downarrow)= (L^\uparrow_{T - t} - L^\uparrow_T, R^\uparrow_{T-t})$, so the law of $(R_t^\downarrow, L_t^\downarrow)$ is $C\int_\bbR  \mu^\gamma_\bbH(0, bi-c) \, dc$; reflecting $\bbH$ along the main diagonal to get $\bbR_+ \times \bbR$ gives the first claim. 
The second claim follows from the reversibility of $F$ and the second claim of Lemma~\ref{lem-mot-triangle}.
\end{proof}

\subsubsection{Conformal welding identity from mating of trees for the quantum disk}\label{subsec-mot-welding}
The following mating-of-trees result for the quantum disk was first proved for $\gamma \in (\sqrt2,2)$ by \cite{DMS14} and subsequently extended to the full range $\gamma \in (0,2)$ by \cite{AG19}.
\begin{proposition}\label{prop-mot-disk}
Let $(\bbD, \phi, 1)$ be an embedding of a sample from $\QD_{0,1}$ and let $\eta$ be an independent counterclockwise space-filling $\SLE_{16/\gamma^2}$ loop in $\bbD$ rooted at $1$. Parametrize $\eta$ by quantum area and let $T$ be its duration. For $t \in [0,T]$ let $L_t$ (resp.\ $R_t$) be the quantum length of the left (resp.\ right)  boundary arc of $\eta([t,T])$. Then the law of $(L_t, R_t)_{[0,T]}$ is $C \int_0^\infty \mu_{\bbR_+^2}(ri, 0) \, dr$ for some $C>0$. 
\end{proposition}
\begin{proof}
\cite[Theorem 1]{AG19} gives the result when $\QD_{0,1}$ and  $C \int_0^\infty \mu_{\bbR_+^2}(ri, 0) \, dr$ are replaced by $\QD_{0,1}(r)$ and  $C \mu_{\bbR_+^2}(ri, 0)$ for $r = 1$. By rescaling as in the proof of Proposition~\ref{prop-thin-disk-mot} we can remove the condition $r=1$, and integrating over all $r>0$ then yields the result. 
\end{proof}
 
\begin{proposition}\label{prop-mot-QD11} Let  $\gamma \neq \sqrt2$.
Let $(\bbD, \phi, 0, 1)$ be an embedding of a sample from $\QD_{1,1}$ and let $\eta$ be an independent counterclockwise space-filling $\SLE_\kappa$ loop in $\bbD$ rooted at $1$. Let $T_1,T_2 \subset \ol \bbD$ be the regions traced by $\eta$ before and after hitting $0$, and let $q$ be the endpoint of the largest counterclockwise arc from $1$ in $\partial T_1 \cap \partial \bbD$.
Then the joint law of $\cT_1 = (T_1, \phi, 0, q, 1)/{\sim}$ and $\cT_2 = (D, \phi, 0, 1, q)/{\sim}$ is
\eqb\label{eq:prop-mot-QD11} C\int_0^\infty \int_0^\infty \QT(2-\frac{\gamma^2}2, \gamma^2, \frac{\gamma^2}2)(b, x) \times \QT(2- \frac{\gamma^2}2, 2-\frac{\gamma^2}2, 2)(x,b) \, db \, dx \quad \text{ for some }C>0. \eqe
Here, $\QT(2-\frac{\gamma^2}2, \gamma^2, \frac{\gamma^2}2)(b,x)$ denotes the disintegration of $\QT(2-\frac{\gamma^2}2, \gamma^2, \frac{\gamma^2}2)$ where the quantum lengths of the two boundary arcs clockwise from the weight $\frac{\gamma^2}2$ vertex are $b$ and $x$ respectively, and $\QT(2- \frac{\gamma^2}2, 2-\frac{\gamma^2}2, 2)(b,x)$ denotes the disintegration of $\QT(2- \frac{\gamma^2}2, 2-\frac{\gamma^2}2, 2)$ where the quantum lengths of the two boundary arcs clockwise from the weight $2$ vertex are $x$ and $b$ respectively. 
\end{proposition}
\begin{proof}
See Figure~\ref{fig:mot-triangle} (bottom right). Parametrize $\eta$ by quantum area.
Let $Z$ denote the boundary length process of $\eta$ in $(\bbD, \phi, 1)$, let $t$ be the duration of $Z$, and let $s_1$ be the time $\eta$ hits $0$. Since the marked bulk point is sampled from the quantum area measure, by Proposition~\ref{prop-mot-disk} the joint law of $(s_1,t, Z)$ is $C \cdot 1_{0 < s_1 < t}\,ds_1\, 1_{t>0} dt \,
    \int_0^\infty \mu_{{\bbR_+^2}}(t;ri, 0) (dZ) \, dr$. Reparametrizing $s_2 := t - s_1$, the joint law of $(s_1, s_2, Z)$ is $C \cdot 1_{s_1>0}\,ds_1\, 1_{s_2>0} ds_2 \,
    \int_0^\infty \mu_{{\bbR_+^2}}(s_1+s_2;ri, 0) (dZ)\,dr$.
    
    Let $Z^1 = Z|_{[0,s_1]}$ and $Z^2 = Z(\cdot + s_1)|_{[0, s_2]}$. 
    By a variant of the Markov property~\eqref{eq-markov-bulk}, the joint law of $(Z^1, Z^2)$ is 
    \[C \int_0^\infty \int_0^\infty  \int_0^\infty \int_{\bbR_+^2} \mu^\gamma_{\bbR_+^2 } (s_1;ri, z) \mu^\gamma_{\bbR_+^2}(s_2;z, 0) \, dz  \, dr  \, ds_1\, ds_2 = C \int_0^\infty \int_{\bbR_+^2} \mu_{\bbR_+^2}^\gamma (ri, z) \mu_{\bbR_+^2}^\gamma(z,0)\, dz \, dr. \]
The measure $\mu_{\bbR_+^2}^\gamma(ri, z)$ is the restriction of $\mu_{\bbR_+ \times \bbR}^\gamma(ri, z)$ to paths $Z^1$ such that $G := \inf_t \mathrm{Im}(Z^1(t)) > 0$. Disintegrate over the value of $G$ to get $\mu_{\bbR_+^2}^\gamma(ri, z) = \int_0^\infty \mu_{\bbR_+ \times \bbR}^\gamma (ri, z; g) \, dg$. Letting $\wt Z^1 = Z^1 - Z^1(0)$, the joint law of $(\wt Z^1, Z^2)$ is 
\[ C \int_0^\infty \int_0^r \int_{\bbR_+^2} \mu_{\bbR_+\times \bbR}^\gamma (0, z - ri; g-r) \mu_{\bbR_+^2}^\gamma (z, 0)  \, dz \, dg \, dr.\]
Reparametrizing $y = r - g$, $b = \mathrm{Re}(z)$ and $x = \mathrm{Im}(z) - g$, we can rewrite as 
\[ C \int_0^\infty \int_0^\infty \int_0^\infty \mu_{\bbR_+\times \bbR}^\gamma (0, b + (x-y)i; -y) \left( \int_0^\infty \mu_{\bbR_+^2}^\gamma (b + (x+g)i, 0) \, dg\right) \, dy \, db \,dr.\]
By Propositions~\ref{prop-mot-triangle-down} and~\ref{prop-thin-disk-mot}, writing $\cD = (D, \phi, 0, 1)/{\sim}$, the joint law of $\cT_1$ and $\cD$ is 
\[ C\int_0^\infty \int_0^\infty \QT(2-\frac{\gamma^2}2, \gamma^2, \frac{\gamma^2}2)(b, x) \times \left( \int_0^\infty \Md_{0,2}(2-\frac{\gamma^2}2 ;x+g, b)\, dg \right) \, db \, dx. \]
To conclude, by Lemma~\ref{lem:QT(W,2,W)}, a sample from $\QT(2-\frac{\gamma^2}2, 2-\frac{\gamma^2}2, 2; g, x, b)$ can be obtained from a sample from $C \Md_{0,2}(2-\frac{\gamma^2}2 ;x+g, b)$ by adding a marked point to the boundary splitting the length $x+g$ arc into arcs of lengths $g$ and $x$. 
\end{proof}

\subsection{Proof of Theorem~\ref{thm:welding}}\label{subsec:pf-welding}

\begin{figure}[htb]
    \centering
    \includegraphics[scale=0.6]{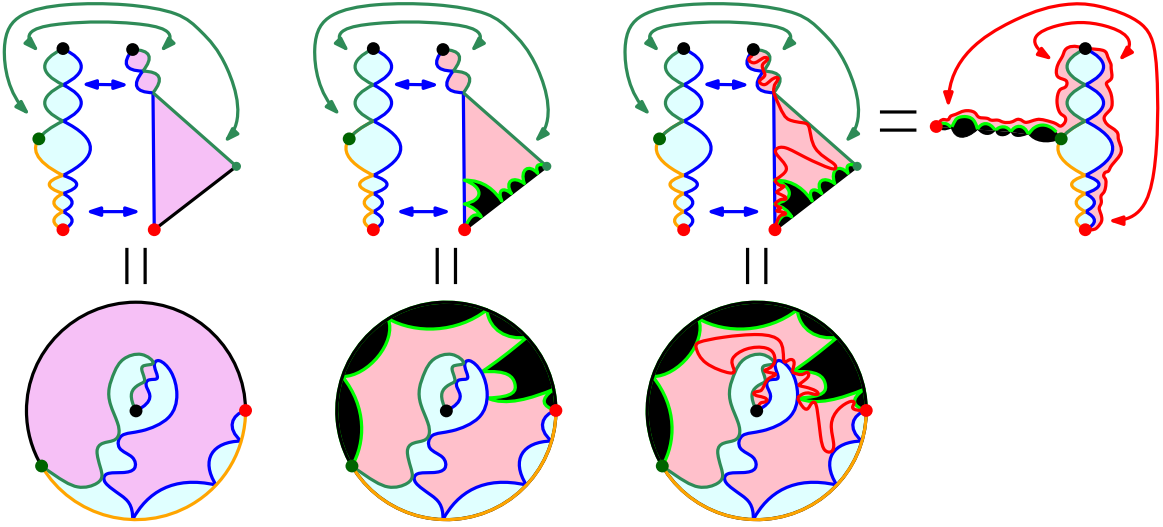}
    \caption{Diagram for proof of Theorem~\ref{thm:welding}. \textbf{Left:} The conformal welding result in Proposition~\ref{prop-mot-QD11}. The blue and green curves are the boundaries $\eta_0^L$ and $\eta_0^R$ of the space-filling $\SLE_\kappa$ loop stopped when hitting 0. \textbf{Middle:} The continuation of the $\eta_0^R$ after merging into the boundary (green). This cuts the weight $(2-\frac{\gamma^2}{2},\frac{\gamma^2}{2},\gamma^2)$ quantum triangle in the left panel from Proposition~\ref{prop-mot-QD11} into a weight $(2-\frac{\gamma^2}{2},2-\frac{\gamma^2}{2},2)$ quantum triangle (which equals a constant times $\Md_{0,2}(2-\frac{\gamma^2}{2})$) and a weight $\gamma^2-2$ quantum disk (black). \textbf{Right:} By drawing $\SLE_\kappa(\frac{\kappa}{2}-4;\frac{\kappa}{2}-4)$ curves (red) in each pocket of the weight $2-\frac{\gamma^2}{2}$ quantum disk (pink), we obtain two forested line segments by Proposition~\ref{prop:weld:segment}. Conformally welding the blue and green boundaries then gives the desired welding picture in Theorem~\ref{thm:welding} as on the top right, and it follows from Proposition~\ref{prop:space-filling-loop-radial} that the concatenation of the interfaces form a radial $\SLE_\kappa(\kappa-6)$ curve from 1 to 0 with force point $e^{i0^-}$.}
    \label{fig:radial-weld}
\end{figure}

Let $\wt\kappa=\gamma^2=\frac{16}{\kappa}$. We begin with the setting of Proposition~\ref{prop-mot-QD11}, where $\eta$ is a counterclockwise space-filling $\SLE_\kappa$ loop in $\bbD$ from 1 to 1 drawn on a quantum disk from $\Md_{1,1}(\gamma;\gamma)$ (Recall that by Proposition~\ref{prop:mdisk11-QD} $\Md_{1,1}(\gamma;\gamma) = C\QD_{1,1}$ for some constant $C$). Let $\tau_0$ be the first time $\eta$ hits $0$, and $\eta_0^L,\wt{\eta}_0^R$ be the left and right boundaries of $\eta([0,\tau_0])$. Then $\eta_0^L,\wt{\eta}_0^R$ are the interfaces under the conformal welding~\eqref{eq:prop-mot-QD11}. Let $\tau_0^R$ be the time when $\wt{\eta}_0^R$ hits $\partial\bbD$, and $\sigma_0^R$ be the last time before $\tau_0^R$ when $\wt{\eta}_0^R$ hits $\eta_0^L$. See also the left panel of Figure~\ref{fig:radial-weld} for the setup.

We draw an independent $\SLE_{\wt\kappa}(\wt\kappa-4;-\frac{\wt\kappa}{2})$ curve $\hat{\eta}^R$ from $\wt{\eta}_0^R(\tau_0^R)$ to 1 in the connected component of $\eta([0,\tau_0])$ with 1 on the boundary, where the force points are located at $\wt{\eta}_0^R(\tau_0^R)^-$ and $\wt{\eta}_0^R(\sigma_0^R)$, and let $\eta_0^R$ be its concatenation with $\wt{\eta}_0^R$. Then by Theorem~\ref{thm:disk+QT}, the quantum disk $(\bbD,\phi,0,1)$ decorated with $\eta_0^L$ and $\eta_0^R$ is equal to
\begin{equation}\label{eq:radial-weld-pf}
    \iiint_{\bbR_+^3}\Wd\left(\QT(2-\frac{\gamma^2}{2},2-\frac{\gamma^2}{2},2;b,x,y), \Md_{0,2}(\gamma^2-2;y), \QT(2-\frac{\gamma^2}{2},2-\frac{\gamma^2}{2},2;x,b) \right)dbdxdy.
\end{equation}
Here in~\eqref{eq:radial-weld-pf}, $b$ and $x$ represent the quantum lengths of $\eta_0^L$ and $\eta_0^R$, where $y$ represents the quantum length of $\hat\eta^R$. As in the middle panel of Figure~\ref{fig:radial-weld}, $b,x,y$ correspond to the quantum lengths of the blue, dark green and light green curve segments. Now we perform a change of variables $r = x+y$ and $s = y$, and rewrite~\eqref{eq:radial-weld-pf} as
\begin{equation}\label{eq:radial-weld-pf-A}
    \iint_{\bbR_+^2}\int_0^r\Wd\left(\QT(2-\frac{\gamma^2}{2},2-\frac{\gamma^2}{2},2;b,r-s,s), \Md_{0,2}(\gamma^2-2;s), \QT(2-\frac{\gamma^2}{2},2-\frac{\gamma^2}{2},2;r-s,b) \right)ds\,drdb.
\end{equation}
On one hand, by Definition~\ref{def-qt-thin}, we have the natural disintegration 
\begin{equation*}
    \QT(2-\frac{\gamma^2}{2},2-\frac{\gamma^2}{2},\gamma^2-2;r,b) = \int_0^r \QT(2-\frac{\gamma^2}{2},2-\frac{\gamma^2}{2},2;r-s,b) \times  \Md_{0,2}(\gamma^2-2;s)\,ds
\end{equation*}
where $r$ represent the quantum length of the boundary arc of a weight $(2-\frac{\gamma^2}{2},2-\frac{\gamma^2}{2},\gamma^2-2)$ quantum triangle from the weight $\gamma^2-2$ vertex to the weight $2-\frac{\gamma^2}{2}$ vertex. On the other hand, by Lemma~\ref{lem:QT(W,2,W)}, up to a constant, $\QT(2-\frac{\gamma^2}{2},2-\frac{\gamma^2}{2},2;b,r-s,s)$ can be generated by marking the point on the right boundary of a quantum disk from $\Md_{0,2}(2-\frac{\gamma^2}{2};b,r)$ with distance $s$ to the bottom vertex. As a consequence, by a re-arranging and forgetting the marked point $\wt{\eta}_0^R(\tau_R^0)$ (the dark green dot in Figure~\ref{fig:radial-weld}), ~\eqref{eq:radial-weld-pf-A} is now a constant times
\begin{equation}\label{eq:radial-weld-pf-B}
    \iint_{\bbR_+^2} \Wd\left(\Md_{0,2}(2-\frac{\gamma^2}{2};b,r),   \QT(2-\frac{\gamma^2}{2},2-\frac{\gamma^2}{2},\gamma^2-2;r,b) \right)drdb,
\end{equation}
where $r$ correspond to the quantum length of $\eta_0^R$.

Finally, in each pocket $D$ of $\bbD\backslash(\eta_0^L\cup\eta_0^R)$ between $\eta_0^L$ and $\eta_0^R$, we draw an independent $\SLE_\kappa(\frac{\kappa}{2}-4;\frac{\kappa}{2}-4)$ curve $\eta_D$. By Proposition~\ref{prop:weld:segment}, the quantum surface $(\bbD,\phi,0,1)$ decorated with the curves $\eta_0^L,\eta_0^R,(\eta_D)_D$ is equal to 
\begin{equation}\label{eq:radial-weld-pf-C}
     \iint_{\bbR_+^2}\int_0^\infty \Wd\left(\mathcal{M}_2^{\text{f.l.}}(b;\ell), \mathcal{M}_2^{\text{f.l.}}(r;\ell),  \QT(2-\frac{\gamma^2}{2},2-\frac{\gamma^2}{2},\gamma^2-2;r,b) \right)d\ell\,drdb.
\end{equation}
If we further forest the outer boundary of $(\bbD,\phi,0,1)$, then from Definition~\ref{def:q.t.f.s.} (along with the identification ~\eqref{eq:def-forest-bdry}), the surface from~\eqref{eq:radial-weld-pf-C} equals a constant times the right hand side of~\eqref{eq:thm-welding}. On the other hand, by Proposition~\ref{prop:space-filling-loop-radial} (and a conformal map from $\bbD$ to $\bbH$), the union of $\eta_0^L$, $\eta_0^R$ along with all the $\eta_D$'s is equal to the trace of a radial $\SLE_\kappa(\kappa-6)$ curve with force point at $e^{i0^-}$. Therefore we conclude the proof of Theorem~\ref{thm:welding}.
\qed
 
\section{Derivation of the touching probability and CLE conformal radii moments}
\label{sec:proof}

In this section we prove Theorems~\ref{thm:touch}--\ref{thm:CR-widetilde-D}. 
In Section~\ref{subsec:weld-touch}, we cut a generalized quantum disk using an independent radial ${\rm SLE}_\kappa(\kappa-6)$ until the first time it closes a loop around the marked bulk point; Proposition~\ref{prop:weld-wind} identifies the two resulting generalized quantum surfaces.
In Section~\ref{subsec:computation-proof}, we state Proposition~\ref{prop:CRratio}, which gives the ratio between the moments of conformal radii of clockwise/counterclockwise loops, and prove Theorems~\ref{thm:touch}--\ref{thm:CR-widetilde-D} via Proposition~\ref{prop:CRratio}. In Sections~\ref{subsec:computation-prelim}--\ref{subsec:computation} we use exact formulas from the Liouville CFT theory to carry out the computations and prove Proposition~\ref{prop:CRratio}.

\subsection{Boundary touching event from conformal welding}\label{subsec:weld-touch}
The goal of this section is to prove Proposition~\ref{prop:weld-wind}. Consider a forested quantum triangle $\cT^f$ of weights $(2-\frac{\gamma^2}{2},2-\frac{\gamma^2}{2},\gamma^2-2)$ in Theorem~\ref{thm:welding}. By Definition~\ref{def-qt-thin}, we have the decomposition $(\cT_1^f,\cD^f)$ of $\cT^f$:
\eqb\label{eq:qt-decom}
(\cT_1^f,\cD^f)\sim  \QT^f(\frac{3\gamma^2}{2}-2,2-\frac{\gamma^2}{2},\gamma^2-2)\times \Mfd_{0,2}(2-\frac{\gamma^2}{2}).
\eqe
In other words, $\cT^f$ can be generated by  connecting $(\cT_1^f,\cD^f)$ sampled from~\eqref{eq:qt-decom} as in Definition~\ref{def-qt-thin}.
We write $L_1$ and $L_2$ for the generalized boundary lengths for the left and right boundary arcs of $\cT_1$; see Figure~\ref{fig:qt-decomposition} for an illustration.
\begin{figure}
    \centering
    \begin{tabular}{cc}
      \includegraphics[scale=0.55]{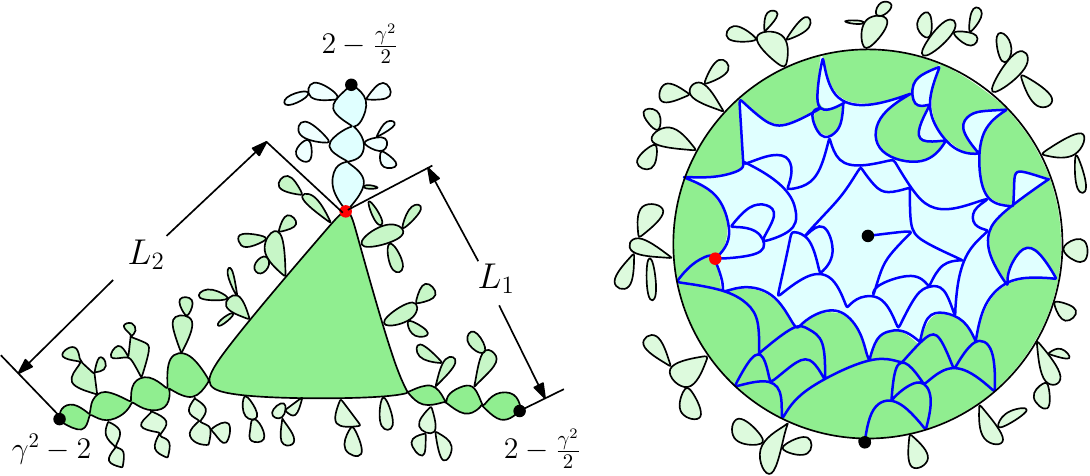}    &  \   \includegraphics[scale=0.6]{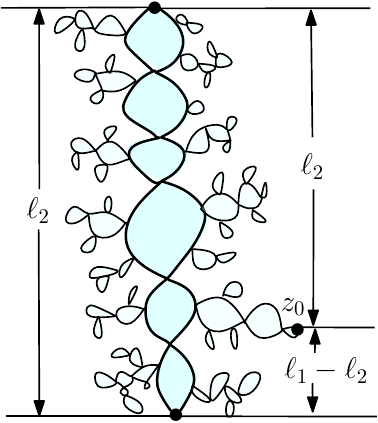}
    \end{tabular}
    \caption{\textbf{Left:} The decomposition $(\cT_1^f,\cD^f)$ of the weight $(2-\frac{\gamma^2}{2},2-\frac{\gamma^2}{2},\gamma^2-2)$ quantum triangle. \textbf{Middle:} The red marked point is the point at which $\eta$ first closes a loop around 0. Under the event $L_2>L_1$, the first loop is counterclockwise, and therefore the CLE loop surrounding 0 touches the boundary. \textbf{Right:} The welding of the $\cD^f$ to itself in Lemma~\ref{lem:-wd-QTf(W,2,W)}.}
    \label{fig:qt-decomposition}
\end{figure}

Consider the conformal welding of $\cT^f$ as in Theorem~\ref{thm:welding} and let $\eta$ be the interface. Since the left and right boundaries of $\cT^f$ are glued together according to the generalized quantum length, it turns out that on the event $\{L_2>L_1\}$, a fraction of the right boundary of $\cD^f$ is glued to a fraction of the left boundary of $\cT_1^f$. This forces the first loop around 0 made by the radial $\SLE_\kappa(\kappa-6)$ interface $\eta$ to be counterclockwise and therefore by Proposition~\ref{prop:prelim} the CLE loop surrounding 0 touches the boundary. On the other hand, on the event $\{L_2<L_1\}$, a fraction of the left boundary of $\cD^f$ is glued to a fraction of the right boundary of $\cT_1^f$, and thus the first loop is clockwise. This gives an expression of the boundary touching event for the CLE in terms of boundary lengths $L_1,L_2$ of $\cT_1^f$.

Let $W>0$. Recall the definition of $\Md_{0,2,\bullet}(W)$ given above Lemma~\ref{lem:QT(W,2,W)}. We now  define $\mathcal{M}_{0,2, \bullet}^{\textup{f.d.}}(W)$ analogously.   First sample  a forested quantum disk from $\Mfd_{0,2}(W)$ and weight its law by the generalized quantum length of its left boundary arc. Then sample a marked point on the left boundary according to the probability measure proportional to the generalized quantum length. We denote the law of the triply marked quantum surface  by $\mathcal{M}_{0,2, \bullet}^{\textup{f.d.}}(W)$.
\begin{lemma}\label{lem:QTf(W,2,W)}
    For $W \in (0, \frac{\gamma^2}{2}) \cup (\frac{\gamma^2}{2}, \infty)$, we have
    \begin{equation}
         \Mfd_{0,2,\bullet}(W) = C_0 \QT^f(W,\gamma^2-2,W) \quad \mbox{with} \quad C_0 = \frac{\gamma^2}{4}.
    \end{equation}
\end{lemma}
\begin{proof}
     By Definition~\ref{def-thin-disk}, a sample from $\QT^f(W,\gamma^2-2,W)$ can be constructed by concatenating samples from $\QT^f(W,2,W)\times \Mfd_{0,2}(\gamma^2-2)$. By Lemma~\ref{lem:QT(W,2,W)} and~\cite[Lemma 3.15]{AHSY23}, we get Lemma~\ref{lem:QTf(W,2,W)}.
\end{proof}
Combining with Theorem~\ref{thm:welding}, we have the following lemma; see the right panel of Figure~\ref{fig:qt-decomposition}.
\begin{lemma}\label{lem:-wd-QTf(W,2,W)}
    Let $\cD^f$ be a sample from $\Mfd_{0,2}(2-\frac{\gamma^2}{2})$ restricted to the event where its left boundary length $\ell_2$ is less than right boundary length $\ell_1$. Mark the point $z_0$ on the right boundary with distance $\ell_2$ to its top vertex. Then if we weld the left boundary of $\cD^f$ to its right boundary starting from the top vertex, then the resulting curve-decorated surface has law $C_0C_{\gamma}^{-1}\Mfd_{1,1}(\gamma,\gamma)\otimes\rm{raSLE}_\kappa(\kappa-6)$, where  $C_\gamma$ and $C_0$ are the constants from Theorem~\ref{thm:welding} and Lemma~\ref{lem:QTf(W,2,W)}, respectively.
\end{lemma}
\begin{proof}
    By marking the point $z_0$ on $\cD^f$, by a disintegration over Lemma~\ref{lem:QTf(W,2,W)}, the surface $\cD^f$ has law 
    \begin{equation*}
       C_0\int_0^{\infty} \int_{\ell_2}^\infty  \QT^f(2-\frac{\gamma^2}{2},\gamma^2-2,2-\frac{\gamma^2}{2};\ell_2,\ell_2,\ell_1-\ell_2)\,d\ell_1 \, d\ell_2.
    \end{equation*}
    By a change of variables, the above expression is the same as
    \begin{equation*}
       C_0 \int_0^\infty \int_0^\infty \QT^f(2-\frac{\gamma^2}{2},\gamma^2-2,2-\frac{\gamma^2}{2};\ell_2,\ell_2,\ell_1')\,d\ell_2 \, d\ell_1'.
    \end{equation*}
    Then the lemma follows directly follows from Theorem~\ref{thm:welding}.
\end{proof}

For $\alpha\in\bbR$, we write $\Delta_\alpha = \frac{\alpha}{2}(Q-\frac{\alpha}{2})$. Let $\mathsf{m}$ be the law of a radial $\SLE_\kappa(\kappa-6)$ curve $\wt\eta$ from 1 to 0 with force point $1e^{i0^-}$ stopped at the first time $\sigma_1$ when it closes a loop  around 0 as in Section~\ref{subsec:pre-CLE}. Recall that $T$ is the event where $\Loop$ touches the boundary, which by Proposition~\ref{prop:prelim} is the same as the event where $\wt\eta$ forms a counterclockwise loop. Let $D_{\wt\eta}$ be the connected component of $\bbD\backslash{\wt\eta}$ containing 0. Define the measure $\mathsf{m}^\alpha(\widetilde \eta)$ by $\frac{d \mathsf{m}^\alpha(\widetilde \eta)}{d \mathsf{m}(\widetilde \eta)} = {\rm CR}(0,D_{\widetilde \eta})^{2\Delta_\alpha-2}$.

Now we prove the main result of this section. See Figure~\ref{fig:weld5} for an illustration.
\begin{proposition}\label{prop:weld-wind} Let $\gamma \in (\sqrt2, 2)$ and $\alpha\in\bbR$. For some constant $C_0$ depending only on $\gamma$, we have
\begin{equation}
\label{eq:weld2}
\begin{aligned}
&\Mfd_{1,1}(\alpha,\gamma) \otimes \mathsf{m}^\alpha(\widetilde \eta) \mathbbm{1}_T = C_0 \iint_{L_2>L_1>0} {\rm Weld}( {\rm QT}^f (\frac{3\gamma^2}{2}-2,2-\frac{\gamma^2}{2},\gamma^2-2;L_1,L_2), \Mfd_{1,1}(\alpha,\gamma; L_2-L_1)) dL_1dL_2\,;\\
&\Mfd_{1,1}(\alpha,\gamma) \otimes \mathsf{m}^\alpha(\widetilde \eta) \mathbbm{1}_{T^c} = C_0 \iint_{L_1>L_2>0} {\rm Weld}( {\rm QT}^f (\frac{3\gamma^2}{2}-2,2-\frac{\gamma^2}{2},\gamma^2-2;L_1,L_2), \Mfd_{1,1}(\alpha,\gamma; L_1-L_2)) dL_1dL_2.
\end{aligned}
\end{equation}
Here, for the quantum triangle we conformally weld the two forested boundary arcs adjacent to the weight $\frac{3\gamma^2}{2}-2$ vertex, starting by identifying the weight $\gamma^2-2$ vertex with the weight $2-\frac{\gamma^2}2$ vertex, and conformally welding until the shorter boundary arc has been completely welded to the longer boundary arc. Then, the quantum disk is conformally welded to the remaining segment of the longer boundary arc, identifying its boundary marked point with the weight $\frac{3\gamma^2}2-2$ vertex of the quantum triangle.
\end{proposition}

\begin{figure}[htb]

\centering
\includegraphics[scale=0.55]{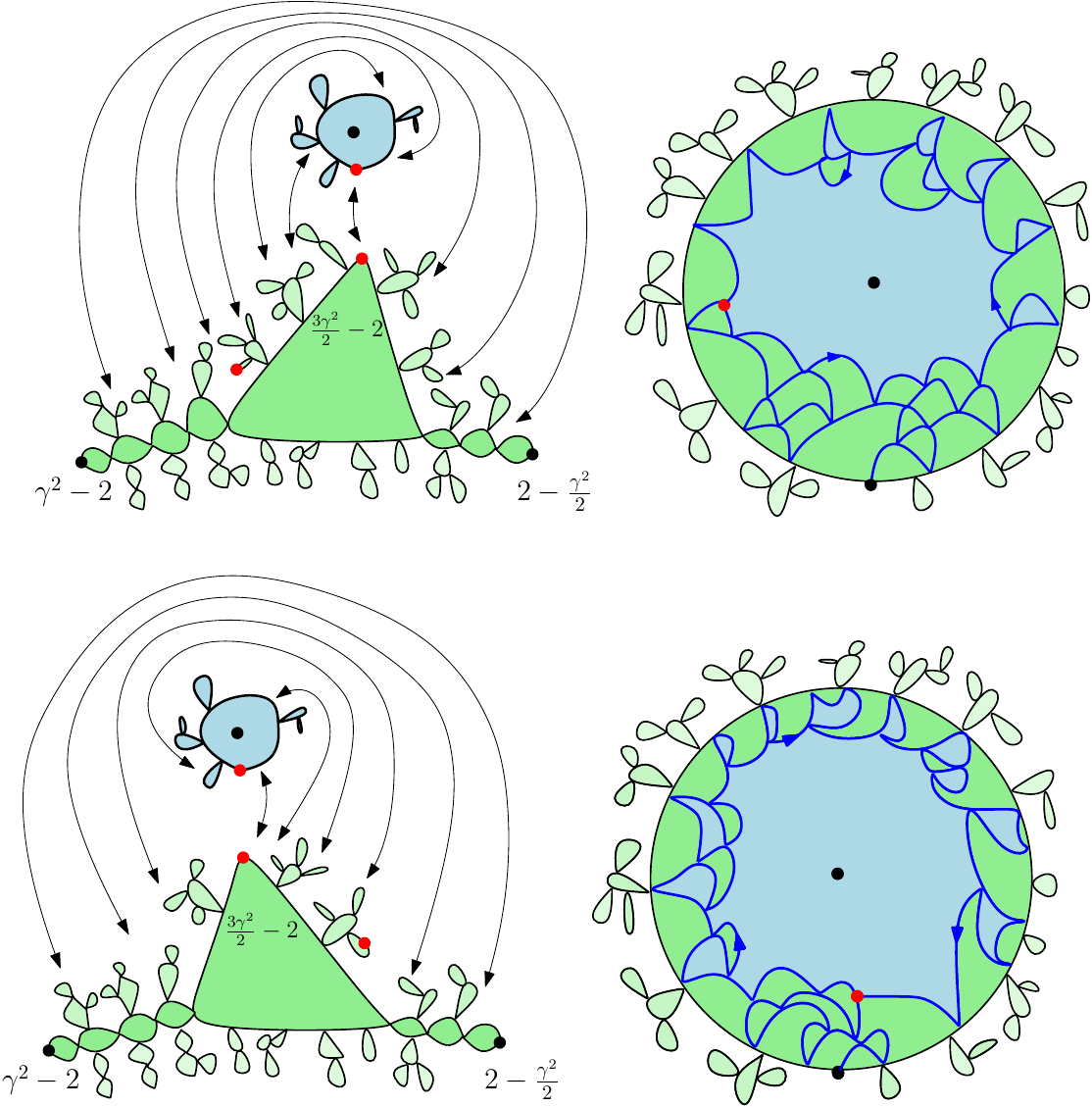}
\caption{Illustration of Proposition~\ref{prop:weld-wind}. The first panel corresponds to the case of $L_2>L_1$, and the second panel corresponds to the the case of $L_2<L_1$. The red point $z_0$ corresponds to the location that $\widetilde \eta$ first closes a loop around 0.}\label{fig:weld5}
\end{figure}

\begin{proof}
We start with the $\alpha=\gamma$ case and restrict to the event $\{L_2>L_1\}$. Let $\cT^f$ be a sample from $\QT^f(2-\frac{\gamma^2}{2},2-\frac{\gamma^2}{2},\gamma^2-2)$ and let $(\cT_1^f,\cD^f)$ be the decomposition of $\cT^f$ in~\eqref{eq:qt-decom}.  By Theorem~\ref{thm:welding}, for some constant $C_\gamma\in(0,\infty)$, the conformal welding on right hand side of~\eqref{eq:thm-welding} when restricted to the event $\{L_2>L_1\}$ can be written as
\begin{equation}\label{eq:prop-weld-wind}
     C_\gamma \iiint_{\ell>L_2>L_1>0} \mathrm{Weld}\big(\QT^f(\frac{3\gamma^2}{2}-2,2-\frac{\gamma^2}{2},\gamma^2-2;L_1,L_2),\Mfd_{0,2}(2-\frac{\gamma^2}{2};\ell-L_1,\ell-L_2) \big)\,dL_1dL_2d\ell
\end{equation}
and equals $\Mfd_{1,1}(\gamma,\gamma)\otimes \mathbbm{1}_{T}\rm{raSLE}_\kappa(\kappa-6)$. Since $L_2>L_1$, we have $\ell-L_2<\ell-L_1$. Mark the point $z_0$ on the right boundary of $\cD^f$ with distance $\ell-L_2$ to the top vertex. By Lemma~\ref{lem:-wd-QTf(W,2,W)} and a disintegration, we can first weld $\cD^f$ to itself to get a forested quantum disk $\wt{\cD}^f$, whose law is $C_0C_\gamma^{-1}\Mfd_{1,1}(\gamma,\gamma;L_2-L_1)$. This corresponds to integrating over $\ell$ in the expression~\eqref{eq:prop-weld-wind}. Then \eqref{eq:weld2} for $\alpha=\gamma$ then follows by welding $\wt{\cD}^f$ to $\cT_1^f$ as in the top panel of Figure~\ref{fig:weld5}. The setting where $\{L_2<L_1\}$ and $\alpha=\gamma$ follows analogously with the same constant $C_0$. 

For $\alpha\neq\gamma$, let $(\phi,\eta)$ be a sample from the left hand side of~\eqref{eq:weld2}. Let $\psi_\eta:D_\eta\to\bbD$ be the conformal map fixing $0$  and sending $z_\eta$ to $1$, where $z_\eta$ is the terminal point of $\eta$ (and $\eta$ closes a loop surrounding 0). Set $X = \phi\circ\psi_\eta^{-1}+Q\log|(\psi_\eta^{-1})'|$. Then the claim follows by weighting the law of $(\phi,\eta)$ by $\e^{\alpha^2-\gamma^2}e^{\frac{\alpha-\gamma}{2}X_\e(0)}$ and sending $\e\to0$, where $X_\e(0)$ is the average of the field $X$ around the origin. The proof is identical to that of~\cite[Theorem 4.6]{ARS21} by taking~\cite[Lemmas 4.7 and 4.8]{ARS21} as input. We omit the details. 
\end{proof}

\subsection{Proof of Theorems~\ref{thm:touch}--\ref{thm:CR-widetilde-D}}\label{subsec:computation-proof}

Based on Proposition~\ref{prop:weld-wind} and exact formulas from LCFT, we shall prove the following in Sections~\ref{subsec:computation-prelim}--\ref{subsec:computation}.
\begin{proposition}
    \label{prop:CRratio}
    For any $Q<\alpha<\frac{4}{\gamma}$, we have
    \eqb\label{eq:prop-CRratio}
    \frac{\mathbb{E} [{\rm CR}(0,D_{\widetilde \eta})^{2 \Delta_\alpha-2}  \mathbbm{1}_{T^c}]}{\mathbb{E} [{\rm CR}(0,D_{\widetilde \eta})^{2 \Delta_\alpha-2}  \mathbbm{1}_{T}]} =  \frac{\sin(\pi(\gamma-\frac{2}{\gamma})(Q-\alpha))}{2\cos(\pi(1-\frac{\gamma^2}{4})) \sin(\pi(\frac{2}{\gamma}-\frac{\gamma}{2})(Q-\alpha))}\,.
    \eqe
\end{proposition}

Using Proposition~\ref{prop:CRratio}, we can now prove Theorems~\ref{thm:touch}, \ref{thm:CR}, and \ref{thm:CR-widetilde-D}.

\begin{proof}[Proof of Theorems~\ref{thm:CR} and \ref{thm:CR-widetilde-D}]
    For $Q<\alpha<\frac{4}{\gamma}$, let \begin{align*}
        A(\alpha) = \mathbb{E} [{\rm CR}(0,D_{\Loop})^{2 \Delta_\alpha-2}  \mathbbm{1}_{T}]\quad \mbox{and} \quad B(\alpha) = \mathbb{E} [{\rm CR}(0,D_{\Loop})^{2 \Delta_\alpha-2}  \mathbbm{1}_{T^c}].
    \end{align*}
    By Equation \eqref{eq:SSW09},
    \begin{equation}
        \label{eq:AB1}
        A(\alpha) + B(\alpha) = \mathbb{E} [{\rm CR}(0,D_{\Loop})^{2 \Delta_\alpha-2} ] = \frac{\cos(\pi(1-\frac{\gamma^2}{4}))}{\cos(\pi\frac{\gamma}{2}(Q-\alpha))}.
    \end{equation}
    By Proposition~\ref{prop:prelim}, we have
    \begin{equation}
        \label{eq:B}
        \begin{aligned}
        \mathbb{E} [{\rm CR}(0,D_{\Loop})^{2 \Delta_\alpha-2} |T ] &= 
    {\mathbb{E} [{\rm CR}(0,D_{\widetilde \eta})^{2 \Delta_\alpha-2} |  T]} ;\\
    \mathbb{E} [{\rm CR}(0,D_{\Loop})^{2 \Delta_\alpha-2} |T^c ] &= 
    {\mathbb{E} [{\rm CR}(0,D_{\widetilde \eta})^{2 \Delta_\alpha-2} |  T^c]} \cdot \mathbb{E} [{\rm CR}(0,D_{\Loop})^{2 \Delta_\alpha-2} ].
        \end{aligned}
    \end{equation}
    Therefore by Proposition~\ref{prop:CRratio} and Equation~\eqref{eq:SSW09},
    \begin{equation}
        \label{eq:AB2}
        \begin{aligned}
            \frac{B(\alpha)}{A(\alpha)} &=\mathbb{E} [{\rm CR}(0,D_{\Loop})^{2 \Delta_\alpha-2} ] \times  \frac{\mathbb{E} [{\rm CR}(0,D_{\widetilde \eta})^{2 \Delta_\alpha-2}  \mathbbm{1}_{T^c}]}{\mathbb{E} [{\rm CR}(0,D_{\widetilde \eta})^{2 \Delta_\alpha-2}  \mathbbm{1}_{T}] }  = \frac{\sin(\pi(\gamma-\frac{2}{\gamma})(Q-\alpha))}{2\cos(\pi\frac{\gamma}{2}(Q-\alpha)) \sin(\pi(\frac{2}{\gamma}-\frac{\gamma}{2})(Q-\alpha)) }.
        \end{aligned}
    \end{equation}
    Combining \eqref{eq:AB1} and \eqref{eq:AB2}, we get that
    \begin{equation}
    \label{eq:AB3}
    \begin{aligned}
        &A(\alpha) = \mathbb{E} [{\rm CR}(0,D_{\Loop})^{2 \Delta_\alpha-2}  \mathbbm{1}_{T}] =  \frac{2\cos(\pi (1-\frac{\gamma^2}{4})) \sin(\pi(\frac{2}{\gamma}-\frac{\gamma}{2})(Q-\alpha))}{\sin(\pi\frac{2}{\gamma}(Q-\alpha))}; \\
        &B(\alpha) = \mathbb{E} [{\rm CR}(0,D_{\Loop})^{2 \Delta_\alpha-2}  \mathbbm{1}_{T^c}] =  \frac{\cos(\pi (1-\frac{\gamma^2}{4})) \sin(\pi(\gamma - \frac{2}{\gamma})(Q-\alpha))}{\cos(\pi \frac{\gamma}{2}(Q-\alpha))\sin(\pi\frac{2}{\gamma}(Q-\alpha))}.
    \end{aligned}
    \end{equation}
    By the analytic extension in $\alpha$, see e.g.~\cite[Lemma 4.15]{NQSZ}, the first equation holds for $\alpha \in (Q-\frac{\gamma}{2} , Q+\frac{\gamma}{2})$, and the second equation holds for $\alpha  \in (Q-\frac{1}{\gamma}, Q+\frac{1}{\gamma})$. This proves Theorem~\ref{thm:CR}.
    
To see Theorem~\ref{thm:CR-widetilde-D}, recall from Lemma~\ref{lem:wt-D} that $\wt D = D_{\wt \eta}$ a.s.\ on the event $T^c$. Now by~\eqref{eq:B}
    $$
    \mathbb{E}[{\rm CR}(0,\wt D)^{2 \Delta_\alpha-2}  \mathbbm{1}_{T^c}] = \mathbb{E} [{\rm CR}(0,D_{\Loop})^{2 \Delta_\alpha-2}  \mathbbm{1}_{T^c}]/ \mathbb{E} [{\rm CR}(0,D_{\Loop})^{2 \Delta_\alpha-2}].
    $$
    We conclude the proof of Theorem~\ref{thm:CR-widetilde-D} by using~\eqref{eq:SSW09}, \eqref{eq:AB3}, and analytic extensions.
\end{proof}

\begin{proof}[Proof of Theorem~\ref{thm:touch}]
    Taking $\lambda = 0$ in the first claim of Theorem~\ref{thm:CR} yields the result.
\end{proof}

The rest of this section is devoted to the proof of Proposition~\ref{prop:CRratio}. In principle, one can use the conformal welding result Proposition~\ref{prop:weld-wind} to express the ratio on the left side of~\eqref{eq:prop-CRratio} via the boundary length distribution of samples from the generalized quantum triangle ${\rm QT}^f (\frac{3\gamma^2}{2}-2,2-\frac{\gamma^2}{2},\gamma^2-2)$, which in turn can be expressed via three-point structure constant of boundary LCFT computed in~\cite{RZ22}. However, these formulae are highly complicated. To arrive at the simple expression on the right side of~\eqref{eq:prop-CRratio}, we use an auxiliary conformal welding result to reduce the problem to calculations only about the two-pointed quantum disk of weight $\frac{3}{2}\gamma^2-2$, which can be computed via the more tractable boundary reflection coefficient of LCFT. We perform these calculations in Section~\ref{subsec:2ptQD}, after supplying two elementary ingredients in Section~\ref{subsec:computation-prelim}. We then conclude the proof of Proposition~\ref{prop:CRratio} in Section~\ref{subsec:computation}.

\subsection{Preliminary calculations}\label{subsec:computation-prelim}
We record two elementary  results (Lemmas~\ref{lem:levy}-\ref{lem:integral2}) that will be used in the proof of  Proposition~\ref{prop:CRratio}. We need the following conventions.
For $x\in\bbR$, we write $x_+:=\max\{x,0\}$.
We define fractional powers of complex numbers as follows: for $z = r e^{i \theta}$ with $r \in [0,\infty)$ and $\theta \in (-\pi, \pi]$, let $z^p = r^p e^{i \theta p}$.
\begin{lemma}
\label{lem:levy}
Fix $\gamma \in (0,2)$. For $\ell_1,\ell_2 > 0$, let $Y_{\ell_1} $ and $Y_{\ell_2}'$ be two independent random variables satisfying $\mathbb{E} e^{-t Y_{\ell_1}} = e^{-t^{\frac{\gamma^2}{4}}\ell_1}$ and $\mathbb{E} e^{-t Y_{\ell_2}'} = e^{-t^{\frac{\gamma^2}{4}}\ell_2}$ for every $t>0$. 
Then, for any $p \in (-1,0)$ we have 
$$
\mathbb{E} (Y_{\ell_1} - Y_{\ell_2}')_+^p = \frac{4}{\pi \gamma^2} \Gamma(-\frac{4}{\gamma^2}p)\Gamma(p+1) \times \mathrm{Re} \big[ e^{\frac{i \pi (p+1)}{2}} ( e^{\frac{i \pi \gamma^2}{8} } \ell_1 +  e^{-\frac{i \pi \gamma^2}{8} } \ell_2 )^{\frac{4}{\gamma^2}p}  \big] .
$$
\end{lemma}
\begin{proof}
Fix $-1 < p < 0$.  Using the identity $\int_0^\infty u^{-p-1} e^{iu} du = \Gamma(-p) e^{-\frac{i \pi p}{2}}$, we can derive the following: 
\begin{equation}
\label{eq:lem4.8-1}
\frac{\Gamma(p+1)}{\pi} \int_0^\infty u^{-p-1} \cos(\frac{\pi}{2}(p+1) - u) du = 1 \quad \mbox{and} \quad \frac{\Gamma(p+1)}{\pi} \int_0^\infty u^{-p-1} \cos(\frac{\pi}{2}(p+1) + u) du = 0.
\end{equation}
Therefore, for any $x \in \mathbb{R} \setminus \{0\}$, $x_+^p = \frac{\Gamma(p+1)}{\pi} \int_0^\infty u^{-p-1} \cos(\frac{\pi}{2}(p+1) - ux)du $.

Next, we will use the identity for $x_+^p$ and the characteristic function for $Y_{\ell_1} - Y_{\ell_2}'$ to compute $\mathbb{E} (Y_{\ell_1} - Y_{\ell_2}')_+^p$. By the preceding identity,
\begin{equation}
\label{eq:lem4.8-2}
\begin{aligned}
&\quad \frac{\pi}{\Gamma(p+1)}\mathbb{E} (Y_{\ell_1} - Y_{\ell_2}')_+^p =\mathbb{E} \Big( \int_0^\infty u^{-p-1} \cos(\frac{\pi}{2}(p+1) - u (Y_{\ell_1} - Y_{\ell_2}') )du  \Big)\\
&= \underbrace{\mathbb{E} \Big(\int_0^N u^{-p-1} \cos(\frac{\pi}{2}(p+1) - u (Y_{\ell_1} - Y_{\ell_2}') )du \Big)}_{J_N^1} + \underbrace{\mathbb{E}\Big( \int_N^\infty u^{-p-1} \cos(\frac{\pi}{2}(p+1) - u (Y_{\ell_1} - Y_{\ell_2}') )du \Big)}_{J_N^2}
\end{aligned}
\end{equation}
for any $N > 0$. Denote the two integrals by $J_N^1$ and $J_N^2$, respectively. We will take $N$ to infinity and calculate the limit of $J_N^1$. We will also show that $\lim_{N \rightarrow \infty} J_N^2 = 0$. Combining these yields the desired lemma.

We first consider $J_N^1$. For any $t \in \mathbb{R}$, the characteristic function of $Y_{\ell}$ is given by $\mathbb{E} e^{it Y_{\ell}} = \exp [-e^{-\frac{i \pi \gamma^2}{8} {\rm sgn}(t)} |t|^{\frac{\gamma^2}{4}} \ell ]$, where ${\rm sgn}(t)$ is the sign of $t$. Thus, exchanging the integral in $J_N^1$ gives 
$$
J_N^1 = \int_0^N u^{-p-1}{\rm Re} \Big[ e^{\frac{i \pi (p+1)}{2}} \exp(-u^{\frac{\gamma^2}{4}}( e^{\frac{i \pi \gamma^2}{8}}\ell_1 +  e^{-\frac{i \pi \gamma^2}{8}}\ell_2)) \Big] du.
$$
For $z \in \mathbb{C}$ with ${\rm Re} z > 0$, $\int_0^\infty u^{-p-1} e^{-u^{\frac{\gamma^2}{4}} z } du =\frac{4}{\gamma^2}\Gamma(-\frac{4}{\gamma^2}p) z^{\frac{4}{\gamma^2} p}$. Using this identity, we further have
\begin{equation}
\label{eq:lem4.8-3}
\begin{aligned}
    \lim_{N \rightarrow \infty} J_N^1 =\frac{4}{\gamma^2}\Gamma(-\frac{4}{\gamma^2}p) {\rm Re} \Big[e^{\frac{i \pi (p+1)}{2}} ( e^{\frac{i \pi \gamma^2}{8}}\ell_1 +  e^{-\frac{i \pi \gamma^2}{8}}\ell_2)^{\frac{4}{\gamma^2} p} \Big].
\end{aligned}
\end{equation}

Now we consider $J_N^2$. By~\eqref{eq:lem4.8-1}, there exists $M>0$ such that $\sup_{s \geq 0} |\int_s^\infty u^{-p-1}\cos(\frac{\pi}{2}(p+1) \pm u) du | \leq M$. In addition, we have $\lim_{T \rightarrow \infty}\sup_{s \geq T} |\int_s^\infty u^{-p-1}\cos(\frac{\pi}{2}(p+1) \pm u) du | = 0$. Therefore, as $N$ tends to infinity,
\begin{align*}
    |J_N^2| &\leq \mathbb{E} \Big{|} \int_N^\infty u^{-p-1} \cos(\frac{\pi}{2}(p+1) - u (Y_{\ell_1} - Y_{\ell_2}') )du \mathbbm{1}_{|Y_{\ell_1} - Y_{\ell_2}'| < \sqrt{N}} \Big{|} \\
    &\quad + \mathbb{E} \Big{|}  \int_N^\infty u^{-p-1} \cos(\frac{\pi}{2}(p+1) - u (Y_{\ell_1} - Y_{\ell_2}') )du \mathbbm{1}_{|Y_{\ell_1} - Y_{\ell_2}'| \geq \sqrt{N}} \Big{|} \\
    &\leq o_N(1) \mathbb{E} \big( |Y_{\ell_1} - Y_{\ell_2}'|^p\mathbbm{1}_{|Y_{\ell_1} - Y_{\ell_2}'| < \sqrt{N}} \big) + M \mathbb{E}\big( |Y_{\ell_1} - Y_{\ell_2}'|^p\mathbbm{1}_{|Y_{\ell_1} - Y_{\ell_2}'| \geq \sqrt{N}}\big).
\end{align*}
In the second equality, we take $u' =u |Y_{\ell_1} - Y_{\ell_2}'| $ and apply the preceding inequalities. Since the density of $Y_\ell$ is uniform bounded (see e.g., \cite[Equation (2.25)]{Hitting-time-density}) and $\mathbb{E} Y_\ell^\beta < \infty$ for any $\beta < \frac{\gamma^2}{4}$, we have $\mathbb{E} |Y_{\ell_1} - Y_{\ell_2}'|^p < \infty$. Therefore, $\lim_{N \rightarrow \infty}J_N^2 = 0$. This, combined with~\eqref{eq:lem4.8-2} and \eqref{eq:lem4.8-3}, yields the desired result. \qedhere
\end{proof}

\begin{lemma}
\label{lem:integral2}
For $\gamma \in (\sqrt{2}, 2)$ and $p \in (\frac{\gamma^2}{4}-1, 0)$, we have
\begin{align*}
&\quad \int_0^\infty \frac{z^{-\frac{4p}{\gamma^2}-1}}{z-1} \bigg(  \frac{e^{-\frac{ i \pi \gamma^2}{4}} z^2 + e^{\frac{i \pi \gamma^2}{4}}  - 2\cos(\frac{\gamma^2}{4}\pi)  z }{ z^{\frac{4}{\gamma^2}}  -1}  + \frac{i \gamma^2}{2} \sin(\frac{\pi \gamma^2}{4})z^{2-\frac{4}{\gamma^2}}  \bigg)dz \\
\smallskip
&= \frac{\pi \gamma^2}{4} \cos ( \frac{\pi \gamma^2}{4})\Big( \cot (\pi (p- \frac{\gamma^2}{4})) - \cot(\pi p) \Big) + \frac{i \pi \gamma^2}{4} \sin ( \frac{\pi \gamma^2}{4})  \Big( 2 \cot (\frac{4 \pi (p+1 )}{\gamma^2} ) - \cot (\pi p) - \cot (\pi (p- \frac{\gamma^2}{4})) \Big).
\end{align*}
\end{lemma}
\begin{proof}
    We first recall an integral formula whose proof can be found in \cite[Lemma 4.14]{NQSZ}:
    \begin{equation}\label{eq:lem-integral1}
    \int_0^\infty \frac{t^{-a} - t^{-b}}{t-1} dt = \pi(\cot(\pi b) - \cot( \pi a)) \quad \mbox{for all } -1 < a, b < 0.
    \end{equation}
    
    Denote the integral in the lemma by $K$. We will compute the real and imaginary parts of $K$ separately. The real part of $K$ is equal to $\cos(\frac{\pi \gamma^2}{4}) \int_0^\infty z^{-\frac{4p}{\gamma^2} -1 } (z-1)/(z^{\frac{4}{\gamma^2}}  -1) dz$. Setting $z = t^{\frac{\gamma^2}{4}}$ and then applying~\eqref{eq:lem-integral1} yields that
    $$
        \mathrm{Re}\, K =  \frac{\gamma^2}{4} \cos ( \frac{\pi \gamma^2}{4}) \int_0^\infty \frac{t^{-p+\frac{\gamma^2}{4}-1} - t^{-p-1}}{t-1} dt = \frac{\pi \gamma^2}{4} \cos ( \frac{\pi \gamma^2}{4})\Big( \cot (\pi (p- \frac{\gamma^2}{4})) - \cot(\pi p) \Big).
    $$
    Furthermore, the imaginary part of $K$ equals 
    $$
        \sin ( \frac{\pi \gamma^2}{4}) \lim_{\epsilon \rightarrow 0} \bigg(\int_{(0, 1-\epsilon) \cup (1+\epsilon, \infty) }  \frac{-z^{-\frac{4p}{\gamma^2}-1} - z^{-\frac{4p}{\gamma^2}}}{ z^{\frac{4}{\gamma^2}}  -1} dz + \frac{\gamma^2}{2} \int_{(0, 1-\epsilon) \cup (1+\epsilon, \infty) }  \frac{z^{1-\frac{4}{\gamma^2}(p+1)}}{z-1} dz \bigg) .
    $$
    Taking $z = t^{\frac{\gamma^2}{4}}$ in the first integral gives
    \begin{align*}
        \frac{1}{ \sin( \frac{\pi \gamma^2}{4})} {\rm Im} \, K &= \lim_{\epsilon \to 0} \frac{\gamma^2}{4} \int_{(0, (1-\epsilon)^{\frac{4}{\gamma^2}}) \cup ((1+\epsilon)^{\frac{4}{\gamma^2}}, \infty) }  \frac{-t^{-p-1} - t^{-p+\frac{\gamma^2}{4}-1} }{t-1} dt + \frac{\gamma^2}{2} \int_{(0, 1-\epsilon) \cup (1+\epsilon, \infty) }  \frac{z^{1-\frac{4}{\gamma^2}(p+1)}}{z-1} dz\\
        &=\lim_{\epsilon \rightarrow 0} \frac{\gamma^2}{4} \int_{(0, 1-\epsilon) \cup (1+ \epsilon, \infty) }  \frac{2 t^{1-\frac{4}{\gamma^2}(p+1)} -t^{-p-1} - t^{-p+\frac{\gamma^2}{4}-1} }{ t -1} dt - K_\epsilon
    \end{align*}
    where the error term  $ K_\epsilon = \int_{((1-\epsilon)^{\frac{4}{\gamma^2}}, 1-\epsilon) \cup (1+\epsilon, (1+\epsilon)^{\frac{4}{\gamma^2}})} \frac{-t^{-p-1} - t^{-p+\frac{\gamma^2}{4}-1} }{t-1} dt $ further equals
    \begin{align*}
        \int_{1 - \frac{4}{\gamma^2} \epsilon +o(\epsilon)}^{1-\epsilon} \Big(\frac{-2}{t-1} + o(\epsilon^{-1})\Big) dt + \int_{1+\epsilon}^{1 + \frac{4}{\gamma^2} \epsilon +o(\epsilon)} \Big( \frac{-2}{t-1} + o(\epsilon^{-1})\Big) dt = o(1) \quad \mbox{as }\epsilon \to 0.
    \end{align*}
    Therefore \({\rm Im} \, K = \frac{\gamma^2}{4} \sin ( \frac{\pi \gamma^2}{4}) \int_0^\infty  \frac{2 t^{1-\frac{4}{\gamma^2}(p+1)} -t^{-p-1} - t^{-p+\frac{\gamma^2}{4}-1} }{ t -1} dt .\) Now applying~\eqref{eq:lem-integral1} yields the desired result. \qedhere
\end{proof}

\subsection{Calculation for the two-pointed quantum disk with weight $\frac{3}{2}\gamma^2-2$}\label{subsec:2ptQD}

Recall $\mathcal{M}^{\rm disk}_{0,2, \bullet}(W)$ from Lemma~\ref{lem:QT(W,2,W)} which is the law of the quantum surface obtained by adding a marked point to the left boundary of a sample from $\mathcal{M}^{\rm disk}_{0,2}(W)$. In this subsection we perform a  calculation concerning the quantum lengths of the three boundary arcs for a sample of $\mathcal{M}^{\rm disk}_{0,2, \bullet}(W)$ with $W = \frac{3}{2}\gamma^2-2$, which is crucial to the proof of Proposition~\ref{prop:CRratio}.
 \begin{lemma}
\label{lem:lem5.8}
Let $W = \frac{3}{2}\gamma^2-2$ and consider a sample from $\mathcal{M}^{\rm disk}_{0,2, \bullet}(W)$. Let $L_{12}$, $L_{13}$, $L_{23}$ be the quantum lengths of the three boundary arcs ordered counterclockwise with $L_{12}$ being the length for the arc between the two weight-$W$  vertices; see Figure~\ref{fig:weldnew2} (right) for an illustration. The following holds for any $\frac{\gamma^2}{4}-1<p<0$ with a real constant $C_2$ depending only on $\gamma$ and $p$:  
    \begin{align*}
        &\quad \mathcal{M}^{\rm disk}_{0,2, \bullet}(W)[( e^{\frac{i \pi \gamma^2}{8} } L_{12} +  e^{-\frac{i \pi \gamma^2}{8} } L_{13} )^{\frac{4p}{\gamma^2}} e^{- L_{23}}] \\
        &= C_2 e^{\frac{i \pi p}{2} - \frac{i \pi \gamma^2}{4}} \Big(  \cot (\pi (p- \frac{\gamma^2}{4})) - \cot(\pi p)  + i  \tan ( \frac{\pi \gamma^2}{4})  \big( 2 \cot (\frac{4\pi (p+1 )}{\gamma^2} ) - \cot (\pi p) - \cot (\pi (p- \frac{\gamma^2}{4}) )\big)  \Big).
    \end{align*}
\end{lemma}
To prove this, 
we need the following fact extracted from
the boundary reflection coefficient of LCFT.
\begin{lemma}
\label{lem:Rfunction}
Let $W = \frac{3}{2}\gamma^2-2$. Let $L$ and $R$ be the left and right boundary lengths of a two-pointed quantum disk sampled from $\mathcal{M}^{\rm disk}_{0,2}(W)$. For any $\mu_1, \mu_2 \in \mathbb{C}$ with ${\rm Re}(\mu_1), {\rm Re}(\mu_2) > 0$, we have
$$
\mathcal{M}^{\rm disk}_{0,2}(W)[e^{-\mu_1 \LL - \mu_2 \RR}-1] = C_1 \frac{\mu_1^2 - 2 \cos(\frac{\pi\gamma^2}{4}) \mu_1 \mu_2 + \mu_2^2}{\mu_1^{4/\gamma^2} + \mu_2^{4/\gamma^2}}\,,
$$
where $C_1$ is a real constant depending only on $\gamma$.
\end{lemma}
\begin{proof} 
Note that $W \in (\frac{\gamma^2}{2}, \gamma^2)$. \cite[Proposition 3.4]{AHS21} gives the identity
\[\Md_{0,2}(W) [e^{-\mu_1 L - \mu_2 R} - 1] = \frac{\gamma}{2(Q-\beta)} R(\beta; \mu_1, \mu_2), \qquad \beta := Q + \frac\gamma2 - \frac W \gamma = \frac 4\gamma - \frac\gamma2\]
where $R(\beta; \mu_1, \mu_2)$ is the so-called \emph{boundary reflection coefficient} for LCFT. \cite[Proposition
3.4]{AHS21} is stated for $\mu_1, \mu_2 \in \mathbb{R}$, but it can be easily extended to our setting by using \cite[Theorem 1.8]{RZ22}. The value of $R(\beta; \mu_1, \mu_2)$ is easily calculated using \cite[Equations (3.2), (3.3), (3.5)]{AHS21}.
\end{proof}

\begin{proof}[Proof of Lemma~\ref{lem:lem5.8}]
    Let $I:= \mathcal{M}^{\rm disk}_{0,2, \bullet}(W) [( e^{\frac{i \pi \gamma^2}{8} } L_{12} +  e^{-\frac{i \pi \gamma^2}{8} } L_{13} )^{\frac{4p}{\gamma^2}} e^{- L_{23}}]$. We first show the absolute integrability of the integral in $I$ by verifying that $\mathcal{M}^{\rm disk}_{0,2, \bullet}(W) [L_{13}^{\frac{4p}{\gamma^2}} e^{- L_{23}}] < \infty$. Recall from Lemma~\ref{lem:QT(W,2,W)} that $\mathcal{M}^{\rm disk}_{0,2, \bullet}(W) = C {\rm QT}(W,2,W)$. Let $W' = -2p \in (0,\frac{\gamma^2}{2})$. By \cite[Proposition 3.6]{AHS21}, the right boundary length of a sample from $\mathcal{M}_{0,2}^{\rm disk}(W')$ follows the law $C 1_{\ell > 0}\ell^{\frac{4p}{\gamma^2}} d \ell$. Therefore, by taking $(W, W_1, W_2, W_3) = (W', \frac{3}{2}\gamma^2-2, 2, \frac{3}{2}\gamma^2-2$) in Theorem~\ref{thm:disk+QT} and using the conformal welding from~\eqref{eqn:disk+QT}, we get that 
    $$
    \mathcal{M}^{\rm disk}_{0,2, \bullet}(W)[L_{13}^{\frac{4p}{\gamma^2}} e^{- L_{23}}] = C {\rm QT}(W,2,W) [L_{13}^{\frac{4p}{\gamma^2}} e^{- L_{23}}] = C' {\rm QT}(W+W', 2+W', W)[e^{- L_{23}'} ],
    $$
    where $L_{23}'$ is the quantum length of the boundary arc between the $2+W'$ and $W$ weight vertices of a sample from ${\rm QT}(W+W', 2+W', W)$. By \cite[Proposition 2.23]{ASY22}, the law of $L_{23}'$ is $C 1_{\ell > 0}\ell^{\frac{4(p+1)}{\gamma^2} - 2} d \ell$. Therefore, $\mathcal{M}^{\rm disk}_{0,2, \bullet}(W) [L_{13}^{\frac{4p}{\gamma^2}} e^{- L_{23}}] < \infty$, and thus, the integral in $I$ is absolutely integrable.
    
    Next, we will calculate $I$. For $z \in \mathbb{C}$ with ${\rm Re}(z) > 0$, we have the identity $z^{\frac{4p}{\gamma^2}} = \frac{1}{\Gamma(-\frac{4p}{\gamma^2})} \int_0^\infty  e^{-zt} t^{-\frac{4p}{\gamma^2}-1} dt$. Using this identity and exchanging the integral in $I$ yields that
    \begin{equation}
    \label{eq:disk-2}
    \begin{aligned}
        I&= \frac{1}{\Gamma(-\frac{4p}{\gamma^2})} \int_0^\infty \mathcal{M}^{\rm disk}_{0,2, \bullet}(W)  \Big[e^{-(e^{\frac{i \pi \gamma^2}{8} }L_{12} + e^{-\frac{i \pi \gamma^2}{8} }L_{13})t } e^{- L_{23}} \Big] t^{-\frac{4p}{\gamma^2}-1} dt .
    \end{aligned}
    \end{equation}

    Recall that the law of $\mathcal{M}^{\rm disk}_{0,2, \bullet}(W)$ is obtained by adding a marked point to the left boundary of a sample from $\mathcal{M}^{\rm disk}_{0,2}(W)$. Therefore, denoting the left and right boundary lengths of a sample from $\Md_{0,2}(W)$ by $(L, R)$, we have
    \begin{align*}
        &\quad \mathcal{M}^{\rm disk}_{0,2, \bullet}(W)  \Big[e^{-(e^{\frac{i \pi \gamma^2}{8} }L_{12} + e^{-\frac{i \pi \gamma^2}{8} }L_{13})t } e^{- L_{23}} \Big] = \mathcal{M}^{\rm disk}_{0,2}(W) \Big[ \int_0^\LL e^{-e^{\frac{i \pi \gamma^2}{8}} t s - (\LL-s)}ds \cdot e^{-e^{-\frac{i \pi \gamma^2}{8}}t \RR} \Big] \\
        &=\mathcal{M}^{\rm disk}_{0,2}(W)\Big[ \frac{1 }{e^{\frac{i \pi \gamma^2}{8} } t - 1} \Big(e^{- \LL -e^{-\frac{i \pi \gamma^2}{8} }t \RR }-e^{-e^{\frac{i \pi \gamma^2}{8} }t \LL -e^{-\frac{i \pi \gamma^2}{8} }t \RR}\Big)    \Big].
    \end{align*} 
    Putting this into~\eqref{eq:disk-2} and then applying Lemma~\ref{lem:Rfunction} with $(\mu_1, \mu_2) = (1, e^{-\frac{i \pi \gamma^2}{8} }t)$ and $(e^{\frac{i \pi \gamma^2}{8} }t, e^{-\frac{i \pi \gamma^2}{8} }t)$ gives
    \begin{align*}
        I&=\frac{C_1}{\Gamma(-\frac{4p}{\gamma^2})} \int_0^\infty  \frac{t^{-\frac{4p}{\gamma^2}-1}}{e^{\frac{i \pi \gamma^2}{8}}t - 1} \bigg( \frac{1 - 2 \cos(\frac{\pi \gamma^2}{4}) e^{-\frac{i \pi \gamma^2}{8}}  t + e^{-\frac{i \pi \gamma^2}{4}} t^2}{1 + e^{-\frac{i \pi }{2}} t^{\frac{4}{\gamma^2}}}  -\frac{\gamma^2}{2} \sin(\frac{\pi \gamma^2}{4})  t^{2-\frac{4}{\gamma^2}}\bigg)dt .
    \end{align*}
    Setting $t = e^{-\frac{i \pi \gamma^2}{8}  } z$, we further have $I$ is equal to
    \begin{equation*}
    \begin{aligned}
        \frac{ - C_1 e^{\frac{i \pi p}{2} - \frac{i \pi \gamma^2}{4}} }{\Gamma(-\frac{4p}{\gamma^2})} \int_{e^{\frac{i \pi \gamma^2}{8}  }\mathbb{R}_+} \frac{z^{-\frac{4p}{\gamma^2}-1}}{z-1} \bigg(  \frac{e^{-\frac{ i \pi \gamma^2}{4}} z^2 + e^{\frac{i \pi \gamma^2}{4}}  - 2\cos(\frac{\gamma^2}{4}\pi)  z }{ z^{\frac{4}{\gamma^2}}  -1}  + \frac{i \gamma^2}{2} \sin(\frac{\pi \gamma^2}{4})z^{2-\frac{4}{\gamma^2}}  \bigg)dz.
    \end{aligned}
    \end{equation*}
    As the function inside the integral is analytic for $z \in \mathbb{C} \backslash (-\infty, 0]$ and decays as $|z|^{-\frac{4(p+1)}{\gamma^2}}$ at infinity, which is faster than $|z|^{-1}$ when $p > \frac{4}{\gamma^2}-1$, we can deform the integral contour from $e^{\frac{i \pi \gamma^2}{8}  }\mathbb{R}_+$ to $\mathbb{R}_+$ without changing its value. Finally, applying Lemma~\ref{lem:integral2} yields the desired result, where we take $C_2 = -C_1 \frac{\pi \gamma^2}{4} \cos(\frac{\pi \gamma^2}{4})  /\Gamma(-\frac{4 p}{\gamma^2}) \in \mathbb{R}$.
\end{proof}

\subsection{Proof of Proposition~\ref{prop:CRratio}}\label{subsec:computation}
 
\begin{figure}
\label{fig:QTtilde}
\centering
\includegraphics[width=0.5\textwidth]{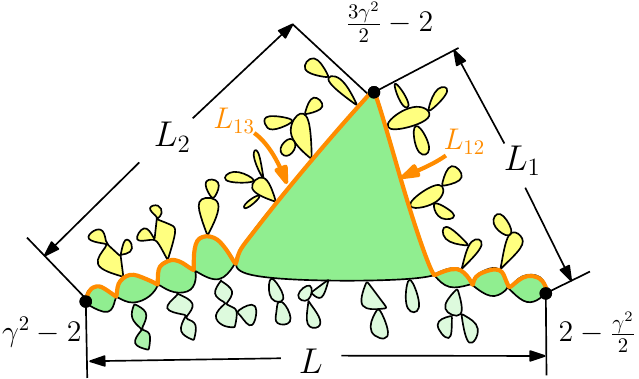}
\caption{Illustration of the decomposition~\eqref{eq:forest-tri}. The green surface corresponds to $\wt{\rm QT}(W_1,W_2,W_3; L_{12},L_{13},L)$. The yellow pieces correspond to $\mathcal{M}_2^{\rm{f.l.}}(L_{12};L_1) $ and $  \mathcal{M}_2^{\rm{f.l.}}(L_{13};L_2)$, and they are glued to the green surface along the two orange boundary arcs as shown in~\eqref{eq:forest-tri}.}\label{fig:weldnew3}
\end{figure}
In this section, we will prove Proposition~\ref{prop:CRratio} based on Proposition~\ref{prop:weld-wind} and results from Sections~\ref{subsec:computation-prelim} and \ref{subsec:2ptQD}. The following lemma expresses the desired ratio in Proposition~\ref{prop:CRratio} in terms of quantities related to $\Md_{0,2,\bullet}(\frac{3}{2}\gamma^2-2)$ that are computed in Lemma~\ref{lem:lem5.8}. 
\begin{lemma}\label{lm:wt-cT_1}
    For any $\alpha \in (Q, \frac{4}{\gamma})$ and $a,b>0$, let
\begin{equation}
\label{eq:def-g}
p = \frac{\gamma}{2} \alpha - 1 \in (\frac{\gamma^2}{4}-1,0) \quad \mbox{and} \quad g(a,b) = \mathrm{Re} [ e^{\frac{i \pi (p+1)}{2}} ( e^{\frac{i \pi \gamma^2}{8} } a +  e^{-\frac{i \pi \gamma^2}{8} } b )^{\frac{4p}{\gamma^2}}  ] > 0.
\end{equation}
Let $W = \frac{3}{2}\gamma^2-2$. Recall  notations  from Lemma~\ref{lem:lem5.8}. We have:
\begin{equation}
\label{eq:transf1}
\frac{\mathbb{E} [{\rm CR}(D_{\widetilde \eta},0)^{2 \Delta_\alpha-2}  \mathbbm{1}_{T^c}]}{\mathbb{E} [{\rm CR}(D_{\widetilde \eta},0)^{2 \Delta_\alpha-2}  \mathbbm{1}_{T}]}  = \frac{\mathcal{M}^{\rm disk}_{0,2, \bullet}(W)[g(L_{12},L_{13})e^{- L_{23}}]}{\mathcal{M}^{\rm disk}_{0,2, \bullet}(W)[g(L_{13},L_{12})e^{- L_{23}}]}.
\end{equation}
\end{lemma}
Lemma~\ref{lm:wt-cT_1} is immediate from the following Lemma~\ref{lem:5.5} concerning the forested quantum triangle with weights:
\begin{equation}
    \label{eq:weight}
    W_1 = \frac{3\gamma^2}{2}-2\,, \quad  W_2 = 2-\frac{\gamma^2}{2}\,, \quad \mbox{and} \quad W_3 = \gamma^2-2 \,.
\end{equation}
Suppose $\cT_1$ is a sample from $\QT(W_1,W_2,W_3)$ and $\cT_1^f$ is obtained by foresting the three boundary arcs of $\cT_1$. Here we abuse notation and  use $L_{12} $ (resp.\ $L_{13}$) to denote the quantum length of the boundary arc of $\cT_1$ between 
the $W_1$ and $W_2$ (resp.\ $W_3$) weight vertices by 
$L_{12}$ (resp.\ $L_{13}$), as used for the boundary lengths of $\mathcal{M}^{\rm disk}_{0,2, \bullet}(W)$ in Lemma~\ref{lem:lem5.8}. 
(It will be clear from the proof of Lemma~\ref{lem:5.5} that this abuse of notation is natural.)
Let $\wt{\cT_1}$ be the quantum surface obtained by only foresting the boundary arc of $\cT_1$ between the weight $W_2$ and $W_3$ vertices, and we write $\wt{\QT}(W_1,W_2,W_3)$ for its law. Then for any $L_1,L_2,L>0$, by Definition~\ref{def:f.s.}, we have
\begin{equation}
\label{eq:forest-tri}
 {\rm QT}^f(W_1,W_2,W_3;L_1,L_2,L) = \iint_{\bbR_+^2} {\rm Weld}(\wt{\rm QT}(W_1,W_2,W_3; L_{12},L_{13},L), {\mathcal{M}_2^{\rm{f.l.}}}(L_{12};L_1),   {\mathcal{M}_2^{\rm{f.l.}}}(L_{13};L_2)) dL_{12}dL_{13}\,
\end{equation}
where $L$ indicates the generalize quantum length of the bottom boundary arc. See Figure~\ref{fig:weldnew3}.
\begin{lemma}
\label{lem:5.5} 
For $L>0$,    the numerator and denominator on the following ratio are both finite:
\begin{equation}
\label{eq:lem5.5} 
\frac{\iint_{\bbR_+^2} | \wt{\rm QT} (W_1,W_2,W_3;L_{12},L_{13},L)|\cdot g(L_{12},L_{13}) dL_{12}dL_{13}}{\iint_{\bbR_+^2} | \wt{\rm QT} (W_1,W_2,W_3;L_{12},L_{13},L)|\cdot g(L_{13},L_{12}) dL_{12}dL_{13}}.
\end{equation}
Moreover, this ratio equals both 
\begin{equation}
 \label{eq:lem4.10-step2}
 \frac{\mathbb{E} [{\rm CR}(D_{\widetilde \eta},0)^{2 \Delta_\alpha-2}  \mathbbm{1}_{T^c}]}{\mathbb{E} [{\rm CR}(D_{\widetilde \eta},0)^{2 \Delta_\alpha-2}  \mathbbm{1}_{T}]}
  \quad 
 \textrm{and}\quad 
\frac{\mathcal{M}^{\rm disk}_{0,2, \bullet}(W)[g(L_{12},L_{13})e^{- L_{23}}]}{\mathcal{M}^{\rm disk}_{0,2, \bullet}(W)[g(L_{13},L_{12})e^{- L_{23}}]}.
\end{equation}
\end{lemma}

\begin{proof}

   \textbf{Step 1: Finiteness in~\eqref{eq:lem5.5}.} By the definition of $g$, Lemma~\ref{lem:lem5.8} gives explicit formulas for $\mathcal{M}^{\rm disk}_{0,2, \bullet}(W)[g(L_{12},L_{13})e^{- L_{23}}]$ and $\mathcal{M}^{\rm disk}_{0,2, \bullet}(W)[g(L_{13},L_{12})e^{- L_{23}}]$ which are in particular finite.
   
   Now we weld a forested line segment along with a quantum disk of weight $\gamma^2-2$ to the bottom boundary arc of a sample from $\wt{\QT}(W_1,W_2,W_3)$ as in Figure~\ref{fig:weldnew2}. By Proposition~\ref{prop:weld:segment} and Theorem~\ref{thm:disk+QT}, this gives a quantum triangle of weights $(\frac{3}{2}\gamma^2-2,2,\frac{3}{2}\gamma^2-2) = (W,2,W)$. By Lemma~\ref{lem:QT(W,2,W)}, it is also equal to $\mathcal{M}^{\rm disk}_{0,2, \bullet}(W)$ up to a constant. Therefore, we have
   \begin{align*}
     &\quad \mathcal{M}^{\rm disk}_{0,2, \bullet}(W)[g(L_{12},L_{13})e^{- L_{23}}] \\
     &=C \iint_{\mathbb{R}_+^3} \Big(\iint_{\bbR_+^2} | \wt{\rm QT} (W_1,W_2,W_3;L_{12},L_{13},L)|\cdot g(L_{12},L_{13}) dL_{12}dL_{13}\Big) |\mathcal{M}_2^{\rm f.l.}(\ell, L)| |\mathcal{M}_{0,2}^{\rm disk}(\ell, L_{23})| e^{- L_{23}} d \ell dL dL_{23}
   \end{align*}
   for some constant $C \in (0, \infty)$ depending only on $\gamma$. From this equality and the fact that $\mathcal{M}^{\rm disk}_{0,2, \bullet}(W)[g(L_{12},L_{13})e^{- L_{23}}]< \infty$, we see that the numerator in~\eqref{eq:lem5.5} is finite for a.e.\ $L$. This extends to any $L>0$ by noting that $| \wt{\rm QT} (W_1,W_2,W_3;L_{12},L_{13},L)|$ is a homogeneous function. That is, there exists $\alpha \in \mathbb{R}$ such that $| \wt{\rm QT} (W_1,W_2,W_3;tL_{12},tL_{13},tL)| = t^\alpha | \wt{\rm QT} (W_1,W_2,W_3;L_{12},L_{13},L)|$ for any $L_{12}, L_{13}, L, t>0$. This holds because using Definition~\ref{def:f.s.} and similar arguments to \cite[Proposition 2.24]{ASY22}, the Laplace transform of $| \wt{\rm QT} (W_1,W_2,W_3;L_{12},L_{13},L)|$ can be explicitly expressed in terms of the LCFT boundary reflection coefficient and three-point function. Both of these functions are homogeneous function (see e.g.\ \cite{RZ22}) hence $|\wt{\rm QT} (W_1,W_2,W_3;L_{12},L_{13},L)|$ is homogeneous. Similarly, the denominator in~\eqref{eq:lem5.5} is also finite for any $L>0$ since $\mathcal{M}^{\rm disk}_{0,2, \bullet}(W)[g(L_{13},L_{12})e^{- L_{23}}]< \infty$. 

   \textbf{Step 2: Equality with the first ratio in~\eqref{eq:lem4.10-step2}.} 
   By Lemma~\ref{lem:M11-f.s.length}, we can disintegrate   \eqref{eq:weld2} over the generalized quantum length $L$ of the boundary of both sides of~\eqref{eq:weld2}. This, together with Lemma~\ref{lem:M11-f.s.length}, gives
\begin{equation}
\label{eq:forest-tri-9}
\begin{aligned}
\frac{\mathbb{E} [{\rm CR}(D_{\widetilde \eta},0)^{2 \Delta_\alpha-2}  \mathbbm{1}_{T^c}]}{\mathbb{E} [{\rm CR}(D_{\widetilde \eta},0)^{2 \Delta_\alpha-2}  \mathbbm{1}_{T}]}  &= \frac{\iint_{L_1>L_2>0} |{\rm QT}^f (W_1,W_2,W_3;L_1,L_2,L)| \cdot |\Mfd_{1,1}(\alpha,\gamma; L_1-L_2)| dL_1dL_2}{\iint_{L_2>L_1>0} |  {\rm QT}^f (W_1,W_2,W_3;L_1,L_2,L)|\cdot |\Mfd_{1,1}(\alpha,\gamma; L_2-L_1)| dL_1dL_2} \\
&=\frac{\iint_{L_1>L_2>0} |  {\rm QT}^f (W_1,W_2,W_3;L_1,L_2,L)| (L_1-L_2)^p dL_1dL_2}{\iint_{L_2>L_1>0} | {\rm QT}^f (W_1,W_2,W_3;L_1,L_2,L)| (L_2-L_1)^p dL_1dL_2}.
\end{aligned}
\end{equation}
Recall the stable L\'{e}vy process $(X_t)_{t\ge0}$ of index $\frac{4}{\gamma^2}$ with only upward jumps from Definition~\ref{def:forested-line}, and $Y_t = \inf\{s\ge0: X_s\le -t\}$. Using~\eqref{eq:forest-tri} and the definition of generalized quantum length, ~\eqref{eq:forest-tri-9} equals 
\begin{equation}\label{eq:forest-tri-10}
    \frac{\iint_{\bbR_+^2} | \wt{\rm QT} (W_1,W_2,W_3;L_{12},L_{13},L)|\cdot \mathbb{E} (Y_{L_{12}} - Y_{L_{13}}')_+^p dL_{12}dL_{13}}{\iint_{\bbR_+^2} | \wt{\rm QT} (W_1,W_2,W_3;L_{12},L_{13},L)| \cdot \mathbb{E} (Y_{L_{13}}' - Y_{L_{12}})_+^p dL_{12}dL_{13}}
\end{equation}
where $(Y_t')_{t\ge0}$ is an independent copy of $(Y_t)_{t\ge0}$. 
The result follows from Lemma~\ref{lem:levy}.

\begin{figure}
\centering
\includegraphics[scale=0.68]{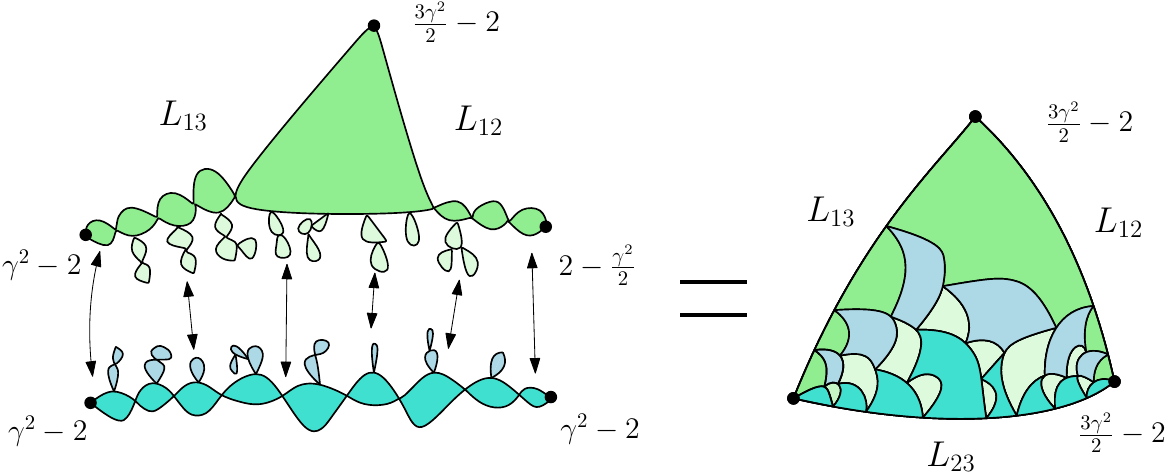}
\caption{An illustration of welding the forested line segment with the weight $\gamma^2-2$ quantum disk to the bottom boundary arc of a sample from $\wt{\QT}(W_1,W_2,W_3)$. The right-hand side corresponds to $\mathcal{M}^{\rm disk}_{0,2, \bullet}(W)$.}\label{fig:weldnew2}
\end{figure}
   \textbf{Step 3: Equality with the second ratio of~\eqref{eq:lem4.10-step2}.} By Step 2, \eqref{eq:lem5.5} does not depend on $L$ since it is equal to the first ratio of~\eqref{eq:lem4.10-step2}. Moreover, the operation of welding a quantum surface to the bottom boundary of a sample from $\wt{\QT}(W_1,W_2,W_3)$ does not change $L_{12}$ and $L_{23}$. Through the conformal welding as in Figure~\ref{fig:weldnew2}, we get the equality. \qedhere
\end{proof}

Now we complete the proof of Proposition~\ref{prop:CRratio} using Lemmas~\ref{lem:lem5.8}  and \ref{lm:wt-cT_1}.
\begin{proof}[Proof of Proposition~\ref{prop:CRratio}]
By Lemma~\ref{lm:wt-cT_1} and the definition of  $g$ from~\eqref{eq:def-g}, we have
\begin{equation}
\label{eq:prop-cr-ratio-1}
\begin{aligned}
&\frac{\mathbb{E} [{\rm CR}(D_{\widetilde \eta},0)^{2 \Delta_\alpha-2}  \mathbbm{1}_{T^c}]}{\mathbb{E} [{\rm CR}(D_{\widetilde \eta},0)^{2 \Delta_\alpha-2}  \mathbbm{1}_{T}]} =  \frac{e^{\frac{i \pi (p+1)}{2}} O_1 + e^{-\frac{i \pi (p+1)}{2}} O_2 }{e^{\frac{i \pi (p+1)}{2}} O_2 + e^{-\frac{i \pi (p+1)}{2}} O_1} \\
&\qquad \qquad \mbox{with} \quad  O_1 = \mathcal{M}^{\rm disk}_{0,2, \bullet}(W)[( e^{\frac{i \pi \gamma^2}{8} } L_{12} +  e^{-\frac{i \pi \gamma^2}{8} } L_{13} )^{\frac{4p}{\gamma^2}} e^{- L_{23}}];\\
& \qquad \qquad \qquad\quad  O_2 = \mathcal{M}^{\rm disk}_{0,2, \bullet}(W)[( e^{\frac{i \pi \gamma^2}{8} } L_{13} +  e^{-\frac{i \pi \gamma^2}{8} } L_{12} )^{\frac{4p}{\gamma^2}} e^{- L_{23}}].
\end{aligned}
\end{equation}
We now calculate the above ratio using Lemma~\ref{lem:lem5.8}. Let 
$$
\mathsf{a} = e^{i \pi \frac{\gamma^2}{4}}, \quad \mathsf{b} = e^{i \pi \frac{4}{\gamma^2}}, \quad \mathsf{c} = e^{i \pi \frac{4 p}{\gamma^2}}, \quad \mbox{and} \quad \mathsf{d} = e^{i \pi p}.
$$
Then, we have $\cot(\pi (p - \frac{\gamma^2}{4})) = i \frac{\mathsf{d}^2 + \mathsf{a}^2}{\mathsf{d}^2 - \mathsf{a}^2}$, $\cot(\pi p) = i \frac{\mathsf{d}^2 + 1}{\mathsf{d}^2 - 1}$, $\tan(\frac{\pi \gamma^2}{4}) = \frac{\mathsf{a}^2-1}{i(\mathsf{a}^2+1)}$, and $\cot (\frac{4 \pi (p+1 )}{\gamma^2} ) = i \frac{\mathsf{b}^2 \mathsf{c}^2 + 1}{\mathsf{b}^2 \mathsf{c}^2 -1} $. Therefore, by Lemma~\ref{lem:lem5.8},
\begin{align*}
    &\quad \frac{\mathcal{M}^{\rm disk}_{0,2, \bullet}(W)[( e^{\frac{i \pi \gamma^2}{8} } L_{12} +  e^{-\frac{i \pi \gamma^2}{8} } L_{13} )^{\frac{4p}{\gamma^2}} e^{- L_{23}}]}{\mathcal{M}^{\rm disk}_{0,2, \bullet}(W)[( e^{-\frac{i \pi \gamma^2}{8} } L_{12} + e^{\frac{i \pi \gamma^2}{8} } L_{13} )^{\frac{4p}{\gamma^2}} e^{- L_{23}}]} \\
    &= \frac{\frac{\sqrt{\mathsf{d}}}{\mathsf{a}}\Big(i \frac{\mathsf{d}^2 + \mathsf{a}^2}{\mathsf{d}^2 - \mathsf{a}^2} - i \frac{\mathsf{d}^2 + 1}{\mathsf{d}^2 - 1} + i \frac{\mathsf{a}^2-1}{i(\mathsf{a}^2+1)} \big( 2 i \frac{\mathsf{b}^2 \mathsf{c}^2 + 1}{\mathsf{b}^2 \mathsf{c}^2 -1} - i \frac{\mathsf{d}^2 + 1}{\mathsf{d}^2 - 1} - i \frac{\mathsf{d}^2 + \mathsf{a}^2}{\mathsf{d}^2 - \mathsf{a}^2} \big)\Big) }{\frac{\mathsf{a}}{\sqrt{\mathsf{d}}}\Big(i \frac{\mathsf{d}^2 + \mathsf{a}^2}{\mathsf{d}^2 - \mathsf{a}^2} - i \frac{\mathsf{d}^2 + 1}{\mathsf{d}^2 - 1} - i \frac{\mathsf{a}^2-1}{i(\mathsf{a}^2+1)} \big( 2 i \frac{\mathsf{b}^2 \mathsf{c}^2 + 1}{\mathsf{b}^2 \mathsf{c}^2 -1} - i \frac{\mathsf{d}^2 + 1}{\mathsf{d}^2 - 1} - i \frac{\mathsf{d}^2 + \mathsf{a}^2}{\mathsf{d}^2 - \mathsf{a}^2} \big)\Big)} = \frac{\mathsf{d}}{\mathsf{a}^2} \frac{\mathsf{d}^4 - (\mathsf{a}^2+1)\mathsf{d}^2 + \mathsf{a}^2\mathsf{b}^2\mathsf{c}^2 }{-\mathsf{d}^4 + (\mathsf{a}^2+1)\mathsf{b}^2\mathsf{c}^2\mathsf{d}^2 - \mathsf{a}^2\mathsf{b}^2\mathsf{c}^2}.
\end{align*}
Note that $e^{\frac{i \pi (p+1)}{2}} = i \sqrt{d}$ and $e^{-\frac{i \pi (p+1)}{2}} = -i \frac{1}{\sqrt{d}}$. After simplifying, \eqref{eq:prop-cr-ratio-1} becomes
\begin{align*}
    \frac{\mathbb{E} [{\rm CR}(D_{\widetilde \eta},0)^{2 \Delta_\alpha-2}  \mathbbm{1}_{T^c}]}{\mathbb{E} [{\rm CR}(D_{\widetilde \eta},0)^{2 \Delta_\alpha-2}  \mathbbm{1}_{T}]} = \frac{-1}{\mathsf{a} + \mathsf{a}^{-1}} \frac{\mathsf{a}^{-2} \mathsf{b}^{-1}  \mathsf{c}^{-1} \mathsf{d}^2 - \mathsf{a}^2 \mathsf{b}  \mathsf{c} \mathsf{d}^{-2}}{\mathsf{a}^{-1} \mathsf{b}^{-1}  \mathsf{c}^{-1} \mathsf{d} - \mathsf{a} \mathsf{b}  \mathsf{c} \mathsf{d}^{-1}} = \frac{-1}{2 \cos(\frac{\pi \gamma^2}{4})}\frac{2i \sin(\pi( - \frac{\gamma^2}{2} - \frac{4}{\gamma^2} - \frac{4 p}{\gamma^2} +2p)}{2i \sin(\pi(- \frac{\gamma^2}{4} - \frac{4}{\gamma^2} - \frac{4p}{\gamma^2} + p))}.
\end{align*}
Recall from~\eqref{eq:def-g} that $p = \frac{\gamma}{2}\alpha - 1$. This proves Proposition~\ref{prop:CRratio}.
\end{proof}

\section{The nested-path exponent for CLE: proof of Theorem~\ref{thm:nested-path-value}}
\label{sec:thm1.4}
For $a>0$, let $\mathrm{Root}(a) $ be the unique solution smaller than $1-\frac{\kappa}{8}$ to the equation~\eqref{eq:sol-nested-path}.
For $\lambda \in \mathbb{R}$, let $\Lambda(\lambda) = \log \mathbb{E}[{\rm CR}(0, \widetilde D)^{-\lambda}|T^c]$. By Theorem~\ref{thm:CR-widetilde-D}, $\Lambda(\lambda)$ is an increasing convex function and $\Lambda(\lambda) = \infty$ for $\lambda \geq 1 - \frac{\kappa}{8}$. 
Moreover, $\mathbb{E}[{\rm CR}(0, \widetilde D)^{-\mathrm{Root}(a) } \mid  T^c] = \frac{1}{a \mathbb{P}[T^c]}$ hence $\Lambda(\mathrm{Root}(a) ) =- \log (a \mathbb{P}[T^c])$. 
To prove Theorem~\ref{thm:nested-path-value}, it suffices to show that
\begin{equation}\label{eq:NP-key}
 \limsup_{\epsilon\to 0} \frac{\log \mathbb{E}[a^{\ell_\epsilon} \mathbbm{1}_{\mathcal{R}_\epsilon}]}{\log \epsilon}\le \Lambda^{-1}(- \log (a \mathbb{P}[T^c]))\quad\textrm{and}\quad 
  \liminf_{\epsilon\to 0} \frac{\log \mathbb{E}[a^{\ell_\epsilon} \mathbbm{1}_{\mathcal{R}_\epsilon}]}{\log \epsilon}\ge \Lambda^{-1}(- \log (a \mathbb{P}[T^c])).
\end{equation}

Recall the notion of open circuit in the definition of nested-path exponent for CLE above~\eqref{eq:def-nested-path-continuum}.
We define a sequence of nested open circuits  $g_0, g_1, \ldots, g_\tau$ as follows. Let $g_0 = \partial \bbD$ be the zeroth open circuit. If $T = \{\Loop \cap \partial \bbD \neq \emptyset\}$ occurs, we stop the exploration and set $\tau=0$. Otherwise, if ${\Loop} \cap \partial \bbD =  \emptyset$, we let $g_1 = \partial \widetilde D$ be the first open circuit. By the domain Markov property, conditioning on $\widetilde D$ we have an independent CLE inside. Inductively, given the $k$-th open circuit $g_k$, which is a simple loop surrounding the origin, if it intersects $\Loop$, we stop and let $\tau=k$. Otherwise, we iterate the procedure to find the $(k+1)$-th open circuit $g_{k+1}$ surrounding the origin.
For $0 \leq i \leq \tau$, let ${\rm CR}(0, g_i)$ be the conformal radius of the domain enclosed by $g_i$ as seen from the origin. By the domain Markov property of CLE, the law of $\{ {\rm CR}(0, g_i) \}_{0 \leq i \leq \tau}$ can be described as follows. Let $X_1, X_2, \ldots$ be a sequence of i.i.d.\ random variables sampled from $\mathbb{P}[{\rm CR}(0, \widetilde D) \in \cdot | T^c]$, and let $\sigma$ be an independent random variable sampled from the geometric distribution with success probability $\mathbb{P}[T^c]$. Namely, $\mathbb{P}[\sigma \geq k] = \mathbb{P}[T^c]^{k}$ for any integer $k \geq 0$. Then we have 
\begin{equation}
\label{eq:law-sequence-CR}
    \Big( \frac{{\rm CR}(0, g_1)}{{\rm CR}(0, g_0)}, \frac{{\rm CR}(0, g_2)}{{\rm CR}(0, g_1)}, \ldots, \frac{{\rm CR}(0, g_{\tau})}{{\rm CR}(0, g_{\tau-1})} \Big) \overset{d}{=} (X_1, \ldots, X_\sigma).
\end{equation}
This is because, given the event $\tau \geq i$ and $g_0,g_1,\ldots, g_i$, with probability of $1-\mathbb{P}[T^c]$ we have $\tau = i$ and the sequence terminates; with  probability $\mathbb{P}[T^c]$ we have $\tau \geq i+1$ and $g_{i+1}$ is the open circuit defined within the domain enclosed by $g_i$. In particular, $g_{i+1}$ has the same law as $f(\partial \widetilde D)$ conditioned on $T^c$, where $f$ is the conformal map from $\bbD$ to the domain enclosed by $g_i$, fixing $0$. Therefore, $\frac{{\rm CR}(0, g_{i+1})}{{\rm CR}(0, g_i)}$ has the same law as $\mathbb{P}[{\rm CR}(0, \widetilde D) \in \cdot | T^c]$. The right-hand side of~\eqref{eq:law-sequence-CR} is sampled in the same way, and thus, \eqref{eq:law-sequence-CR} holds.

We now prove~\eqref{eq:NP-key} in the case  $0 < a \mathbb{P}[T^c] < 1$. Let $u = \log \frac{1}{\epsilon}$ and $c_1 = \frac{1}{\mathbb{E} [-\log {\rm CR}(0, \widetilde D)|T^c]}$. Let $\Lambda^*$ be the Legendre transform of $\Lambda$, namely $\Lambda^*(t) = \sup_{\lambda \in \mathbb{R}} \{ \lambda t  - \Lambda(\lambda) \}$ for $t \in \mathbb{R}$.\footnote{Here we record some properties of $\Lambda^*(t)$: for $t \leq 0$, $\Lambda^*(t) = \infty$; for $0 < t < c_1^{-1}$, $\Lambda^*(t) > 0$ with the supremum taken at negative $\lambda$; at $t = c_1^{-1}$, $\Lambda^*(t) = 0$; and for $t > c_1^{-1}$, $\Lambda^*(t) > 0$ with the supremum taken at positive $\lambda$. The first property follows from $\Lambda(\lambda) \sim - \sqrt{|\lambda|}$ as $\lambda \rightarrow -\infty$. When $t>0$, the supremum is taken at the solution to $\Lambda'(\lambda) = t$. Note that $\Lambda'(\lambda)$ is an increasing function on $(-\infty, 1 - \frac{\kappa}{8})$ which takes value 0 at $-\infty$ and $\infty$ at $1 -\frac{\kappa}{8}$. Hence, the supremum is taken at negative $\lambda$ if $t < \Lambda'(0) = c_1^{-1}$; positive $\lambda$ if $t > \Lambda'(0)$. Since $\Lambda(0) = 0$, we always have $\Lambda^*(t) \geq 0$, and $\Lambda^*(t) = 0$ if and only if $t = c_1^{-1}$.} Fix $0 < t < c_1$. Applying Cram\'er's theorem from \cite[Theorem 2.2.3]{dembo-ld} with the random variable $\frac{1}{\log X_i}$, $n = \lfloor tu \rfloor$, and the same $\Lambda^*$, we have
    \begin{equation}
    \label{eq:large-deviation-nested}
    \mathbb{P}\Big[\sum_{i=1}^{\lfloor tu \rfloor} \log \frac{1}{X_i} > u \Big] = \exp(- t \Lambda^*(t^{-1})u  + o(u) ) \quad \mbox{as } \epsilon \rightarrow 0.
    \end{equation}
    Furthermore, for any fixed $\delta>0$, as $\epsilon$ tends to 0, we have
    \begin{equation}
    \label{eq:large-deviation-nested-two}
    \mathbb{P}\Big[\sum_{i=1}^{\lfloor (1 - \delta) t u \rfloor} \log  \frac{1}{X_i} <  u - \log 4, \quad \sum_{i=1}^{\lfloor tu \rfloor} \log  \frac{1}{X_i} > u \Big] = \exp(- t \Lambda^*(t^{-1})u  + o(u) ).
    \end{equation}
    (The upper bound follows directly from~\eqref{eq:large-deviation-nested} and the lower bound follows by noting that $\mathbb{P}[\sum_{i=1}^{\lfloor (1 - \delta) t u \rfloor} \log  \frac{1}{X_i} \geq  u - \log 4] = \exp(-(1 - \delta) t u \Lambda^*(\frac{1}{1 - \delta} t^{-1}) + o(u))$. Using the convexity of $\Lambda^*$ and $\Lambda^*(c_1^{-1}) = 0$, we see that this probability is exponentially smaller than the right-hand side of~\eqref{eq:large-deviation-nested-two}. Hence, by~\eqref{eq:large-deviation-nested}, \eqref{eq:large-deviation-nested-two} holds.) By the definition of  $g_0, g_1, \ldots, g_\tau$, if the Euclidean distance between $g_\tau$ and $0$ is smaller than $\epsilon$, then the event $\mathcal{R}_\epsilon$ occurs and
    $\ell_\epsilon $ counts the number of open circuits in $g_1, \ldots, g_\tau$ that surround $\epsilon \bbD$. Therefore, by the Koebe $1/4$ theorem, on the event $\tau \geq \lfloor tu \rfloor$ and ${\rm CR}(0, g_{\lfloor tu \rfloor}) < \epsilon < \frac{1}{4} {\rm CR}(0, g_{\lfloor (1-\delta) tu \rfloor})$, the event $\mathcal{R}_\epsilon$ occurs and $ \lfloor (1-\delta) tu \rfloor  \leq \ell_\epsilon < \lfloor tu \rfloor $.   By~\eqref{eq:law-sequence-CR} and \eqref{eq:large-deviation-nested-two}, we obtain: 
    \begin{align*}
        \mathbb{E}[a^{\ell_\epsilon} \mathbbm{1}_{\mathcal{R}_\epsilon}] &\geq \min \{ a^{\lfloor (1-\delta) tu \rfloor } , a^{\lfloor tu \rfloor } \} \times \mathbb{P}\Big[\tau \geq \lfloor tu \rfloor , {\rm CR}(0, g_{\lfloor tu \rfloor}) < \epsilon < \frac{1}{4} {\rm CR}(0, g_{\lfloor (1-\delta) tu \rfloor}) \Big] \\
        &= \min \{ a^{\lfloor (1-\delta) tu \rfloor } , a^{\lfloor tu \rfloor } \} \times \mathbb{P}\Big[\sigma \geq \lfloor tu \rfloor , \sum_{i=1}^{\lfloor tu \rfloor} \log  \frac{1}{X_i} > u,  \sum_{i=1}^{\lfloor (1 - \delta) t u \rfloor} \log  \frac{1}{X_i} <  u - \log 4 \Big]\\
        &= \min \{ a^{-\delta t u}, 1\} \exp\big(-t \Lambda^*(t^{-1}) u  +\log (a \mathbb{P}[T^c]) \cdot t u  + o(u) \big). 
    \end{align*}
    First taking $\delta$ to 0 and then taking the supremum over $t \in (0,c_1)$ yields that
    \begin{equation}
    \label{eq:thm1.4-lower-bound}
    \begin{aligned}
        \liminf_{\epsilon \rightarrow 0} \frac{1}{u} \log \mathbb{E}[a^{\ell_\epsilon} \mathbbm{1}_{\mathcal{R}_\epsilon}] \geq \sup_{t \in (0,c_1)} \{ \log (a \mathbb{P}[T^c]) \cdot t -t \Lambda^*(t^{-1}) \}.
    \end{aligned}
    \end{equation}
    Now we show that 
    \begin{equation}
        \label{eq:nested-path-01}
        \sup_{t \in (0,c_1)} \{ \log (a \mathbb{P}[T^c]) \cdot t -t \Lambda^*(t^{-1}) \} = -  \Lambda^{-1}(- \log (a \mathbb{P}[T^c])).
    \end{equation}Let $r(t)$ be the Legendre transform of the convex function $-\Lambda^{-1}(-\lambda)$. Then $r(t) = t \Lambda^*(t^{-1})$ for $t > 0$ and $r(t) = \infty$ for $t \leq 0$. Since the iteration of the Legendre transform is identity (see e.g.\ \cite[Lemma 4.5.8]{dembo-ld}), we obtain that $\sup_{t \in \mathbb{R}} \{ \log (a \mathbb{P}[T^c]) \cdot t - r(t) \} = -  \Lambda^{-1}(- \log (a \mathbb{P}[T^c]))$. Since $\log (a \mathbb{P}[T^c]) < 0$, the supremum in the former term is taken when $t \in (0, c_1)$ and thus~\eqref{eq:nested-path-01} holds. Combining~\eqref{eq:thm1.4-lower-bound} and ~\eqref{eq:nested-path-01} yields the first inequality in~\eqref{eq:NP-key}.

    The second inequality in~\eqref{eq:NP-key} can be obtained using the large deviation principle and \eqref{eq:nested-path-01}. Fix $\delta' >0$. By definition, ${\rm CR}(0, g_{\ell_\epsilon}) \geq \epsilon$. We first consider the case of ${\rm CR}(0, g_{\ell_\epsilon}) \in [\epsilon, \epsilon^{1- \delta'}]$. In this case, we have
    \begin{align*}
    \mathbb{E}\Big[a^{\ell_\epsilon} \mathbbm{1}_{\{ {\rm CR}(0, g_{\ell_\epsilon}) \in [\epsilon, \epsilon^{1- \delta'}] \} }\Big] &\leq  \mathbb{E}\Big[a^{\ell_\epsilon} \mathbbm{1}_{\{ {\rm CR}(0, g_{\ell_\epsilon}) \leq \epsilon^{1- \delta'}, \ell_\epsilon \leq (1 - \delta')c_1 u \} }\Big] + \mathbb{E}\Big[a^{\ell_\epsilon} \mathbbm{1}_{\{ \ell_\epsilon > (1 - \delta')c_1 u \} }\Big]\\
    &= \mathbb{E}\Big[a^{\sigma} \mathbbm{1}_{\{ \sum_{i=1}^{\sigma} \log \frac{1}{X_i} \geq (1- \delta') u , \sigma \leq (1 - \delta')c_1 u \} }\Big] + \mathbb{E}\Big[a^{\sigma} \mathbbm{1}_{\{ \sigma > (1 - \delta')c_1 u \} }\Big].
    \end{align*}
    We bound the first term by decomposing the possible values of $\sigma / u $ into small intervals, whose length tends to zero with $\epsilon$, and then applying Cram\'er's theorem similarly to~\eqref{eq:large-deviation-nested}. Using $ a \mathbb{P}[T^c] < 1$, the second term is bounded by $\frac{1}{1-a \mathbb{P}[T^c]}(a \mathbb{P}[T^c])^{(1 - \delta')c_1 u}$. Therefore, by~\eqref{eq:nested-path-01}, we have
    \begin{equation}
    \label{eq:sec5-proof-1}
    \begin{aligned}
        \mathbb{E}\Big[a^{\ell_\epsilon} \mathbbm{1}_{\{ {\rm CR}(0, g_{\ell_\epsilon-1}) \in [\epsilon, \epsilon^{1- \delta'}] \} }\Big] &\leq \exp\big( u \cdot \sup_{t \in (0, (1-\delta')c_1]}\{ \log (a \mathbb{P}[T^c]) \cdot t -t \Lambda^*((1 - \delta')t^{-1}) \} + o(u)\big) \\
        &= \exp \big(- u ( 1-\delta') \Lambda^{-1}(- \log (a \mathbb{P}[T^c])) +o(u) ).
    \end{aligned}
    \end{equation}
    
    Now we consider the case of ${\rm CR}(0, g_{\ell_\epsilon}) \in [\epsilon^{1 - \delta'}, \epsilon^{1- 2\delta'}]$. Similar to before, we have:
    \begin{equation}\label{eq:upper-bound-1}
    \mathbb{E}\Big[a^{\ell_\epsilon} \mathbbm{1}_{\{ {\rm CR}(0, g_{\ell_\epsilon}) \in [\epsilon^{1 - \delta'}, \epsilon^{1- 2\delta'}] \} }\Big] \leq \exp \big(- u ( 1-2 \delta') \Lambda^{-1}(- \log (a \mathbb{P}[T^c])) +o(u) ).
    \end{equation}
    Now we show that 
    \begin{equation}\label{eq:upper-bound-2}
    \mathbb{P}\big[\mathcal{R}_\epsilon | \ell_\epsilon, \{{\rm CR}(0, g_{\ell_\epsilon}) \in [\epsilon^{1- \delta'}, \epsilon^{1- 2 \delta'}]\}\big] \leq \exp \big(- u \delta' \Lambda^{-1}(- \log (a \mathbb{P}[T^c])) +o(u)).
    \end{equation}
    To make the event $\mathcal{R}_\epsilon$ happen, we know that either ${\rm CR}(0, g_{\ell_\epsilon + 1}) < \epsilon^{\delta'} {\rm CR}(0, g_{\ell_\epsilon})$, or $\tau = \ell_\epsilon$ and the CLE loop $\Loop$ has a conformal radius of at most $\epsilon^{\delta'} {\rm CR}(0, g_{\ell_\epsilon})$. By the first claim in Theorem~\ref{thm:CR} and Theorem~\ref{thm:CR-widetilde-D}, on the event ${\rm CR}(0, g_{\ell_\epsilon}) \in [\epsilon^{1- \delta'}, \epsilon^{1- 2 \delta'}]$ and given $\ell_\epsilon$, the probabilities of both events are at most $C_\eta \epsilon^{(1-\frac{\kappa}{8}) \delta' - \eta}$ for any $\eta>0$. Using the fact that $\Lambda^{-1}(- \log (a \mathbb{P}[T^c])) < 1-\frac{\kappa}{8}$, we obtain~\eqref{eq:upper-bound-2}.
    
    Combining~\eqref{eq:upper-bound-1} and \eqref{eq:upper-bound-2}, we further have
    \begin{align*}
        \mathbb{E}\Big[a^{\ell_\epsilon} \mathbbm{1}_{\mathcal{R}_\epsilon \cap \{{\rm CR}(0, g_{\ell_\epsilon}) \in [\epsilon^{1- \delta'}, \epsilon^{1- 2 \delta'}] \}} \Big] &= \mathbb{E}\Big[a^{\ell_\epsilon} \mathbbm{1}_{\{{\rm CR}(0, g_{\ell_\epsilon}) \in [\epsilon^{1- \delta'}, \epsilon^{1- 2 \delta'}] \}} \mathbb{P}\big[\mathcal{R}_\epsilon | \ell_\epsilon, \{{\rm CR}(0, g_{\ell_\epsilon}) \in [\epsilon^{1- \delta'}, \epsilon^{1- 2 \delta'}]\}\big] \Big] \\
        &\leq \exp \big(- u ( 1-\delta') \Lambda^{-1}(- \log (a \mathbb{P}[T^c])) +o(u) ).
    \end{align*}
    The same inequality holds for the case of ${\rm CR}(0, g_{\ell_\epsilon}) \in [\epsilon^{1- n \delta'}, \epsilon^{1- (n+1) \delta'}]$ for any $2 \leq n \leq \lfloor \frac{1}{\delta'} \rfloor$. Summing all these inequalities together and taking $\delta'$ to 0 yields the second inequality in~\eqref{eq:NP-key}.

    The case $ a\mathbb{P}[T^c] > 1$ can be treated similarly, as we now elaborate. Fix $t > c_1$ and $\delta''>0$ which will tend to zero in the end. Similar to~\eqref{eq:large-deviation-nested-two}, by Cram\'er's theorem, we have
    $$
    \mathbb{P}\Big[\sum_{i=1}^{\lfloor t u \rfloor} \log  \frac{1}{X_i} <  u - \log 4, \quad \sum_{i=1}^{\lfloor (1+\delta'')tu \rfloor} \log \frac{1}{X_i} > u \Big] = \exp(- t \Lambda^*(t^{-1})u  + o(u) ) \quad \mbox{as } \epsilon \rightarrow 0.
    $$
    Moreover, on the event $\tau \geq \lfloor (1+\delta'')tu \rfloor$ and $\frac{1}{4} {\rm CR}(0, g_{\lfloor tu \rfloor}) > \epsilon >  {\rm CR}(0, g_{\lfloor (1+\delta'') tu \rfloor})$, the event $\mathcal{R}_\epsilon$ occurs and $ \lfloor  tu \rfloor  \leq \ell_\epsilon < \lfloor (1+\delta'')  tu \rfloor $. Therefore, together with~\eqref{eq:law-sequence-CR}, we get
    \begin{equation}
    \label{eq:sec5-large-a-1}
    \begin{aligned}
    \mathbb{E}[a^{\ell_\epsilon} \mathbbm{1}_{\mathcal{R}_\epsilon}] &\geq \min \{ a^{\lfloor tu \rfloor } , a^{\lfloor (1+\delta'') tu \rfloor } \} \times \mathbb{P}\Big[\sigma \geq \lfloor (1+\delta'') tu \rfloor , \sum_{i=1}^{\lfloor tu \rfloor} \log \frac{1}{X_i} < u - \log 4,  \sum_{i=1}^{\lfloor (1 + \delta'') t u \rfloor} \log \frac{1}{X_i} >  u  \Big]\\
    &= \min \{ a^{-\delta t u}, 1\} \times \mathbb{P}[T^c]^{\delta'' tu} \times \exp\big(-t \Lambda^*(t^{-1}) u  +\log (a \mathbb{P}[T^c]) \cdot t u  + o(u) \big). 
    \end{aligned}
    \end{equation}
    We have the following variant of \eqref{eq:nested-path-01} in the case when $ a\mathbb{P}[T^c] > 1$ and the proof follows verbatim the same argument:
    $$
    \sup_{t \in (c_1, \infty)} \{ \log (a \mathbb{P}[T^c]) \cdot t -t \Lambda^*(t^{-1}) \} = -  \Lambda^{-1}(- \log (a \mathbb{P}[T^c])).
    $$
    Similar to before, first taking $\delta''$ to 0 and then taking the supremum of the right side of~\eqref{eq:sec5-large-a-1} over $t \in (c_1, \infty)$ yields the first inequality in~\eqref{eq:NP-key}. The proof for the second inequality is similar to before and we omit it here. Finally, the case  $a =  \mathbb{P}[T^c]^{-1}$ follows by taking the limit. This concludes the proof.

\bibliographystyle{alpha}
\bibliography{theta}

\newcommand{\etalchar}[1]{$^{#1}$}
\begin{thebibliography}{DCMT21}

\bibitem[AG21]{AG19}
Morris Ang and Ewain Gwynne.
\newblock Liouville quantum gravity surfaces with boundary as matings of trees.
\newblock {\em Ann. Inst. Henri Poincar\'{e} Probab. Stat.}, 57(1):1--53, 2021.

\bibitem[AG23]{AG23a}
Morris {Ang} and Ewain {Gwynne}.
\newblock {Cutting $\gamma$-Liouville quantum gravity by Schramm-Loewner
  evolution for $\kappa \not\in \{\gamma^2, 16/\gamma^2\}$}.
\newblock {\em arXiv e-prints}, page arXiv:2310.11455, October 2023.

\bibitem[AHS17]{AHS17}
Juhan Aru, Yichao Huang, and Xin Sun.
\newblock Two perspectives of the 2{D} unit area quantum sphere and their
  equivalence.
\newblock {\em Comm. Math. Phys.}, 356(1):261--283, 2017.

\bibitem[AHS23]{AHS20}
Morris Ang, Nina Holden, and Xin Sun.
\newblock Conformal welding of quantum disks.
\newblock {\em Electron. J. Probab.}, 28:Paper No. 52, 50, 2023.

\bibitem[AHS24]{AHS21}
Morris Ang, Nina Holden, and Xin Sun.
\newblock Integrability of {SLE} via conformal welding of random surfaces.
\newblock {\em Comm. Pure Appl. Math.}, 77(5):2651--2707, 2024.

\bibitem[AHSY23]{AHSY23}
Morris {Ang}, Nina {Holden}, Xin {Sun}, and Pu~{Yu}.
\newblock {Conformal welding of quantum disks and multiple SLE: the non-simple
  case}.
\newblock {\em arXiv e-prints}, page arXiv:2310.20583, October 2023.

\bibitem[ALS22]{ALS-CLE4}
Juhan Aru, Titus Lupu, and Avelio Sep\'ulveda.
\newblock Extremal distance and conformal radius of a {$\rm CLE_4$} loop.
\newblock {\em Ann. Probab.}, 50(2):509--558, 2022.

\bibitem[ARS21]{ARS21}
Morris {Ang}, Guillaume {Remy}, and Xin {Sun}.
\newblock {FZZ formula of boundary Liouville CFT via conformal welding}.
\newblock {\em arXiv e-prints}, page arXiv:2104.09478, April 2021.

\bibitem[ARS22]{ARS2022moduli}
Morris {Ang}, Guillaume {Remy}, and Xin {Sun}.
\newblock {The moduli of annuli in random conformal geometry}.
\newblock {\em arXiv e-prints}, page arXiv:2203.12398, March 2022.

\bibitem[ARSZ23]{ARSZ23}
Morris {Ang}, Guillaume {Remy}, Xin {Sun}, and Tunan {Zhu}.
\newblock {Derivation of all structure constants for boundary Liouville CFT}.
\newblock {\em arXiv e-prints}, page arXiv:2305.18266, May 2023.

\bibitem[AS21]{AS21}
Morris {Ang} and Xin {Sun}.
\newblock {Integrability of the conformal loop ensemble}.
\newblock {\em arXiv e-prints}, page arXiv:2107.01788, July 2021.

\bibitem[ASY22]{ASY22}
Morris {Ang}, Xin {Sun}, and Pu~{Yu}.
\newblock {Quantum triangles and imaginary geometry flow lines}.
\newblock {\em arXiv e-prints}, page arXiv:2211.04580, November 2022.

\bibitem[AY23]{AY23}
Morris {Ang} and Pu~{Yu}.
\newblock {Reversibility of whole-plane SLE for $\kappa > 8$}.
\newblock {\em arXiv e-prints}, page arXiv:2309.05176, September 2023.

\bibitem[BH19]{benoist-hongler-cle3}
St\'{e}phane Benoist and Cl\'{e}ment Hongler.
\newblock The scaling limit of critical {I}sing interfaces is {CLE$_3$}.
\newblock {\em Ann. Probab.}, 47(4):2049--2086, 2019.

\bibitem[BM17]{BM17}
J\'{e}r\'{e}mie Bettinelli and Gr\'{e}gory Miermont.
\newblock Compact {B}rownian surfaces {I}: {B}rownian disks.
\newblock {\em Probab. Theory Related Fields}, 167(3-4):555--614, 2017.

\bibitem[BPZ84]{BPZ84}
A.~A. Belavin, A.~M. Polyakov, and A.~B. Zamolodchikov.
\newblock Infinite conformal symmetry in two-dimensional quantum field theory.
\newblock {\em Nuclear Phys. B}, 241(2):333--380, 1984.

\bibitem[Car02]{Car02}
John Cardy.
\newblock Crossing formulae for critical percolation in an annulus.
\newblock {\em J. Phys. A}, 35(41):L565--L572, 2002.

\bibitem[Car06]{Car06}
John Cardy.
\newblock The {${\rm O}(n)$} model on the annulus.
\newblock {\em J. Stat. Phys.}, 125(1):1--21, 2006.

\bibitem[Cer21]{cercle2021unit}
Baptiste Cercl\'{e}.
\newblock Unit boundary length quantum disk: a study of two different
  perspectives and their equivalence.
\newblock {\em ESAIM Probab. Stat.}, 25:433--459, 2021.

\bibitem[CK14]{CK13looptree}
Nicolas Curien and Igor Kortchemski.
\newblock Random stable looptrees.
\newblock {\em Electron. J. Probab.}, 19:no. 108, 35, 2014.

\bibitem[CN08]{camia-newman-sle6}
Federico Camia and Charles~M. Newman.
\newblock {${\rm SLE}_6$} and {${\rm CLE}_6$} from critical percolation.
\newblock In {\em Probability, geometry and integrable systems}, volume~55 of
  {\em Math. Sci. Res. Inst. Publ.}, pages 103--130. Cambridge Univ. Press,
  Cambridge, 2008.

\bibitem[DCMT21]{DCMT21}
Hugo Duminil-Copin, Ioan Manolescu, and Vincent Tassion.
\newblock Planar random-cluster model: fractal properties of the critical
  phase.
\newblock {\em Probab. Theory Related Fields}, 181(1-3):401--449, 2021.

\bibitem[DKK{\etalchar{+}}20]{DC-invariance}
Hugo {Duminil-Copin}, Karol {Kajetan Kozlowski}, Dmitry {Krachun}, Ioan
  {Manolescu}, and Mendes {Oulamara}.
\newblock {Rotational invariance in critical planar lattice models}.
\newblock {\em arXiv e-prints}, page arXiv:2012.11672, December 2020.

\bibitem[DKRV16]{DKRV16}
Fran\c~cois David, Antti Kupiainen, R\'emi Rhodes, and Vincent Vargas.
\newblock Liouville quantum gravity on the {R}iemann sphere.
\newblock {\em Comm. Math. Phys.}, 342(3):869--907, 2016.

\bibitem[DMS21]{DMS14}
Bertrand Duplantier, Jason Miller, and Scott Sheffield.
\newblock Liouville quantum gravity as a mating of trees.
\newblock {\em Ast\'erisque}, (427):viii+257, 2021.

\bibitem[dN83]{dN83}
Marcel den Nijs.
\newblock Extended scaling relations for the magnetic critical exponents of the
  {P}otts model.
\newblock {\em Phys. Rev. B (3)}, 27(3):1674--1679, 1983.

\bibitem[DO94]{DO94}
H.~Dorn and H.-J. Otto.
\newblock Two- and three-point functions in {L}iouville theory.
\newblock {\em Nuclear Phys. B}, 429(2):375--388, 1994.

\bibitem[DS11]{DS11}
Bertrand Duplantier and Scott Sheffield.
\newblock Liouville quantum gravity and {KPZ}.
\newblock {\em Invent. Math.}, 185(2):333--393, 2011.

\bibitem[DZ10]{dembo-ld}
Amir Dembo and Ofer Zeitouni.
\newblock {\em Large deviations techniques and applications}, volume~38 of {\em
  Stochastic Modelling and Applied Probability}.
\newblock Springer-Verlag, Berlin, 2010.
\newblock Corrected reprint of the second (1998) edition.

\bibitem[FK72]{FK72}
C.~M. Fortuin and P.~W. Kasteleyn.
\newblock On the random-cluster model. {I}. {I}ntroduction and relation to
  other models.
\newblock {\em Physica}, 57:536--564, 1972.

\bibitem[GKRV20]{GKRV20_bootstrap}
Colin {Guillarmou}, Antti {Kupiainen}, R{\'e}mi {Rhodes}, and Vincent {Vargas}.
\newblock {Conformal bootstrap in Liouville Theory}.
\newblock {\em arXiv e-prints}, page arXiv:2005.11530, May 2020.

\bibitem[GKRV21]{GKRV21_Segal}
Colin {Guillarmou}, Antti {Kupiainen}, R{\'e}mi {Rhodes}, and Vincent {Vargas}.
\newblock {Segal's axioms and bootstrap for Liouville Theory}.
\newblock {\em arXiv e-prints}, page arXiv:2112.14859, December 2021.

\bibitem[GM21]{gwynne2021convergence}
Ewain Gwynne and Jason Miller.
\newblock Percolation on uniform quadrangulations and {$\rm SLE_6$} on
  {$\sqrt{8/3}$}-{L}iouville quantum gravity.
\newblock {\em Ast\'erisque}, (429):vii+242, 2021.

\bibitem[GMQ21]{gmq-cle-inversion}
Ewain Gwynne, Jason Miller, and Wei Qian.
\newblock Conformal invariance of {${\rm CLE}_\kappa$} on the {R}iemann sphere
  for {$\kappa\in(4, 8)$}.
\newblock {\em Int. Math. Res. Not. IMRN}, (23):17971--18036, 2021.

\bibitem[GP20]{GP-bubble-connect}
Ewain Gwynne and Joshua Pfeffer.
\newblock Connectivity properties of the adjacency graph of {${\rm
  SLE}_{\kappa}$} bubbles for {$\kappa\in(4,8)$}.
\newblock {\em Ann. Probab.}, 48(3):1495--1519, 2020.

\bibitem[HS23]{HS19}
Nina Holden and Xin Sun.
\newblock Convergence of uniform triangulations under the {C}ardy embedding.
\newblock {\em Acta Math.}, 230(1):93--203, 2023.

\bibitem[Kem17]{Kem17SLE}
Antti Kemppainen.
\newblock {\em Schramm-{L}oewner evolution}, volume~24 of {\em SpringerBriefs
  in Mathematical Physics}.
\newblock Springer, Cham, 2017.

\bibitem[KL22]{KL-Fuzzy}
Laurin {K{\"o}hler-Schindler} and Matthis {Lehmkuehler}.
\newblock {The fuzzy Potts model in the plane: Scaling limits and arm
  exponents}.
\newblock {\em arXiv e-prints}, page arXiv:2209.12529, September 2022.

\bibitem[KMS23]{kavvadias2023conformal}
Konstantinos {Kavvadias}, Jason {Miller}, and Lukas {Schoug}.
\newblock {Conformal removability of non-simple Schramm-Loewner evolutions}.
\newblock {\em arXiv e-prints}, page arXiv:2302.10857, February 2023.

\bibitem[KRV20]{KRV20}
Antti Kupiainen, R\'emi Rhodes, and Vincent Vargas.
\newblock Integrability of {L}iouville theory: proof of the {DOZZ} formula.
\newblock {\em Ann. of Math. (2)}, 191(1):81--166, 2020.

\bibitem[KS19]{Kemppainen-Smirnov-19}
Antti Kemppainen and Stanislav Smirnov.
\newblock Conformal invariance of boundary touching loops of {FK} {I}sing
  model.
\newblock {\em Comm. Math. Phys.}, 369(1):49--98, 2019.

\bibitem[KW16]{werner-sphere-cle}
Antti Kemppainen and Wendelin Werner.
\newblock The nested simple conformal loop ensembles in the {R}iemann sphere.
\newblock {\em Probab. Theory Related Fields}, 165(3-4):835--866, 2016.

\bibitem[Law18]{Law08}
Gregory~F. Lawler.
\newblock Conformally invariant loop measures.
\newblock In {\em Proceedings of the {I}nternational {C}ongress of
  {M}athematicians---{R}io de {J}aneiro 2018. {V}ol. {I}. {P}lenary lectures},
  pages 669--703. World Sci. Publ., Hackensack, NJ, 2018.

\bibitem[LG13]{LeGall13}
Jean-Fran\c{c}ois Le~Gall.
\newblock Uniqueness and universality of the {B}rownian map.
\newblock {\em Ann. Probab.}, 41(4):2880--2960, 2013.

\bibitem[Lup19]{lupu-loop-soup-cle}
Titus Lupu.
\newblock Convergence of the two-dimensional random walk loop-soup clusters to
  {CLE}.
\newblock {\em J. Eur. Math. Soc. (JEMS)}, 21(4):1201--1227, 2019.

\bibitem[LW04]{LW04}
Gregory~F. Lawler and Wendelin Werner.
\newblock The {B}rownian loop soup.
\newblock {\em Probab. Theory Related Fields}, 128(4):565--588, 2004.

\bibitem[MN04]{MN04}
Saibal Mitra and Bernard Nienhuis.
\newblock Exact conjectured expressions for correlations in the dense o(1) loop
  model on cylinders.
\newblock {\em Journal of Statistical Mechanics: Theory and Experiment},
  2004(10):P10006, oct 2004.

\bibitem[MS16]{MS16a}
Jason Miller and Scott Sheffield.
\newblock Imaginary geometry {I}: interacting {SLE}s.
\newblock {\em Probab. Theory Related Fields}, 164(3-4):553--705, 2016.

\bibitem[MS17]{ig4}
Jason Miller and Scott Sheffield.
\newblock Imaginary geometry {IV}: interior rays, whole-plane reversibility,
  and space-filling trees.
\newblock {\em Probab. Theory Related Fields}, 169(3-4):729--869, 2017.

\bibitem[MS19]{MS19}
Jason Miller and Scott Sheffield.
\newblock Liouville quantum gravity spheres as matings of finite-diameter
  trees.
\newblock {\em Ann. Inst. Henri Poincar\'{e} Probab. Stat.}, 55(3):1712--1750,
  2019.

\bibitem[MSW14]{MSWCLEgasket}
Jason Miller, Nike Sun, and David~B. Wilson.
\newblock The {H}ausdorff dimension of the {CLE} gasket.
\newblock {\em Ann. Probab.}, 42(4):1644--1665, 2014.

\bibitem[MSW17]{CLE-percolations}
Jason Miller, Scott Sheffield, and Wendelin Werner.
\newblock C{LE} percolations.
\newblock {\em Forum Math. Pi}, 5:e4, 102, 2017.

\bibitem[MSW21]{MSW-non-simple-2021}
Jason Miller, Scott Sheffield, and Wendelin Werner.
\newblock Non-simple conformal loop ensembles on {L}iouville quantum gravity
  and the law of {CLE} percolation interfaces.
\newblock {\em Probab. Theory Related Fields}, 181(1-3):669--710, 2021.

\bibitem[MSW22]{MSWsimpleCLE}
Jason Miller, Scott Sheffield, and Wendelin Werner.
\newblock Simple conformal loop ensembles on {L}iouville quantum gravity.
\newblock {\em Ann. Probab.}, 50(3):905--949, 2022.

\bibitem[MW18]{MW18}
Jason Miller and Wendelin Werner.
\newblock Connection probabilities for conformal loop ensembles.
\newblock {\em Comm. Math. Phys.}, 362(2):415--453, 2018.

\bibitem[MWW16]{MWW16}
Jason Miller, Samuel~S. Watson, and David~B. Wilson.
\newblock Extreme nesting in the conformal loop ensemble.
\newblock {\em Ann. Probab.}, 44(2):1013--1052, 2016.

\bibitem[NQSZ23]{NQSZ}
Pierre {Nolin}, Wei {Qian}, Xin {Sun}, and Zijie {Zhuang}.
\newblock {Backbone exponent for two-dimensional percolation}.
\newblock {\em arXiv e-prints}, page arXiv:2309.05050, September 2023.

\bibitem[NRJ24]{NRJ-2024}
Rongvoram Nivesvivat, Sylvain Ribault, and Jesper~Lykke Jacobsen.
\newblock {Critical loop models are exactly solvable}.
\newblock {\em SciPost Phys.}, 17:029, 2024.

\bibitem[Pes08]{Hitting-time-density}
Goran Peskir.
\newblock The law of the hitting times to points by a stable {L}\'{e}vy process
  with no negative jumps.
\newblock {\em Electron. Commun. Probab.}, 13:653--659, 2008.

\bibitem[Pol81]{polyakov1981quantum}
A.~M. Polyakov.
\newblock Quantum geometry of bosonic strings.
\newblock {\em Phys. Lett. B}, 103(3):207--210, 1981.

\bibitem[PT02]{PT02}
B.~Ponsot and J.~Teschner.
\newblock Boundary {L}iouville field theory: boundary three-point function.
\newblock {\em Nuclear Phys. B}, 622(1-2):309--327, 2002.

\bibitem[RZ22]{RZ22}
Guillaume Remy and Tunan Zhu.
\newblock Integrability of boundary {L}iouville conformal field theory.
\newblock {\em Comm. Math. Phys.}, 395(1):179--268, 2022.

\bibitem[SB89]{Saleur-Bauer-1989}
H.~Saleur and M.~Bauer.
\newblock On some relations between local height probabilities and conformal
  invariance.
\newblock {\em Nuclear Phys. B}, 320(3):591--624, 1989.

\bibitem[SD87]{SaleurDuplantier-1987}
H.~Saleur and B.~Duplantier.
\newblock Exact determination of the percolation hull exponent in two
  dimensions.
\newblock {\em Phys. Rev. Lett.}, 58(22):2325--2328, 1987.

\bibitem[She09]{SheffieldCLE}
Scott Sheffield.
\newblock Exploration trees and conformal loop ensembles.
\newblock {\em Duke Math. J.}, 147(1):79--129, 2009.

\bibitem[She16]{She16a}
Scott Sheffield.
\newblock Conformal weldings of random surfaces: {SLE} and the quantum gravity
  zipper.
\newblock {\em Ann. Probab.}, 44(5):3474--3545, 2016.

\bibitem[Shi85]{shimura-cone}
Michio Shimura.
\newblock Excursions in a cone for two-dimensional {B}rownian motion.
\newblock {\em J. Math. Kyoto Univ.}, 25(3):433--443, 1985.

\bibitem[SLN{\etalchar{+}}23]{SJNS23}
Yu-Feng {Song}, Jesper {Lykke Jacobsen}, Bernard {Nienhuis}, Andrea
  {Sportiello}, and Youjin {Deng}.
\newblock {Universality of closed nested paths in two-dimensional percolation}.
\newblock {\em arXiv e-prints}, page arXiv:2311.18700, November 2023.

\bibitem[Smi01]{Smirnov-01}
Stanislav Smirnov.
\newblock Critical percolation in the plane: conformal invariance, {C}ardy's
  formula, scaling limits.
\newblock {\em C. R. Acad. Sci. Paris S\'{e}r. I Math.}, 333(3):239--244, 2001.

\bibitem[Smi10]{Smirnov-10}
Stanislav Smirnov.
\newblock Conformal invariance in random cluster models. {I}. {H}olomorphic
  fermions in the {I}sing model.
\newblock {\em Ann. of Math. (2)}, 172(2):1435--1467, 2010.

\bibitem[SSW09]{SSW09}
Oded Schramm, Scott Sheffield, and David~B. Wilson.
\newblock Conformal radii for conformal loop ensembles.
\newblock {\em Comm. Math. Phys.}, 288(1):43--53, 2009.

\bibitem[STZ{\etalchar{+}}22]{STZJ22}
Yu-Feng Song, Xiao-Jun Tan, Xin-Hang Zhang, Jesper~Lykke Jacobsen, Bernard
  Nienhuis, and Youjin Deng.
\newblock Nested closed paths in two-dimensional percolation.
\newblock {\em J. Phys. A}, 55(20):Paper No. 204002, 11, 2022.

\bibitem[SW01]{SW01}
Stanislav Smirnov and Wendelin Werner.
\newblock Critical exponents for two-dimensional percolation.
\newblock {\em Math. Res. Lett.}, 8(5-6):729--744, 2001.

\bibitem[SW05]{SW05}
Oded Schramm and David~B. Wilson.
\newblock S{LE} coordinate changes.
\newblock {\em New York J. Math.}, 11:659--669, 2005.

\bibitem[SW12]{Sheffield-Werner-CLE}
Scott Sheffield and Wendelin Werner.
\newblock Conformal loop ensembles: the {M}arkovian characterization and the
  loop-soup construction.
\newblock {\em Ann. of Math. (2)}, 176(3):1827--1917, 2012.

\bibitem[SW16]{SW16}
Scott {Sheffield} and Menglu {Wang}.
\newblock {Field-measure correspondence in Liouville quantum gravity almost
  surely commutes with all conformal maps simultaneously}.
\newblock {\em arXiv e-prints}, page arXiv:1605.06171, May 2016.

\bibitem[SY23]{SY23}
Xin {Sun} and Pu~{Yu}.
\newblock {SLE partition functions via conformal welding of random surfaces}.
\newblock {\em arXiv e-prints}, page arXiv:2309.05177, September 2023.

\bibitem[Wu18]{Wu18}
Hao Wu.
\newblock Polychromatic arm exponents for the critical planar {FK}-{I}sing
  model.
\newblock {\em J. Stat. Phys.}, 170(6):1177--1196, 2018.

\bibitem[WW13]{werner-wu-explorations}
Wendelin Werner and Hao Wu.
\newblock On conformally invariant {CLE} explorations.
\newblock {\em Comm. Math. Phys.}, 320(3):637--661, 2013.

\bibitem[ZZ96]{ZZ96}
Alexander~B. Zamolodchikov and Alexei~B. Zamolodchikov.
\newblock {Structure constants and conformal bootstrap in Liouville field
  theory}.
\newblock {\em Nucl. Phys. B}, 477:577--605, 1996.

\end{thebibliography}

\end{document}